\documentclass{proc-l}
\usepackage{hyperref}
\usepackage{xcolor}
\usepackage{amsmath}
\usepackage[T1]{fontenc}
\usepackage{listings}
\usepackage{amssymb}
\usepackage{enumitem}

\RequirePackage{tikz-cd}

\newcommand{\R}{\mathbf{R}}
\newcommand{\C}{\mathbf{C}}
\newcommand{\Z}{\mathbf{Z}}
\newcommand{\K}{\mathbf{K}}

\newcommand{\Proj}{\mathbf{P}}
\newcommand{\N}{\mathbf{N}}
\newcommand{\Q}{\mathbf{Q}}

\newcommand{\F}{\mathbf{F}}
\newcommand{\I}{\mathcal{I}}
\newcommand{\Le}{\mathbf{L}}
\newcommand{\G}{\mathbf{G}}
\newcommand{\T}{\mathbf{T}}

\newcommand{\Sym}{\mathfrak{S}}
\newcommand{\base}{\mathbf{b}}

\newcommand{\coneorb}{\text{C}_{\text{eff},\text{orb}}(X)}

\DeclareMathOperator{\Conj}{Conj}
\DeclareMathOperator{\Spec}{Spec}

\DeclareMathOperator{\Gal}{Gal}
\DeclareMathOperator{\Pic}{Pic}

\DeclareMathOperator{\mor}{Mor}
\DeclareMathOperator{\Aut}{Aut}

\DeclareMathOperator{\can}{can}
\DeclareMathOperator{\id}{Id}

\DeclareMathOperator{\loc}{loc}

\DeclareMathOperator{\Div}{Div}
\DeclareMathOperator{\proj}{proj}
\DeclareMathOperator{\grp}{grp}
\DeclareMathOperator{\Hom}{Hom}
\DeclareMathOperator{\rank}{rank}

\DeclareMathOperator{\Ext}{Ext}

\DeclareMathOperator{\Vol}{Vol}

\DeclareMathOperator{\stack}{stack}
\DeclareMathOperator{\covol}{covol}

\DeclareMathOperator{\coneform}{cone}

\DeclareMathOperator{\num}{num}

\DeclareMathOperator{\car}{char}
\DeclareMathOperator{\age}{age}
\DeclareMathOperator{\NS}{NS}
\DeclareMathOperator{\coarse}{coarse}
\DeclareMathOperator{\rep}{rep}

\DeclareMathOperator{\orb}{orb}
\DeclareMathOperator{\un}{nr}
\DeclareMathOperator{\stab}{stab}
\DeclareMathOperator{\coker}{coker}
\DeclareMathOperator{\tor}{tor}
\DeclareMathOperator{\Boxe}{Box}
\DeclareMathOperator{\Cl}{Cl}
\DeclareMathOperator{\rig}{rig}
\DeclareMathOperator{\ext}{ext}
\DeclareMathOperator{\Sal}{Sal}

\newcommand{\stackX}{\mathcal{X}}
\newcommand{\stackY}{\mathcal{Y}}
\newcommand{\stackZ}{\mathcal{Z}}
\newcommand{\stackS}{\mathcal{S}}
\newcommand{\stackT}{\mathcal{T}}
\newcommand{\stackG}{\mathcal{G}}

\newcommand{\stackP}{\mathcal{P}}
\newcommand{\affine}{\mathbf{A}}
\newcommand{\sheafL}{\mathcal{L}}

\newcommand{\sheafE}{\mathcal{E}}
\newcommand{\sheafI}{\mathcal{I}}
\newcommand{\sheafO}{\mathcal{O}}

\newcommand{\w}{\omega}

\newcommand{\RHom}{\mathbf{R Hom}}
\newcommand{\sector}{\pi_0(\sheafI_{\mu} X)}
\newcommand{\twistsector}{\pi_0^*(\sheafI_{\mu} X)}
\newcommand{\adelicprod}{\mathop{\underset{v \in M_{\Q}^0}{\prod\nolimits'}}}

\newtheorem{theorem}{Theorem}[section]
\newtheorem{cor}[theorem]{Corollary}

\newtheorem{lemma}[theorem]{Lemma}
\newtheorem{prop}[theorem]{Proposition}

\theoremstyle{definition}
\newtheorem{definition}[theorem]{Definition}
\newtheorem{hypothesis}[theorem]{Hypothesis}
\newtheorem{example}[theorem]{Example}

\theoremstyle{remark}
\newtheorem{remark}[theorem]{Remark}

\newtheorem{Notations}[theorem]{Notations}

\numberwithin{equation}{section}

\begin{document}

\title[Multi-height analysis on toric stacks]{Multi-height distribution of rational points of split toric stacks}


\author{Nicolas Bongiorno}
\address{}
\curraddr{}
\email{nicolas.bongiorno@univ-grenoble-alpes.fr}
\thanks{}

\subjclass[2020]{}

\date{}

\dedicatory{}


\begin{abstract}
We study the distribution of rational points of split toric stacks with all heights bounded over $\Q$ by lifting the counting problem to an extended universal torsor under the torus associated with the orbifold Picard group. To achieve this, we prove the existence of an integral parametrization of rational points on toric stacks, which allows us to define a lift of the stacky height to this extended universal torsor. This allows us to define the Tamagawa number of a toric stack $X$ as an Euler product and, for a prime number $p$, to interpret the $p$-adic factor via a mass formula counting $\F_p$-points of the sectors of $X$.
\end{abstract}


\maketitle

\setcounter{tocdepth}{1} 
\tableofcontents

\section{Introduction}

For a quasi-Fano variety $V$ with infinitely many rational points, one may study asymptotically the finite set of rational points of bounded height $H$ associated with the anticanonical line bundle $\w_{V}^{-1}$. In \cite{FrankeManinTschinkel1989}, \cite{Batyrev1990} and \cite{peyre_duke}, Batyrev, Franke, Manin, Tschinkel and Peyre provided strong evidence supporting conjectures relating the asymptotic behaviour of the number of rational points of bounded height on open subsets of $V$ to geometric invariants of $V$.

In \cite[Question 4.8]{Peyre_beyond_height}, Peyre proposed a new framework to study the asymptotic distribution of rational points of bounded height. Instead of considering a height relative to a single line bundle, it is natural to consider all possible heights. In \cite[Theorem 1.14]{bongiorno2024multiheightanalysisrationalpoints}, the author showed that the expected asymptotic behaviour holds for smooth, projective and split toric varieties over $\Q$, using a lift of the count to universal torsors. This approach is inspired by the work of Salberger in \cite{salberger_torsor} and employs the method of counting lattice points in a bounded domain, developed in \cite{davenport_geometry_numbers}.

The purpose of this work is to extend the multi-height formalism to Deligne–\allowbreak Mumford (DM) stacks and show that, for toric stacks, one can study the multi-height distribution of rational points and interpret the results using geometric invariants of the stacks.

When studying rational points on stacks, the main issue is that even for a proper DM stack, a rational point does not necessarily extend to an integral point, as pointed out in \cite{dardathesis} and \cite{ellenberg2022heights}. In the latter article, Ellenberg, Satriano, and Zureick-Brown showed that instead of extending a rational point to an integral point, it is possible to extend it to a stacky integral lift (see \cite[Proposition 2.5]{ellenberg2022heights}). Later, Loughran and Santens established an analogous result in the local and tame cases, offering a new perspective (see \cite[Paragraph 4.4]{loughran_santens_malle_conjecture}).

In \cite{darda2024batyrevmanin}, Darda and Yasuda introduced a new framework for working with heights over DM stacks. The main innovation in their paper is the introduction of an $\R$-vector space that incorporates both the usual Picard group $\Pic(X)_{\R}$ and contributions from the twisted sectors. These sectors capture the obstruction, at a given place, to lifting a rational point to an integral one.

This formalism allows Darda and Yasuda, and later Loughran and Santens, to reinterpret the work of number theorists on Malle's and Bhargava's conjectures (see \cite{Malle2002,Malle2004} and \cite{bhargava_mass_formula}) through the lens of Manin's program for $B G$. Many results have been obtained in this direction by the aforementioned authors.

The work of Gundlach (see \cite{gundlach2022mallesconjecturemultipleinvariants}) and, respectively, the work of Shankar and Thorne (see \cite{shankar2024asymptoticscubicfieldsordered}) can be reinterpreted as instances of a multi-height conjecture over $B G$ for $G$ an abstract finite abelian group, and over $B \mathfrak{S}_3$, respectively. In the case of toric stacks, the asymptotic behaviour of the number of rational points bounded by the height associated with the orbifold anticanonical bundle was obtained by Darda and Yasuda in \cite[Theorem 1.3.1]{darda_yasuda_toric_stacks_batyrev}, using the Poisson formula on the height zeta function. The non-simplicial nature of the cone of convergence of the height zeta function complicates the use of analytic number theory in the multi-height setting.

To overcome this difficulty, we develop an extended integral parametrization of the rational points of toric stacks (see Theorem~\ref{theorem_description_of_the_extended_parmetrization}), in the spirit of the parametrization provided by universal torsors in the case of toric varieties. In this setting, the role of the Néron--Severi torus $T_{\NS}$ is played by the torus associated with the orbifold Picard group $\Pic_{\orb}(X)$. Using this parametrization, we prove the following result on the asymptotic behaviour of multi-heights of rational points on toric stacks.

\begin{theorem}
Let $X$ be a toric stack over $\Q$. Let $\mathrm{D}_1$ be a finite union of compact polyhedra in $\Pic_{\orb}(X)^{\vee}_{\R}$, and let $u$ be an element of the interior of the dual of the orbifold effective cone $(\coneorb^{\vee})^{\circ}$. For $B > 1$, set
\[
\mathrm{D}_B = \mathrm{D}_1 + \log(B)\, u.
\]
Then the multi-height asymptotic behaviour is given by
\[
\sharp\bigl( X(\Q) \bigr)_{h \in \mathrm{D}_B}
\;\underset{B \to +\infty}{\sim}\;
\nu(\mathrm{D}_1)\, \tau_{\orb}(X)\,
B^{\langle \w_{X,\orb}^{-1}, u \rangle},
\]
where $\tau_{\orb}(X)$ is defined in Definition~\ref{definition_orbifold_tamagawa_number}.
\end{theorem}

In particular, the method developed in this paper allows one to compute explicitly the leading constant in the multi-height asymptotic formula for rational points. It admits an Euler product expression of the form

\[
\prod_{p}
\left( 1 - \frac{1}{p} \right)^{\rank \Pic_{\orb}(X)}
\cdot
\mu_{X,p}(X(\Q_p)),
\]where $\mu_{X,p}$ is a Tamagawa measure on $X(\Q_p)$ (see \ref{definition_quotient_measure_over_orbifold}).

One of the achievements of this paper is the interpretation of this $p$-adic factor via a mass formula counting the $\F_p$-points of the sectors of $X$ (Theorem~\ref{theorem_interpretation_mass_formula}), in the spirit of the mass formula developed by Loughran and Santens for $BG$ (see \cite[Corollary~8.11]{loughran_santens_malle_conjecture}).

\begin{theorem}
Let $p$ be a prime number such that $X_{\Z_p}$ is a tame DM stack. Then the following equality holds:
\[
\mu_{X,p}\bigl(X(\Q_p)\bigr)
= \frac{\sharp X(\F_p)}{p^{\dim(X)}}  + \frac{1}{p} \cdot
\sum\limits_{\stackS \in \twistsector} 
\frac{\sharp\, \stackS\!\left(\F_p\right)}{p^{\dim(\stackS)}}
\]where $\sharp$ means the groupoid cardinality (see equation \ref{equation_groupoid_cardinality})
\end{theorem}

In a forthcoming work \cite{bongiorno_hyperbola_method}, the author explains how the multi-height result implies the classical Batyrev--Manin principle, where one counts rational points of bounded height with respect to the (orbifold) anticanonical bundle. In particular, this makes it possible to interpret the leading constant obtained by Darda and Yasuda in \cite[Corollary 4.4.5]{darda_yasuda_toric_stacks_batyrev} as a residue, to the constant introduced in the present paper. 

A key ingredient highlighted in this work is that, for toric stacks, a twisted sector is completely determined by its age pairing with $\Pic(X)$, and moreover that the age can be lifted to torsors (see Corollary~\ref{corollary_sectors_dual_picard_group_toric_stacks} and Theorem~\ref{theorem_age_extended_universal_torsor}). To this end, we relate the computation of the age to the type of a torsor under a multiplicative group (see Proposition \ref{prop_computation_general_age_pairing_via_type_torseur}) and compute the type of the $T_{\NS}$-torsor
$$
\stackT_{\Sigma} \rightarrow X
$$
in Theorem~\ref{theorem_cox_ring_universal_torsor}. Beyond the study of the multi-height distribution of rational points on toric stacks, we also extend the definition of the residue map $\psi_v$ introduced in \cite[Definition 2.17]{darda2024batyrevmanin}, under additional hypotheses on the universal torsor of $X$ (see Corollary~\ref{corollary_description_universal_stacky_lift}), even if $X$ is wild for $v$. This work also offers a new perspective for treating non-split toric stacks, namely
those endowed with a nontrivial Galois action. The study of rational points on such
stacks includes two particularly interesting cases:
\begin{itemize}[label=--]
\item modular curves $\mathcal{X}_0(N)$ for $N \neq 13$ such that the coarse moduli space of $\mathcal{X}_0(N)$ is $\Proj^1$ (this is consequence of the recent work of Darda and Han, see \cite[Theorem 1.5]{darda2026stackybatyrevmaninconjecturemodular}), whose rational points
parametrize pairs $(E,\varphi)$, where $E$ is an elliptic curve and
$\varphi$ is an isogeny of order $N$;
\item Galois extensions with Galois group $A$, where $A$ is a finite abelian group,
which can be interpreted as rational points of the classifying stack $BA$. I thank Santens for explaining to me how to construct the extended universal torsor of $BA$ which in particular sheds light on the construction of the extended universal torsor of non-split toric stacks.
\end{itemize}

\subsection{Terminology}

We begin by recalling the hypotheses of \cite{darda2024batyrevmanin} on the Deligne-Mumford stacks we shall work with:

\begin{definition}
    A \textit{nice stack} is a Deligne-Mumford stack $X$ over a field $\K$ satisfying the following conditions:
    \begin{enumerate}
        \item $X$ is separated, geometrically irreducible and smooth over $\K$,
        \item the coarse moduli space of $X$, denoted $X^{\coarse}$, is a projective $\K$-scheme,
        \item $X$ is not isomorphic to $\Spec(\K)$.
    \end{enumerate}
    If $\K$ is a number field, we shall also assume that we have a proper, normal model $\stackX$ of $X$ over $\sheafO_S$ the ring of $S$-integers of $\K$ with $S$ a finite set of finite places of $\K$, which is an algebraic stack, such that the diagonal $$ \stackX \rightarrow \stackX \times \stackX$$ is finite and $\stackX$ is almost a tame stack. By almost a tame stack the author means that there exists $S'$ a finite set of $M^0_{\K}$ which contains $S$ such that $\stackX_{\sheafO_S'}$ is a DM stack and for any geometric point $P$ of $\stackX_{\sheafO_{S'}}(k)$, with $k$ algebraically closed, $\car(k)$ does not divide the cardinality of $\Aut_k(P)$.
\end{definition}

The classifying stack $B G$ where $G$ is a finite étale group scheme and proper toric stacks are examples of nice stacks, with a natural model defined over $\sheafO_K$.

 \begin{Notations}\label{notation_exponent_algebraic_stack}
    If $X$ is a Deligne-Mumford stack and $S$ is a scheme, we shall denote by $X[S]$ the groupoid of its $S$-points and by $X(S)$ the set obtained when we quotient by the isomorphism relation.

    We denote by $N_X$ the exponent of $X$, that is to say the least common multiple of the cardinal of the stabilizers of each geometric point of $X$. As the diagonal $X \rightarrow X \times X$ is finite, such integral exponent exists.
\end{Notations}

\subsection{Outline of the article}

In the second section of this article, we recall the machinery introduced in \cite{ellenberg2022heights}, \cite{darda2024batyrevmanin} and \cite{loughran_santens_malle_conjecture} to study the rational points of nice DM stacks. In the third section, we present the age pairing and the integral orbifold Picard group and we recall the definition of heights over nice DM stacks. The fourth section aims at showing how one can lift the age pairing to torsors under split torus. We also construct a residue map in the wild case provided some additional hypotheses. In the fifth section, we present the definition and the main properties of toric stacks. We also establish some useful new properties related to the age pairing and the orbifold Picard group of toric stacks. In the sixth section, we present an integral parametrization of the rational points of toric stacks. Sections 7 and 8 are devoted to the construction of a lift of the stacky height to the extended universal torsor and the study of its main properties. In Section 9, we define a Tamagawa measure on toric stacks and interpret it at places with tame inertia using a mass formula. The last section aims at studying the multi-height asymptotic behaviour of the rational points of toric stacks.

\subsection{Acknowledgements}

I am grateful to my thesis advisor, Emmanuel Peyre, for proposing this stimulating research topic and for his reading of the various drafts of this article. I also wish to thank Régis De La Bretèche, Loïs Faisant, Daniel Loughran, Sylvain Brochard, Ratko Darda and Tim Santens for insightful discussions that contributed to the development of this work. I would like to thank Hiroshi Iritani for communicating to me the complete proof of \cite[Proposition 21]{Coates_Corti_Iritani_Tseng_miror_theorem_toric_stack}.

\section{Stacky integral points and the residue map}

If $X$ is a proper model of a nice variety over $\sheafO_S$, the map $ X(\sheafO_S) \rightarrow X(\K) $ is a bijection. This is no longer necessarily the case for proper stacks: the map is not always surjective. In this section, we explain how to address the problem of extending rational points to integral points. In the global setting, we recall the definition and main properties of the notion of a ``tuning stack'' as introduced in \cite{ellenberg2022heights} which we shall refer to as a \textit{stacky integral lift}. In the local setting, we make use of the stacky Hensel’s lemma provided in \cite{loughran_santens_malle_conjecture}. In the third part of this section, we introduce the notion of twisted sectors. These were first applied in an arithmetic context by Darda and Yasuda in \cite{darda2024batyrevmanin}, and, via the residue map, have proven to be particularly useful for understanding the obstruction to lifting a rational point to an integral one.

\subsection{Global extension of a rational point of a nice stack}

For this paragraph, we consider a Dedekind ring $\sheafO$ with fraction field $\K$ and $X$ a proper normal algebraic stack with finite diagonal which is the model over $\sheafO$ of a nice DM stack $X_{\K}$. We follow \cite[section 2.1]{ellenberg2022heights}.

\begin{definition}
For a rational point $P \in X(\K)$, any triplet $(\mathcal{C},\bar{P},p)$ such that $\mathcal{C}$ is a normal separated Deligne-Mumford stack, $p : \mathcal{C} \rightarrow \Spec(\sheafO)$ a birational coarse space map and the following diagram is 2-commutative:
\begin{center}
\begin{tikzcd}
\Spec(K) \arrow[r] \arrow[rd] \arrow[rr, "P", bend left=49, shift left] & \mathcal{C} \arrow[r, "\bar{P}"] \arrow[d, "p"] & X \arrow[ld] \\
                                                                        & \Spec(\sheafO)                                                 &                     
\end{tikzcd}
\end{center}is called a \emph{stacky integral point} which extends $P$ or a \emph{stacky integral lift} of $P$.
\end{definition}

One can naturally define a notion of morphism between stacky integral lifts of $P$. If $(\mathcal{C},\bar{P},p)$ is terminal among the stacky integral points which extend $P$, we shall call it the universal stacky lift of $P$. In \cite[section 2.1]{ellenberg2022heights}, the author gives an existence theorem for this universal stacky lift and a useful characterization. Let us recall it here. We begin with the following proposition which is well-known for nice schemes and proved in \cite[proposition B.1]{Rydh_2015} for nice stacks:

\begin{prop}
   With our hypotheses, given a rational point $P \in X(\K)$, there exists a finite set $S$ of maximal ideals of $\sheafO$, such that $P$ can be extended to an integral point $\bar{P} \in X(\sheafO_{S})$ where $\sheafO_S$ is the Dedekind ring associated with $\Spec(\sheafO) - S$.
\end{prop}

By \cite[proposition 2.5]{ellenberg2022heights}, one can construct the universal stacky lift of $P$ as follows:

\begin{prop}\label{prop_normalization_method_stacky_lift}
    Let $P \in X(\K)$. If $U \rightarrow X$ is any extension of $P$ for $U$ a non empty open subset of $\Spec(\sheafO)$, then the relative normalisation of $X$ in $U$, $$\bar{P} : \mathcal{C} \rightarrow X$$ is a universal lift of $P$, and in particular it is independent of the choice of the lift $U \rightarrow X$ of $P$.
\end{prop}

Moreover we have the following characterization of the universal stacky lift associated to $P$ (\cite[corollary 2.6]{ellenberg2022heights}):

\begin{prop}\label{proposition_carac_universal_stacky_lift_representability}
    Let $(\mathcal{C},\bar{P},p)$ be a stacky integral point which extends $P$. Then it is the universal lift of $P$ if and only if $\bar{P} : \mathcal{C} \rightarrow X$ is representable.
\end{prop}

\begin{remark}
    The construction of the universal stacky lift of a rational point is defined regardless of whether $X$ is tame or not. Moreover if $X \rightarrow \Spec(\sheafO)$ is tame, we have a description of $\mathcal{C}$ as a root stack that we shall only use here in the local case.
\end{remark}

\subsection{The stacky Hensel's lemma and its interpretations}

Now we focus on the case where $\sheafO$ is a henselian discrete valuation ring (DVR) with valuation $v$. We write $\mathfrak{p}$ for the maximal ideal of $\sheafO$ and $\mathbf{F}$ for its residue field, which is supposed to be finite. We suppose that $X$ is a normal and proper algebraic stack with finite diagonal over $\sheafO$.

\begin{Notations}
    Let $n \in \N^*$. We follow the notations from \cite[Section 4]{loughran_santens_malle_conjecture}. We denote by $\Spec(\sheafO[\sqrt[n]{v}])$ the root stack $\sqrt[n]{\mathfrak{p} / \Spec(\sheafO)}$. We shall use the following abbreviations: $X(\sheafO[\sqrt[n]{v}]) = X(\Spec(\sheafO[\sqrt[n]{v}]))$ and for any abelian étale sheaf $\mathcal{F}$, $$H^i(\sheafO[\sqrt[n]{v}],\mathcal{F}) = H^i(\Spec(\sheafO[\sqrt[n]{v}]),\mathcal{F})$$ for $i \in \N$.
\end{Notations}

In the remainder of this paragraph, we assume that for every $P \in X(\K)$ there exists $n \in \N^*$ such that the universal stacky lift of $P$ is of the form
$$
\Spec(\sheafO[\sqrt[n]{v}]) \xrightarrow{\overline P} X .
$$
When $X$ is a tame DM stack, this hypothesis is satisfied by the following proposition.

\begin{prop}\label{proposition_description_tame_universal_stacky_lift}
    Assume that $X$ is a tame DM stack and let $P \in X(\K)$ be a rational point. There exists a unique $n \in \N^*$, and up to $2$-isomorphism, a unique representable morphism
    $$
    \Spec(\sheafO[\sqrt[n]{v}]) \rightarrow X
    $$
    lifting $P$.
\end{prop}

\begin{proof}
Since tame stacky curves can be described as root stacks, the result follows either from the previous discussion or from \cite[Lemma 4.2]{loughran_santens_malle_conjecture}.
\end{proof}

We shall see later that in the case of toric stacks this result remains valid even at places with wild ramification (see Corollary \ref{corollary_description_universal_stacky_lift}). If $\pi \in \sheafO$ is a uniformizer, we have a quotient description of the root stack $$\Spec(\sheafO[\sqrt[n]{v}]) = \left[\Spec\left(\sheafO[t]/(t^n - \pi)\right) / \mu_n \right] \, .$$This gives in some particular cases, an important interpretation of the root stack over the valuation ring of a p-adic field:

\begin{prop}\label{proposition_description_root_stack_over_the_unrammified_closure_in_the_tame_case}
    Suppose that $\K$ is a $p$-adic field, we write $\K' = \K[\pi^{\frac{1}{n}}]$ for the totally ramified extension of $\K$ of degree $n$, then $\sheafO_{\K'} = \sheafO[\pi^{\frac{1}{n}}] $ and we have:
    $$\Spec(\sheafO[\sqrt[n]{v}]) = \left[\Spec\left(\sheafO_{\K'} \right) / \mu_{n,\sheafO} \right] .$$ 
\end{prop}

Thus, the choice of a uniformizer $\pi \in \sheafO$ induces a natural closed immersion
$$
B \mu_{n,\F} \xrightarrow{i_{\pi}} \Spec(\sheafO[\sqrt[n]{v}]).
$$
The dependence of this closed immersion $i_{\pi}$ on the choice of the uniformizer $\pi$ is described in \cite[Lemma 4.3]{loughran_santens_malle_conjecture}. Once this uniformizer is fixed, we shall use the following isomorphism, which follows from \cite[Proposition 5.2.11]{Brochard2008}.

\begin{prop}\label{proposition_cohomology_local_root_stack}
    For any split torus $T$ over $\sheafO$, we have:
    $$H^1(\sheafO[\sqrt[n]{v}],T) = H^1(B \mu_{n,\F} , T_{\F }) = H^1(\mu_{n,\F},T_{\F}) .$$
\end{prop}

\subsection{Twisted sectors and the residue map}

\subsubsection{The cyclotomic inertia stack}

In \cite{darda2024batyrevmanin}, Darda and Yasuda proposed to control the ramification (or the stackyness) of the universal stacky lift associated with a rational point using the notion of twisted sectors. 

\begin{Notations}
    For this paragraph, $\K$ is a number field. We denote by $M_{\K}$ the set of places of $\K$, by $M_{\K}^0$ the set of finite places, and by $M_{\K}^{\infty}$ the set of infinite places. For $S$ a finite subset of $M_{\K}^0$, we write $\sheafO_S$ for the ring of $S$-integers. We shall suppose that $X$ is a proper and  normal algebraic stack with finite diagonal over $\sheafO_S$, which is the model of a nice stack $X_{\K}$ over $\K$. We will sometimes write $X_{\sheafO_S}$ to emphasize the base scheme.

    For $v \in M_{\K}^0$, we denote by $\sheafO_v$ the valuation ring associated with $v$, $\F_v$ its residue field, and ${\K}_v$ its fraction field. 
\end{Notations}

Darda and Yasuda defined a residue map which characterizes the obstruction to lifting a point $P \in X(\K)$ as an integral point. This allows them to define families of relevant heights over nice stacks and to formulate precise conjectures on the distribution of rational points on $X$. Let us begin by defining the cyclotomic inertia stack $\I_{\mu} X$, which is called \textit{the stack of 0-twisted jets} of $X$ in \cite[definition 2.3]{darda2024batyrevmanin}.

\begin{definition}
    $$ \sheafI_{\mu} X = \bigcup\limits_{l > 0} \underline{\Hom}^{\rep}(B \mu_l,X) $$
\end{definition}

The definition can be applied over any base scheme and in particular over $\sheafO_S$.

\begin{definition}
    The connected components of $\sheafI_{\mu} X$, $\sector$, are called the sectors of $X$. This is a pointed set, with one distinguished sector given by $[X]$, which shall be written $0 \in \sector$. All the others are called the twisted sectors. The set of twisted sectors is denoted by $ \twistsector$.
\end{definition}

\begin{remark}
    With our assumptions, $\sheafI_{\mu} X$ is smooth (see \cite[lemma 2.7]{darda2024batyrevmanin}). Hence $\sector$ is the set of generic representable morphism $b : B \mu_r \rightarrow X$. We shall denote by $r_b$ the exponent associated with an element $b \in \sector$.
\end{remark}

\begin{example}
    \begin{itemize}
        \item The set of sectors of $B G$ over $\K$ is in bijection with the set of $\K$-conjugacy classes of $G$ (see \cite[example 2.15]{darda2024batyrevmanin}). For example, if $G = \Sym_n$, this corresponds to the conjugacy classes of $\Sym_n$. 
        \item The set of sectors of the weighted projective space $\stackP(w)$ is in  bijection with $\bigcup\limits_{0 \leqslant i \leqslant n} \frac{1}{w_i} \Z \cap [0,1[ $ (see \cite[proposition 3.9]{mann2006orbifoldquantumcohomologyweighted}). 
    \end{itemize}
\end{example}

By \cite[proposition 2.11]{darda2024batyrevmanin}, we have the following useful characterization:

\begin{prop}
    $$\sheafI_{\mu} X_{\K} = \underline{\Hom}(B \widehat{\mu},X_{\K}) \text{ and } \sheafI_{\mu} X_{\sheafO_S} = \underline{\Hom}(B \widehat{\mu}_{\sheafO_S},X_{\sheafO_S}) .$$Moreover, if $N \in \N^*$ is an integer such that the cardinal of any geometric stabilizer of $X_{\sheafO_S}$ divides $N$, then:
    $$\sheafI_{\mu} X_{\K} = \underline{\Hom}(B \mu_N,X_{\K}) \text{ and } \sheafI_{\mu} X_{\sheafO_S} = \underline{\Hom}(B \mu_{N,\sheafO_S},X_{\sheafO_S}) .$$
\end{prop}

We shall suppose for the rest of this article that $\sheafI_{\mu} X_{\sheafO_S}$ is flat. This is true at least for toric stacks. By \cite[corollary 2.8]{darda2024batyrevmanin}, we obtain:
\begin{prop}\label{prop_sector_over_ring_of_integer}
    The following maps are bijections:$$\pi_0(\sheafI_{\mu} X_{\K}) \rightarrow \pi_0(\sheafI_{\mu} X_{\sheafO_S}) $$and for any $v \not\in S$, $$\pi_0(\sheafI_{\mu} X_{\K_v}) \rightarrow \pi_0(\sheafI_{\mu} X_{\sheafO_v}) .$$
\end{prop}

\subsubsection{Relation with the inertia stack}

Suppose we work in a geometric setting over an algebraically closed field $\F = \overline{\F}$ or that the action of the absolute Galois group of the base field $\F$ is trivial on $\pi_0(\sheafI_{\mu} X)$, i.e. we have a bijection $\pi_0(\sheafI_{\mu} X) \rightarrow \pi_0(\sheafI_{\mu} \overline{X})$. Over an algebraically closed field, there is a non canonical isomorphism between the cyclotomic inertia stack $\sheafI_{\mu} \overline{X}$ with the inertia stack $I \overline{X}$ of $\overline{X}$ (see \cite[proposition 22]{yasuda2004motivicintegrationdelignemumfordstacks}). Suppose $\overline{X} = [Z/G]$ and $\Lambda$ is a set of representatives of $\Conj_{\text{fin}}(G)$ the set of conjugacy classes of the elements of finite order of $G$. Recall that we have $I \overline{X} \simeq [ I_G Z / G ]$ where:
$$ I_G Z = \{ (g,z) \mid g.z = z \}$$ and $G$ acts on $I_G Z$ via:
$$ (h, (g,z)) \longmapsto (h.g.h^{-1},h.z) .$$Hence:
$$I \overline{X} \simeq \bigsqcup\limits_{g \in \Lambda} [Z^g / C_G(g) ]$$where $Z^g$ is the set of points fixed by $g$ and $C_G(g)$ is the centralizer of $g$.

\begin{prop}
    If for any $g \in \Lambda$, $Z^g$ is connected, then the set of sectors of $\overline{X}$ is given by $$\left\{ [Z^g / C_G(g) ] \mid g \in \Lambda \right\} .$$
\end{prop}

\subsubsection{The residue map}\label{subsubsection_the_residue_map_definition}

We again suppose that $\K$ is a number field. Let $v \in M_{\K}^0$ be a finite place such that $X_{\sheafO_v}$ is a normal, proper algebraic stack with finite diagonal over $\sheafO_v$ satisfying the additional property that for every $P \in X(\K_v)$, the universal stacky lift of $P$ is of the form
$$
\Spec(\sheafO[\sqrt[n]{v}]) \xrightarrow{\overline P} X
$$
for some $n \in \N^*$. We have seen that this hypothesis holds if $X_{\sheafO_v}$ is a tame DM stack (see Proposition \ref{proposition_description_tame_universal_stacky_lift}) or if $X_{\sheafO_v}$ is a toric stack (see Corollary \ref{corollary_description_universal_stacky_lift}).

In \cite[Paragraph 4.4]{loughran_santens_malle_conjecture}, Loughran and Santens showed that the correct way to define reduction modulo $\mathfrak{p}_v$ is to consider a map with values in $\sheafI_{\mu} X (\F_v)$. Following their method and fixing a uniformizer $\pi \in \sheafO_v$, we define a natural reduction map
$$
- \text{mod} \ \mathfrak{p}_v : X(\K_v) \rightarrow \sheafI_{\mu} X(\F_v)
$$
which associates to the unique (up to $2$-isomorphism) universal stacky lift
$$
\Spec(\sheafO_v[\sqrt[n]{v}]) \xrightarrow{\overline P} X
$$
of $P$ the representable morphism given by the composition
$$
B \mu_{n,\F_v} \xrightarrow{i_\pi} \Spec(\sheafO_v[\sqrt[n]{v}]) \xrightarrow{\overline P} X .
$$The following lemma explains that this reduction modulo $\mathfrak{p}_v$ map does not require the choice of the universal stacky lift: any stacky lift of $P$ of the form
$$
\Spec(\sheafO_v[\sqrt[m]{v}]) \xrightarrow{\overline P'} X
$$
is sufficient.

\begin{lemma}\label{lemma_reduction_map_choice_stacky_lift}
Let
$$
\Spec(\sheafO_v[\sqrt[m]{v}]) \xrightarrow{\overline P'} X
$$
be a stacky lift of $P$. Then the induced morphism
$$
B \mu_{n,\F_v}
\xrightarrow{i_{\pi}}
\Spec(\sheafO_v[\sqrt[n]{v}])
\xrightarrow{\overline P'}
X
$$
defines an element of $\I_{\mu}X[\F_v]$ which is isomorphic to $P \bmod \mathfrak{p}_v$.
\end{lemma}

\begin{proof}
Let
$$
\Spec(\sheafO_v[\sqrt[n]{v}]) \xrightarrow{\overline P} X
$$
be a universal stacky lift of $P \in X(\K_v)$. By the definition of a universal stacky lift, the morphism $\overline P'$ factors through $\overline P$, so that we obtain the following $2$-commutative diagram:
\begin{center}
\begin{tikzcd}
{\Spec(\sheafO_v[\sqrt[m]{v}])} \arrow[d] \arrow[rd, "\overline P'"] &   \\
{\Spec(\sheafO_v[\sqrt[n]{v}])} \arrow[r, "\overline P"]             & X
\end{tikzcd}
\end{center}

By \cite[Lemma~3.6]{2023_Bresciani}, we have $n \mid m$, and the morphism
$$
\Spec(\sheafO_v[\sqrt[m]{v}]) \longrightarrow \Spec(\sheafO_v[\sqrt[n]{v}])
$$
is $2$-isomorphic to the natural map. Hence we obtain the following $2$-commutative diagram:
\begin{center}
\begin{tikzcd}
{B \mu_{m,\F_v}} \arrow[r, "i_{\pi}"] \arrow[d] 
& {\Spec(\sheafO_v[\sqrt[m]{v}])} \arrow[d] \arrow[rd, "\overline P'"] &   \\
{B \mu_{n,\F_v}} \arrow[r, "i_{\pi}"]           
& {\Spec(\sheafO_v[\sqrt[n]{v}])} \arrow[r, "\overline P"]             & X
\end{tikzcd}
\end{center}
where the morphism
$$
B\mu_{m,\F_v} \longrightarrow B\mu_{n,\F_v}
$$
is given by
$$
t \in \mu_m \longmapsto t^{\frac{m}{n}} \in \mu_n .
$$
This concludes the proof.
\end{proof}

The following proposition shows that this reduction modulo $\mathfrak{p}_v$ map is invariant under unramified extensions of $\K_v$.

\begin{prop}\label{prop_invariance_unramified_extension_reduction_map}
   The isomorphism class of the geometric point defined by $$P \mod \mathfrak{p}_v$$ depends only on $P_{| \K_v^{\un}}$.
\end{prop}

\begin{proof}
    Indeed we have the following commutative diagram:
    \begin{center}
    \begin{tikzcd}
\Spec(\K_v^{\un}) \arrow[r] \arrow[d] & {\Spec(\sheafO_v[\sqrt[n]{v}]) \times_{\sheafO_v} \sheafO_v^{\un}} \arrow[d] &   \\
\Spec(\K_v) \arrow[r]                 & {\Spec(\sheafO_v[\sqrt[n]{v}])} \arrow[r]                                    & X
\end{tikzcd}
    \end{center}Moreover, since the right square and the whole square of the following commutative diagram:
    \begin{center}
    \begin{tikzcd}
{\Spec(\sheafO_v^{\un}[\sqrt[n]{v}])} \arrow[d] \arrow[r]                             & {\Spec(\sheafO_v[\sqrt[n]{v}])} \arrow[r] \arrow[d] & {[\affine^1 / \G_m]} \arrow[d, "\wedge n"] \\
\Spec(\sheafO_v^{\un}) \arrow[r] \arrow[rr, "{(\sheafO_v^{\un},\pi)}"', bend right] & \Spec(\sheafO_v) \arrow[r, "{(\sheafO_v,\pi)}"]   & {[\affine^1 / \G_m]}               
\end{tikzcd}
    \end{center}are cartesian, the left square is also cartesian. Hence we get:
    $$\Spec(\sheafO_v[\sqrt[n]{v}]) \times_{\sheafO_v} \sheafO_v^{\un} = \Spec(\sheafO_v^{\un}[\sqrt[n]{v}]) .$$ We get that $\left( P \mod \mathfrak{p}_v \right)_{\mid \overline{\F}_v}$ is the point given by
$$ B \mu_{n,\overline{\F}_v} \xrightarrow{i_{\pi}} \Spec(\sheafO_v^{\un}[\sqrt[n]{v}]) \xrightarrow{\overline P} X $$ that is to say $\left( P \mod \ \mathfrak{p}_v \right)_{| \overline{\F}_v} $ is isomorphic to $ \left( P_{| \K_v^{\un}} \mod  \mathfrak{p}_v \right)$.
\end{proof}

We will also need the following lemma.

\begin{lemma}\label{lemma_independance_uniformizer}
The connected component $\stackS \in \pi_0(\I_{\mu} X_{\sheafO_S})$ containing $$P \bmod \,\mathfrak{p}_v$$given by the reduction map defined via $i_{\pi}$, does not depend on the choice of the uniformizer $\pi \in \sheafO_v$.
\end{lemma}

\begin{proof}
By \cite[Lemma 4.3]{loughran_santens_malle_conjecture}, if $u \in \sheafO_v^{\times}$, we have the following $2$-commutative diagram relating the closed immersions $i_{u \pi}$ and $i_{\pi}$:

\begin{center}
\begin{tikzcd}
{B \mu_{n,\F_v}} \arrow[r, "i_{u \pi}"] \arrow[d, "t_{\overline u}"] & {\Spec(\sheafO_v[\sqrt[n]{v}])} \\
{B \mu_{n,\F_v}} \arrow[ru, "i_{\pi}"']                              &                                
\end{tikzcd}
\end{center}where the map $t_{\overline u}$ denotes the translation by $\overline u \in H^1(\F_v,\mu_n)$, the class associated with $u$. We denote by $(P \bmod \mathfrak{p}_v)_{\pi}$ the element of $\I_{\mu}X(\F_v)$ associated with the uniformizer $\pi$. Using the previous commutative diagram, the geometric points $(P \bmod \mathfrak{p}_v)_{\pi}\!\mid_{\overline{\F}_v}$ and $(P \bmod \mathfrak{p}_v)_{u\pi}\!\mid_{\overline{\F}_v}$ are isomorphic, which proves the result.
\end{proof}

By Lemma~\ref{prop_invariance_unramified_extension_reduction_map} and Lemma~\ref{lemma_reduction_map_choice_stacky_lift}, if $P \in X(\K_v)$, the connected component $\stackS \in \pi_0(\I_{\mu} X_{\sheafO_v})$ of the point $x \in \I_{\mu}X(\F_v)$ given by
$$
B\mu_{n,\F_v}
\xrightarrow{i_{\pi}}
\Spec(\sheafO_v[\sqrt[n]{v}])
\xrightarrow{\overline P}
X
$$
does not depend on the choice of the uniformizer $\pi \in \sheafO_v$, nor on the choice of a stacky lift
$$
\overline{P} \in X(\sheafO_v[\sqrt[n]{v}])
$$
of $P$. Thus the following definition (introduced by Darda and Yasuda in \cite[Definition 2.17]{darda2024batyrevmanin}) makes sense, after identifying $\pi_0(\I_{\mu} X_{\sheafO_S}) = \pi_0(\I_{\mu} X_{\K})$.

\begin{definition}\label{definition_residue_map}
For a finite place $v \not\in S$, we define a residue map
$$
\psi_v : X(\K_v) \rightarrow \sector
$$
which sends a $\K_v$-point to the unique sector corresponding to the connected component $\stackS \in \pi_0(\I_{\mu} X_{\sheafO_S}) = \sector$ containing the geometric point defined by $$P \mod \mathfrak{p}_v .$$
\end{definition}

\begin{remark}
By Proposition \ref{prop_invariance_unramified_extension_reduction_map} and by its definition, the value of $\psi_v(P)$ for $P \in X(\K_v)$ depends only on $P_{| \K_v^{\un}}$.
\end{remark}



Using the description of the universal stacky lift of a rational point, we obtain the following proposition:

\begin{prop}
    The set $\left\{ v \not\in S \mid \psi_v(P) \neq 0 \right\}$ is finite. Moreover, if $S'$ contains this set and $S$, there exists $ \bar{P} \in X(\sheafO_{S'})  $ that extends $P$.
\end{prop}

We finish this paragraph by presenting a transitive property of the residue map (see \cite[proposition 2.20]{darda2024batyrevmanin}).  Let
$$
X \xrightarrow{f} Y
$$
be a morphism of proper normal algebraic stacks with finite diagonal over $\sheafO_v$. There is a functorially induced morphism
$$
\sector \xrightarrow{f} \pi_0(\sheafI_{\mu} Y),
$$
which we also denote by $f$. Assume that $X$ and $Y$ have the property that any rational points admits a universal stacky lift of the form
$$
\Spec(\sheafO_v[\sqrt[n]{v}])
$$
for some $n \in \N^*$. Then we have the following. 

 \begin{prop}\label{prop_functoriality_residue_map}
For any $P \in X(\K_v)$, $$f \left( \psi_v(P) \right) = \psi_v \left( f(P) \right)  \, .$$
 \end{prop}

\begin{proof}
Let
$$
\Spec(\sheafO_v[\sqrt[n]{v}])
\xrightarrow{\overline P}
X
$$
be a universal stacky lift of $P \in X(\K_v)$. Then $f \circ \overline P$ is a stacky lift of $f(P) \in Y(\K_v)$. The proposition therefore follows from Lemma~\ref{lemma_reduction_map_choice_stacky_lift} and from the definition of the residue map $\psi_v$.
\end{proof}

\section{The age pairing and the orbifold Picard group}

\subsection{The age pairing}

For any $n \in \N^*$, for any connected scheme $T$, we have the following canonical isomorphism: $$\Pic(B \mu_{n,T}) = \Hom_{\grp}(\mu_{n,T},\G_{m,T}) = \frac{1}{n} \Z / \Z .$$Hence, $$\Pic( B \widehat{\mu}_T ) = \underset{\rightarrow}{\text{colim}} \Pic(B \mu_{n,T}) = \Q / \Z .$$

\begin{definition}\label{definition_age_pairing}
    The age pairing over $T$ is defined as follows:
    \begin{align*}
        &\age : \ \sheafI_{\mu} X(T) \times \Pic(X) \rightarrow \Q / \Z \\
        &\left( f : B \mu_{n,T} \rightarrow X , L \right) \longmapsto f^* L .
    \end{align*}
\end{definition}

By \cite[paragraph 7.1]{abramovich2008gromovwittentheorydelignemumfordstacks} or \cite[lemma 2.22]{darda2024batyrevmanin}, $\age(f,L)$ depends only on the connected component to which $f$ belongs, hence the age pairing may be re-written as:
\begin{align*}
    \sector &\times \Pic(X) \rightarrow \Q / \Z \\
    (\stackS&,L) \longmapsto \age(\stackS,L) .
\end{align*}

\begin{prop}\label{proposition_age_pairing_morphism_functoriality}
   The age pairing is functorial: for $X \xrightarrow{f} Y$ a morphism of nice DM stacks, if we denote again by $f$ the induced morphism $\sector \rightarrow \pi_0(\sheafI_{\mu} Y)  $ then for any $\stackS \in \sector$ and for any $L \in \Pic(Y)$, $$\age(f(\stackS),L) = \age(\stackS,f^* L) .$$
   
   Moreover, it defines a natural map $\sector \rightarrow \Hom_{\grp}(\Pic(X),\Q/\Z)$.
\end{prop}

\begin{remark}
    Recall that we denote by $0 = [X] = [ \Hom^{\rep}(B \mu_{1},X) ]$ the trivial connected component in $\sector$. For any $L \in \Pic(X)$, $\age(0,L) \in \frac{1}{1}. \Z / \Z$ i.e $\age(0,-)$ is the zero morphism.
\end{remark}

\begin{remark}
    By functoriality, we have that for any $L \in \Pic(X^{\coarse})$ and for any $\stackY \in \sector$, if we denote by $X \xrightarrow{p} X^{\coarse}$ the coarse map,  $$\age(\stackY,p^* L) = 0 .$$
\end{remark}

We shall see later (see corollary \ref{corollary_sectors_dual_picard_group_toric_stacks}) that for toric stacks the map $$\sector \rightarrow \Hom_{\grp}(\Pic(X),\Q/\Z)$$ is an injection. This is not true in general for example in the case of the classifying stack $B G$ over $\C$, associated to an abstract finite group $G$, this is true if and only if $G$ is abelian. Indeed, this is a consequence of the fact that irreducible characters and characteristic functions of conjugacy classes are two bases of the $\C$-vector space of class functions $Z( \C[G] )$. 

Nevertheless, we may try to extend corollary \ref{corollary_sectors_dual_picard_group_toric_stacks} as follows. Let $R \widehat{\mu}$ denote the ring of continuous representation of $\widehat{\mu}$ over $\C$. By definition of $\widehat{\mu}$,  $$R \widehat{\mu} = \underset{\rightarrow}{\text{colim}} R \mu_n .$$ For any representation $\mathcal{V}$ of $\mu_n$, we have $$\mathcal{V} = \bigoplus\limits_{i = 1}^r \tau_n^{a_i}$$where $\tau_n : \mu_n \rightarrow \G_m$ is the regular representation, hence we have the first Chern class $$c_1 : R \mu_n \rightarrow \frac{1}{n} \Z / \Z $$ which maps $\mathcal{V}$ to $\sum\limits_{i = 1}^r \frac{a_i}{n} \mod \Z.$ Taking the colimit, it gives the map $$c_1 : R \widehat{\mu} \rightarrow \Q / \Z .$$ If we pull-back a vector bundle/coherent sheaf $\sheafE$ over $X$ by a sector $b \in \sector$ and we apply the Chern class, we get a natural map:
$$\sector \rightarrow \Hom(K_0(X),\Q/\Z) $$which may be injective. It is at least true for toric stacks and for $B G$. The fact that it is true for $B G$ over $\C$ is simply because for any representation $\mathcal{V}$ of $G$ a finite abstract group, the trace of the pull-back of a sector of $B G$ corresponding to a conjugacy class $C$ is the value of the character of the representation $\chi_{\mathcal{V}}(C)$ seen as an element of $\Q / \Z$ once we have chosen a primitive $n$-th root of unity for each $n \in \N^*$.

\subsection{The numerical value associated to the age pairing}

We shall need the possibility to associate a numerical value to the $\age$. We follow \cite[paragraph 2.3]{darda2024batyrevmanin}. To a geometric point $\Tilde{x} \in \sheafI_{\mu} X(\F)$ (where $\F$ is an algebraically closed field), we associate a geometric point $x \in X(\F)$ and an injective group morphism $$\phi : \mu_{l} \hookrightarrow \underline{\Aut}(x)$$ over $\F$ for some $l \in \N^*$.

Let $E \rightarrow X$ be a vector bundle of rank $r$. Then we have a natural representation of $\mu_l$ associated with $\Tilde{x}$ given by $$\rho_{\Tilde{x}} : \mu_l \xrightarrow{\phi} \underline{\Aut}(x) \rightarrow \text{GL}_r( E(x)) .$$ If $\tau$ denotes the standard representation $\mu_l \rightarrow \G_m$ then there exists an isomorphism $$\rho_{\Tilde{x}} \simeq \bigoplus\limits_{i = 1}^r \tau^{a_i}$$ with $a_i \in \{ 0,..,l-1 \}$.

\begin{definition}
    The numerical age is defined by
    $$\age_{\num}(\Tilde{x},E) = \frac{1}{l} \sum\limits_{i = 1}^r a_i .$$
\end{definition}

By \cite[lemma 2.22]{darda2024batyrevmanin}, it depends only on the connected component to which $\Tilde{x}$ belongs.

\begin{remark}
    For a line bundle $L$ and a sector $\stackS \in \sector$, $\age_{\num}(\stackS,L)$ is the unique representative of $\age(\stackS,L)$ with value in $\Q \cap [0,1[.$ More generally, we have that for any $b \in \sector$ and any vector bundle $E$, $$c_1(b^* E) \equiv \age_{\num}(b,E) \mod \Z .$$
\end{remark}

\subsection{The orbifold Picard group}

\subsubsection{Definition}

For $\mathcal{C} \subset \twistsector$, the construction of the $\mathcal{C}$-extended Picard group $\Pic_{\mathcal{C}}(\stackX)$ was introduced in \cite[Definition 17]{Coates_Corti_Iritani_Tseng_miror_theorem_toric_stack} and was later rediscovered independently by Loughran and Santens in \cite[Definition 3.8]{loughran_santens_malle_conjecture}. 

\begin{definition}\label{definition_orbifold_picard_group}
The $\mathcal{C}$-extended Picard group $\Pic_{\mathcal{C}}(X)$ is defined as the group
of pairs $(L,\varphi)$, where $L \in \Pic(X)$ and
\[
\varphi : \mathcal{C} \longrightarrow \Q,
\]
such that for any $\stackS \in \mathcal{C}$,
\[
\varphi(\stackS) \equiv \age(\stackS,L) \pmod{\Z}.
\]

Equivalently, $\Pic_{\mathcal{C}}(X)$ is the kernel of the morphism
\[
\Pic(X)
\oplus
\Q^{\mathcal{C}}
\;\longrightarrow\;
\bigoplus_{\stackS \in \mathcal{C}} \Q / \Z,
\]
defined by
\[
\left( L,\,(q_{\stackS})_{\stackS \in \mathcal{C}} \right)
\longmapsto
\left( \age(\stackS,L) - [q_{\stackS}] \right)_{\stackS \in \mathcal{C}}.
\]

When $\mathcal{C} = \twistsector$, the $\twistsector$-extended Picard group is called
the \emph{orbifold Picard group} and is denoted by $\Pic_{\orb}(X)$.
\end{definition}

\begin{remark}
    In particular, we recover the construction of Darda and Yasuda of the orbifold Néron-Severi space defined in \cite[definition 8.1]{darda2024batyrevmanin}, since we have:
    $$\Pic_{\orb}(X)_{\R} \simeq \Pic(X)_{\R} \oplus \R^{\twistsector} .$$
\end{remark}

\begin{Notations}\label{notation_orbi_line_bundle_associated_to_a_sector}
For $\stackS \in \twistsector$, we denote by $[\stackS]$ the class in
$\Pic_{\orb}(X)$ given by $(\sheafO_X,\delta_{\stackS})$, where
\[
\delta_{\stackS} : \twistsector \longrightarrow \Q
\]
is the indicator function of $\{\stackS\}$, that is, for any
$\stackS' \in \twistsector$,
\[
\delta_{\stackS}(\stackS') =
\begin{cases}
1, & \text{if } \stackS' = \stackS,\\
0, & \text{otherwise}.
\end{cases}
\]
\end{Notations}

\begin{prop}\label{prop_exact_sequence_extended_picard_group}
    For $\mathcal{C} \subset \twistsector$, we have the following natural exact sequence:
    $$0 \rightarrow \Hom(\mathcal{C},\Z) \rightarrow \Pic_{\mathcal{C}}(X) \rightarrow \Pic(X) \rightarrow 0 .$$
\end{prop}

Loughran and Santens showed the remarkable fact that the orbifold Picard group of the classifying stack $B G$ is torsion free (see \cite[lemma 3.13]{loughran_santens_malle_conjecture}). In this article, we prove that this is also the case for toric stacks (see Theorem \ref{appendix_a_theorem_freeness_orbifold_picard_group} and its corollary).

\subsubsection{The orbifold degree}\label{paragraph_orbifold_degree}

Let us suppose here that we work over a field of characteristic 0. Let $f : \left(\mathcal{C},P_1,..,P_k\right) \rightarrow X$ be an orbifold stable map (see \cite[section 4.3]{abramovich2008gromovwittentheorydelignemumfordstacks}). For each geometric point $P_i$ with non trivial stabilizer $\mu_{a_i}$, we associate $\psi_{P_i}(f) \in \twistsector$ the corresponding twisted sector given by the natural map:
$$ B \mu_{a_i} \rightarrow \mathcal{C} \rightarrow X .$$

\begin{definition}\label{definition_degree_stacky_curve}
    For $(L,\varphi)$ an element of the orbifold Picard group $\Pic_{\orb}(X)$, we define the associated orbifold degree by:
$$ \deg_{\orb}\left(f^*(L,\varphi)\right) = \deg_{\mathcal{C}}(f^* L) + \sum\limits_{i = 1}^k \varphi(\psi_{P_i}(f)) .$$
\end{definition}


%
%
%

\subsubsection{The pull-back between orbifold Picard groups associated to a morphism of DM stacks}

Let $f : X \rightarrow Y$ be a morphism of Deligne-Mumford stacks. Let us denote again by $f$ the induced map $ \sector \rightarrow \pi_0( \sheafI_{\mu} Y)$. We may use this to define a pull-back between the orbifold Picard groups associated with $f$.

\begin{definition}
    For $(L,\varphi) \in \Pic_{\orb}(X)$, we define the class $f^*(L,\varphi) \in \Pic_{\orb}(Y)$ by:
    $$f^*(L,\varphi) = (f^*L,\varphi \circ f) .$$
\end{definition}

\subsubsection{The orbifold anti-canonical bundle}

\begin{Notations}\label{notation_numerical_age_tangent_bundle}
If $\stackS \in \twistsector$, we write $\age_{\num}(\stackS)$ for $\age_{\num}(\stackS,\text{TX})$ where $TX$ is the tangent vector bundle associated with $X$, that is to say the vector bundle whose sheaf of sections is $\Omega_{X/K}$.  
\end{Notations}

Following \cite[Definition 9.1]{darda2024batyrevmanin}, we give the following definition:

 \begin{definition}\label{definition_anti_canonical_orbifold}
     The class of the orbifold anticanonical line bundle $\w_{\orb,X}^{-1} $ in $\Pic_{\orb}(X)_{\R}$ is defined as the class given by $(\w_{X}^{-1},\varphi)$ where \begin{align*}
         \varphi : &\twistsector \rightarrow \Q \\
         &\stackS \longmapsto 1 - \age_{\num}(\stackS) .
     \end{align*}
\end{definition}


\subsubsection{The cone of orbifold effective divisors}

 To extend the conjecture of Peyre from \cite[question 4.8]{Peyre_beyond_height} to nice Deligne Mumford stack, we need a general definition for the cone of orbifold effective divisors. It was introduced by Darda and Yasuda in \cite[section 8]{darda2024batyrevmanin}. We follow their presentation in this paragraph. Let $X$ be a nice stack over a number field $\K$ and $\overline{X}$ its base change over $\overline{K}$. Let us recall first the definition of a covering family of stacky curves:

 \begin{definition}
     A covering family of stacky curves of $\overline{X}$ is a pair $$\left( \pi : \tilde{\mathcal{C}} \rightarrow T,  \tilde{f} : \tilde{\mathcal{C}} \rightarrow \overline{X} \right)$$ such that:
     \begin{enumerate}
         \item $\pi$ is smooth and surjective,
         \item $T$ is an integral scheme of finite type over $\overline{\K}$,
         \item for each point $t \in T(\overline{\K})$, the morphism $$\tilde{f}_{\mid \pi^{-1}(t)} : \pi^{-1}(t) \rightarrow \overline{X} $$is a stacky curve on $\overline{X}$.
         \item $\tilde{f}$ is dominant.
     \end{enumerate}
 \end{definition}

By \cite[lemma 8.6]{darda2024batyrevmanin}, there exists an open dense subscheme $T_0 \subset T$ such that for any $(L,\varphi) \in \Pic_{\orb}(X)_{\R}$ the associated orbifold degree (see definition \ref{definition_degree_stacky_curve}) via the stacky curve $\tilde{f}_{\mid \pi^{-1}(t)}$ does not depend on the choice of $t \in T_0(\overline{\K})$. Hence we can write the following definitions:

\begin{definition}\label{definition_orbifold_intersection_number}
    For a covering family $\left( \pi : \tilde{\mathcal{C}} \rightarrow T,  \tilde{f} : \tilde{\mathcal{C}} \rightarrow \overline{X} \right)$ and a $\R$-orbifold line bundle $\theta = (L,\varphi)$ we define the intersection number $(\tilde{f},\theta)$ as the orbifold degree given by $$\deg_{\orb}( \tilde{f}_{\mid \pi^{-1}(t)}^* \theta )$$where $t \in T(\overline{\K})$ is a general point.
\end{definition}

\begin{definition}\label{definition_cone_orbifold_effective_divisors}
    The cone of orbifold effective divisors $\coneorb$ is defined as the cone of elements $\theta \in \Pic_{\orb}(X)_{\R}$ such that for any covering family of stacky curves $\left( \pi : \tilde{\mathcal{C}} \rightarrow T,  \tilde{f} : \tilde{\mathcal{C}} \rightarrow \overline{X} \right)$ the intersection number is positive that is to say $$( \tilde{f},\theta) \geqslant 0 .$$
\end{definition}

\subsection{Multi-height over Deligne-Mumford stacks}

\begin{Notations}
   Let $X$ be a nice DM stack over a number field $\K$. We denote by $X^{\coarse}$ the coarse moduli space of $X$. With our hypotheses, it is a projective $\K$-scheme and it is $\Q$-factorial, with only quotient singularities.
\end{Notations}


\subsubsection{Systems of heights over the coarse space}

\begin{Notations}
    We denote by $M_K$ the set of places of $\K$, by $M_K^0$ the set of finite places and by $M_K^{\infty}$ its infinite places.

    If $v \in M_K^0$, we write $q_v = \sharp \F_v$ the cardinal of the residual field associated to $v$, while if $v \in M_K^{\infty}$, we set $q_v = \exp([\K_v:\R])$.

    To $w \in M_K$ which divides $v \in M_{\Q}$, we define $|-|_w $ over $\K_w$ by setting for any $x \in \K_w$ :
    $$ | x |_w = | N_{\K_w/\Q_v}(x) |_v .$$This is an absolute value except when $\K_w$ is $\C$ in which case it is the square of an absolute value.

    To a line bundle $L$ over a scheme $T$, we associate the inversible sheaf $\sheafL$ such that for any open subscheme $U \subset T$, $$\Gamma(U,\sheafL) = \text{Mor}_T(U,L) .$$We denote by $L^{\times} = L \setminus z(T)$, where $z(T)$ is the closed subscheme given by the zero section.
\end{Notations}

We begin by defining the notion of an adelic norm over a vector bundle.

\begin{definition}
    Let $E \xrightarrow{e} X^{\coarse}$ be a vector bundle over $X^{\coarse}$.

   For any field extension $\Le$ over $\K$, for any $P \in X^{\coarse}(\Le)$, we denote by $E_P \subset E(\Le)$ the $\Le$-vector space which corresponds to the fiber $e^{-1}(P)$ of $e$ over $P$.
   
   An adelic norm $E$ over $e$ is the data of a family of continuous applications $(||-|| )_{v \in M_K}$ 
   $$ ||-||_v : E(\K_v) \rightarrow \R_+$$
   such that:
\begin{enumerate}
    \item If $v \in M_K^0$, for any $P \in X^{\coarse}(\K_v)$, the restriction of $||-||_v$ to $E_P$ defines an ultrametric norm with values in the image of the absolute value $|-|_v$ i.e $q_v^{\Z}$.

    \item If $v$ is a real place, for any $P \in X^{\coarse}(\K_v)$, the restriction to $E_P$ is an Euclidean norm.

    \item If $v$ is a complex place, for any $P \in X^{\coarse}(\K_v)$, the restriction to $E_P$ is the square of an Hermitian norm.

    \item There exists a finite number of places $S \subset M_K$ (containing $M_K^{\infty}$) and a model $\mathcal{E} \rightarrow \mathcal{X}^{\coarse}$ of $E \rightarrow V$ over $\sheafO_S$ such that for any place $v \in M_K \setminus S$, for any $P \in \mathcal{X}^{\coarse}(\sheafO_v)$,
    $$ \mathcal{E}_P = \{ y \in E_P | \ \ ||y||_v \leqslant 1 \}$$
    
\end{enumerate}

\end{definition}

\begin{remark}
    For projective, $\Q$-factorial and split toric varieties over $\Q$ such a model always exists over $\Z$.\footnote{A model of the canonical stack (see \cite[Definition 4.4]{fantechi_toric_stack}) associated with the variety is well-defined over $\Z$, as a quotient stack, so the coarse space still exists and gives a model of the toric variety.}
\end{remark}

\begin{definition}
    If $L \rightarrow X^{\coarse}$ is a line bundle, $\left( ||-||_v \right)$ is an adelic norm over $L$, then we define the height $H^{\coarse}_L : X^{\coarse}(\K) \rightarrow \R_+$ for any $x \in X^{\coarse}(\K)$ such that:
    $$H^{\coarse}_L(x) = \prod\limits_{v \in M_K} ||y||_v^{-1}$$where $y \in L^{\times}(\K)$ lies over $x$.
\end{definition}

Now, we recall the definition of the group of Arakelov heights $\mathcal{H}(X^{\coarse})$ (see \cite[Definition 4.1]{Peyre_beyond_height}) over $X^{\coarse}$:

\begin{definition}
    Let $L$ and $L'$ be two adelic line bundles over $X^{\coarse}$ (i.e. equipped with an adelic norm). Let $(||-||_v)$ be the adelic norm on $L$.

    We say that $L$ and $L'$ are equivalent as adelic line bundles if there exist $M > 0$ and a family $(\lambda_v)_v$ of non-negative real numbers such that:
    \begin{enumerate}
        \item $\{v | \lambda_v \neq 1\}$ is finite and $\prod\limits_v \lambda_v = 1$
        \item there exists an isomorphism of adelic line bundles from $L^{\otimes M}$ to $L'^{\otimes M} $, where the adelic norm on $L^{\otimes M}$ is induced by the adelic norm on $L$ given by $\{ (\lambda_v . ||-||_v)_{v \in M_K} \} $.
    \end{enumerate}
    We denote by $\mathcal{H}(X^{\coarse})$ the set of equivalence classes of adelic line bundles. This set is naturally equipped with a group structure given by the tensor product of adelic line bundles.
\end{definition}Hence, we can define the notion of system of heights.

\begin{definition}
    A \textbf{system of heights} over $X^{\coarse}$ corresponds to the data of a section $s : \Pic(X^{\coarse}) \rightarrow \mathcal{H}(X^{\coarse})$ of the forgetful morphism $$ \mathrm{o} : \mathcal{H}(X^{\coarse}) \rightarrow \Pic(X^{\coarse}) \, .$$
\end{definition}

\subsection{Multi-heights over nice stacks}

Heights over nice Deligne-Mumford stacks were defined by Darda and Yasuda in \cite{darda2024batyrevmanin}. We recall their definition here. An important result (see \cite[proposition 3.2]{darda2024batyrevmanin}) is that the coarse map $X \xrightarrow{p} X^{\coarse}$ induces a $\Q$-isomorphism:
$$ \Pic(X^{\coarse})_{\Q} \rightarrow \Pic(X)_{\Q}  .$$To define a stacky multi-height over $X$, we assume that we have fixed a system of Arakelov heights over $X^{\coarse}$ so that for each line bundle $L$ over $ X^{\coarse}$, we have a representative $(||-||_v)_{v \in M_K}$ of an adelic metric on $L$. We can now define the coarse multi-height on $X(\K)$ (called \emph{stable height} in \cite{darda2024batyrevmanin}) as follows.

\begin{definition}\label{definition_coarse_multi_height_map}
    We define
    \[
        h_{\coarse} : X(\K) \to \Pic(X)^{\vee}_{\R}
    \]
    such that for any $P \in X(\K)$, $h_{\coarse}(P)$ is the linear form defined as follows: for any $L \in \Pic(X)_{\R}$, if $L'$ is the unique element of $\Pic(X^{\coarse})_{\R}$ such that $p^* L' = L$, then we set
    \[
        h_{\coarse}(P)(L) = \log H^{\coarse}_{L'}(p(P)).
    \]
\end{definition}

To define the stacky multi-height, we need the notion of adelic stacky data:

\begin{definition}\label{definition_adelic_stacky_data}
    An adelic stacky data over $X$ associated with an element $(L,\varphi)$ of the orbifold Picard group of $X$ is the data of a family of continuous maps $$\{\varphi_v : X(\K_v) \rightarrow \R \}_{v \in M_K^0}$$ such that for all $v \in M_K^0$, except for a finite set of places containing the places $v \mid N_X$, we have:
    $$ \varphi \circ \psi_v = \varphi_v $$where $\psi_v$ is the residue map from definition \ref{definition_residue_map} and $\varphi : \twistsector \rightarrow \Q$ is extended to $0 \in \sector$ by setting $\varphi(0) = 0$.
\end{definition}

For the rest of this paragraph, we shall suppose that we have fixed an adelic stacky data for a basis of $\Pic_{\orb}(X)$.

\begin{definition}\label{definition_stacky_height}
    Let $(L,\varphi) \in \Pic_{\orb}(X)$. Then the logarithm height associated to a pair $$(L,\varphi) \in \Pic_{\orb}(X)_{\R}$$ with our convention of a system of heights over $X^{\coarse}$ and of an adelic stacky data for each element of the orbifold Picard group is:
    $$ h(P)(L,\varphi) = h_{\coarse}(P)(L) + \sum\limits_{v \in M_K^0} \varphi_v(P) . \log(q_v) .$$
\end{definition}

\begin{Notations}
    For $(L,\varphi) \in \Pic_{\orb}(X)$, we write for any $P \in X(\K)$, $$H_{L,\varphi}(P) = e^{ h(P)(L,\varphi) }$$for the exponential height. For $\stackZ \in \twistsector$, we shall write $H_{\stackZ}$ for the exponential height $H_{[\stackZ]}$ where $[\stackZ]$ is the associated element in $\Pic_{\orb}(X)$.
\end{Notations}

\begin{prop}\label{proposition_transivity_property_of_stacky_heights}
    This definition behaves very well by pull-back, that is to say that if $(L,\varphi) \in \Pic_{\orb}(X)_ {\R}$ and $f : Y \rightarrow X$ is a morphism then we have:
    $$h(P)(f^*(L,\varphi)) = h(f(P))(L,\varphi) .$$
\end{prop}

\begin{proof}
    This was shown in \cite[proposition 4.4]{darda2024batyrevmanin} using proposition \ref{prop_functoriality_residue_map} (see also \cite[proposition 2.20]{darda2024batyrevmanin}).
\end{proof}

As the previous definition is additive, we get a natural map, called the stacky multi-height:
$$ h : X(\K) \rightarrow \Pic_{\orb}(X)^{\vee}_{\R}  .$$


\subsubsection{Result on the multi-height distribution of the rational points of toric stacks}

We define a measure $\nu$ on $\Pic_{\orb}(X)^{\vee}_{\R}$ as follows:

\begin{definition}\label{mesure_picard_group}
    $$ \nu(\mathrm{D}) = \int\limits_{\mathrm{D}} e^{\langle \w_{X,\orb}^{-1},y \rangle} dy $$for any compact subset $\mathrm{D}$ of $\Pic_{\orb}(X)^{\vee}_{\R}$, where the Haar measure $dy$ on $\Pic_{\orb}(X)^{\vee}_{\R}$ is normalized so that the covolume of the dual lattice of the orbifold Picard group is one.
\end{definition}

For any domain $\mathrm{D} \subset \Pic_{\orb}(X)^{\vee}_{\R}$, we write:
$$ X(\K)_{h \in \mathrm{D}} = \{ P \in X(\K) \mid h(P) \in \mathrm{D} \} .$$

In this article, we shall prove the following theorem:

\begin{theorem}\label{theorem_manin_peyre_conjecture}
We assume that $X$ is a toric stack over $\Q$. Let $\mathrm{D}_1$ be a finite union of compact polyhedrons of $\Pic_{\orb}(X)^{\vee}_{\R}$ and $u$ be an element of the interior of the dual of the effective cone $(\coneorb^{\vee})^{\circ}$. For a real number $B > 1$, we set:
$$ \mathrm{D}_B = \mathrm{D}_1 + \log(B) u .$$
 
Let $\epsilon \in ]0,\frac{1}{2}[$. Then the multi-height asymptotic behaviour is of the form:

$$ \sharp ( X(\Q) )_{h \in \mathrm{D}_B} \underset{B \rightarrow +\infty}{=} \nu(\mathrm{D}_1 ) . \tau_{\orb}(X) .  B^{\langle \w_{X,\orb}^{-1},u \rangle} \left( 1 + O\left(B^{- ( \frac{1}{2} - \varepsilon).\Delta} \right) \right)$$ where $\tau_{\orb}(X)$ is defined in \ref{definition_orbifold_tamagawa_number} and $\Delta > 0$ is specified in Theorem \ref{theorem_final_champ_torique}.

\end{theorem}

\begin{remark}
    In \cite{bongiorno_hyperbola_method}, using a general version of the hyperbola method developed by Pieropan and Schindler (see \cite{pieropan_hyperbola_method}), we have showed that \[
    \tau_{\orb}(X) 
    = \left(\rank_{\R}\Pic_{\orb}(X)_{\R} - 1\right)! \cdot
    \w_H\bigl(\stackX(\affine_{\Q})\bigr).
\] where $\w_H(\stackX(\affine_{\Q}))$ is the constant obtained in \cite[theorem 4.4.4]{darda_yasuda_toric_stacks_batyrev}.
\end{remark}

\section{Application of the universal torsor in the case of abelian nice stacks}

\subsection{Universal torsors for nice stacks}\label{subsection_universal_torsor}

Universal torsors were introduced by Colliot-Thélène and Sansuc in a series of foundational articles \cite{ColliotTheleneSansuc1976,ColliotTheleneSansuc1979,ColliotTheleneSansuc1987}. Later, Harari and Skorabogatov extended this notion to open varieties in \cite{HarariSkorobogatov2013}. In the setting of Deligne--Mumford stacks, this notion was revisited by Brochard in \cite{Brochard2021}, who in particular explained the computation of the type in this context and provided a geometric interpretation of the elementary obstruction. This notion was also used for algebraic stacks by Santens in \cite{santens2023brauermaninobstructionstackycurves} to study the existence of rational points on stacky curves.

Let $X \xrightarrow{\pi} S$ be a normal proper algebraic stack with finite diagonal over a scheme $S$, which we assume to be either a Dedekind scheme or a field. We assume that $\pi_* \sheafO_X = \sheafO_S$ and that $\pi$ has a section. We write $\mathcal{D}(X)$ for the derived category of sheaves on the small étale site of $X$. Let
$$K D(X) = (\tau_{\leqslant 1} R \pi_* \G_{m,X} )[1]$$
be the shift of the truncation of $R\pi_* \G_{m,X}$. The map $\G_{m,S} \rightarrow R\pi_* \G_{m,X}$ defines a morphism $\G_{m,S}[1] \rightarrow K D(X)$. Its cokernel will be denoted by $K'D(X)$. We thus obtain a distinguished triangle in $\mathcal{D}(S)$:
$$\G_{m,S}[1] \rightarrow KD(X) \rightarrow K'D(X) \xrightarrow{a_X} \G_{m,S}[2].$$
Under our assumptions,
$$K'D(X) = \R^1 \pi_* \G_{m,X} = \Pic_{X/S},$$
and since $\pi$ admits a global section, the triangle splits and $a_X = 0$.

Let $T$ be any group of multiplicative type and denote by $\widehat{T}$ its associated sheaf of characters. Then we may write the first terms of the exact sequence given by the composition via the derived functor $\textbf{RHom}(\widehat{T},-) $:

$$0 \rightarrow H^1(S,T) \rightarrow H^1(X,T) \xrightarrow{\text{type}} \Hom_S(\widehat{T},KD'(X)) .$$The type of a torsor $Y \rightarrow X$ corresponding to $[Y] \in H^1(X,T)$ is computed, according to \cite[Lemma 2.3.1]{Skorobogatov2001}, as follows: to a character $\chi$ of $T$ one associates the line bundle given by the contracted product $Y \overset{T,\chi}{\times} \affine^1$, which corresponds (see \cite{giraud1971cohomologie} and \cite[Lemma 2.2.3]{Skorobogatov2001}) to the quotient stack
$$[ Y \times \affine^1 / T ]$$
where $T$ acts by:
\begin{equation}\label{equation_definition_contracted_product}
    t \cdot (y,\lambda) = (t^{-1} y ,\chi(t)\lambda).
\end{equation}
Observe that one may equivalently view $Y \overset{T,\chi}{\times} \affine^1$
as the quotient stack $[\, Y \times \affine^1 / T \,]$ endowed with the inverse
of the action described in the equation above. This is the description that we
shall use later (see Theorem~\ref{theorem_cox_ring_universal_torsor}).

Thus $\mathrm{type}([Y])$ is the morphism given by
\begin{align*}
    \mathrm{type}([Y]) : \, &\widehat{T} \rightarrow \Pic_{X/S} \\
    & \chi \longmapsto Y \overset{T,\chi}{\times} \affine^1 \, .
\end{align*}

Brochard later provided another way to view the type of a torsor. We state his result here (see \cite[Proposition 9.3 (ii) and its proof]{Brochard2021}):

\begin{prop}\label{proposition_computation_type_torseur}
Let $X \rightarrow BT$ be a morphism corresponding to the torsor given by $[Y] \in H^1(X,T)$. Then the type of the torsor is given by the natural pull-back
$$\widehat{T} = \Pic_{BT/S} \rightarrow \Pic_{X/S}.$$
\end{prop}

We denote by $T_{\NS}$ the group of multiplicative type whose sheaf of characters is $\R^1 \pi_* \G_{m,X} = \Pic_{X/S}$.

\begin{definition}
    Any $T_{\NS}$-torsor with a rational point and whose type corresponds to the identity shall be called a universal torsor over $X$. Up to isomorphism, the universal torsors over $X$ are classified by $H^1(S,T_{\NS})$.
\end{definition}

We now describe an essential property of the universal torsor. We first explain how it can be used to compute the type of a torsor. Suppose that $\stackT \xrightarrow{q} X$ is a universal torsor and that $Y \rightarrow X$ is a $T$-torsor such that there exists a morphism
$$\phi : \stackT \rightarrow Y$$
making the following diagram commute.
\begin{center}
\begin{tikzcd}
\stackT \arrow[dd, "q"'] \arrow[rd, "\phi"] &              \\
                                            & Y \arrow[ld] \\
X                                           &             
\end{tikzcd}
\end{center}
and such that $\phi$ is equivariant with respect to a morphism of groups
$$\varphi : T_{\NS} \rightarrow T.$$
We then have the following lemma:

\begin{lemma}\label{lemma_computation_type_morphism}
Let $\lambda : \widehat{T} \rightarrow \Pic_{X/S}$ be the morphism dual to $\varphi$. Then
$$\mathrm{type}([Y]) = \lambda.$$
\end{lemma}

\begin{proof}
Using the previous commutative diagram, we obtain a natural morphism of $T$-torsors over $X$ given by
$$ \stackT \overset{T_{\NS},\varphi}{\times} T \longrightarrow Y,$$
which is therefore an isomorphism since both are $T$-torsors. Hence we obtain the following $2$-commutative diagram:
\begin{center}
\begin{tikzcd}
X \arrow[r] \arrow[d] & B T_{\NS} \arrow[ld] \\
BT                    &                     
\end{tikzcd}
\end{center}
where the horizontal arrow is given by $[\stackT]$, the vertical arrow by $[Y]$, and the diagonal arrow by $\varphi$. The lemma then follows by applying pull-backs together with Proposition~\ref{proposition_computation_type_torseur}.
\end{proof}

Suppose now that $S = \Spec(K)$. Then
$$H^0(K,\R^1 \pi_* \G_{m,X}) = \Pic(\overline{X})^{\Gal(\overline{K}/K)} = \Pic(X),$$
where the last equality follows from the fact that $X$ has a rational point. Using the functoriality of the previous exact sequence for the morphism $\underline{\Z} \rightarrow \R^1 \pi_* \G_{m,X}$ induced by $L \in \Pic(X)$, together with the fact that $H^1(K,\G_m)=0$ by Hilbert's theorem, we obtain the isomorphism
$$\stackT \overset{T_{\NS},[L]}{\times} \G_{m} \simeq L^{\times}.$$
By the preceding discussion, the resulting morphism $\stackT \rightarrow L^{\times}$ is invariant under the character $$[L] : T_{\NS} \rightarrow \G_m .$$Now if moreover we fix rational points $z \in L^{\times}(\K)$ and $y \in \stackT(\K)$, then since $$H^0(X,\G_m) = \K^{\times} ,$$ we get that there exists a unique morphism of pointed torsors $$(\stackT,y) \longrightarrow (L^{\times},z)$$ equivariant for the character $[L]$. We get the following universal property of pointed universal torsor, used in \cite[section 4.3.1]{Peyre_beyond_height}:

\begin{prop}\label{proposition_universal_property_of_pointed_versal_torsor}
For any pointed universal torsor $$(\stackT,y) \xrightarrow{q} (X,P)$$ with $P \in X(\K)$ and $y \in \stackT(\K)$, there exists a unique map $\Phi_L : (\stackT,y) \rightarrow (L^{\times},z)$ which is invariant for the character $[L]$ such that we have the following commutative diagram:
\begin{center}
\begin{tikzcd}
{(\stackT,y)} \arrow[rd] \arrow[dd,"q"] &                             \\
                                    & {(L^{\times},z) .} \arrow[ld] \\
{(X,P)}                             &                            
\end{tikzcd}
\end{center}
\end{prop}





\subsection{Lift of the age to torsors under split torus}

\begin{Notations}\label{notation_anneau_valuation_discrete}
    Let us suppose that $\sheafO$ is a henselian DVR with valuation $v$. We write $\mathfrak{p}$ for the maximal ideal of $\sheafO$, $\pi$ for a uniformizer of $\sheafO$, $\F$ for its residue field and $\K$ for its fraction field. We shall suppose that $\F$ is a perfect field and $\K$ is a local field with $\car(\K) = 0$. Let $T$ be a split torus. We shall write $X^*(T_{\K}) = H^0(\K,\widehat{T})$, or simply $X^*(T)$, for the character group of $T$.

\end{Notations}

 We suppose that $X_{\sheafO}$ admits a torsor under $T$, that we write $$q : \stackT \rightarrow X \, .$$Recall that $X$ is not necessarily tame over $\sheafO$. We suppose that $\stackT$ verifies the following hypothesis:

 \begin{hypothesis}\label{hypothesis_torsor_is_a_scheme}
    $\stackT$ is a separated scheme over $\sheafO$.
\end{hypothesis}
 
\begin{remark}
    This hypothesis is satisfied by the universal torsor and the extended universal torsor
(see Definition~\ref{definition_extended_universal_torsor}) of a split toric stack. However, for $BG$ with
$G$ a non-commutative finite group, it is more delicate to exhibit such a torsor even under
a group of multiplicative type, and the universal torsor does not provide a suitable example.
\end{remark}
 
Let $P \in X(\K)$ be a rational point and suppose there exists $y \in \stackT(\K)$ such that $q(y) = P$. With our hypothesis on $T$, such $y$ exists by Hilbert's theorem. Let us consider a finite extension $\K' / \K$ and $t \in T(\K')$ such that $$z = t . y \in \stackT(\sheafO_{\K'})$$ where $\sheafO_{\K'}$ is the valuation ring of $\K'$. Such extension exists because by proposition \ref{prop_normalization_method_stacky_lift}, we have the existence of a stacky lift even in the wild case, hence we just have to consider the field of rational functions of a connected étale presentation of this stacky lift. We can state the following lemma:

\begin{lemma}
     For any $\chi \in X^*(T)$, $\frac{v_{\K'}(\chi(t))}{e(\K' / \K)} $ does not depend on the choice of the finite extension $\K' / \K$ and of $z \in \stackT(\sheafO_{\K'})$ such that $q(z) = P_{\mid \K'}$, where $e(\K' / \K)$ is the ramification index of the extension. 
\end{lemma}

\begin{proof}
     Let $\K'' / \K$ be another finite extension such that there exists  $z' \in \stackT(\sheafO_{\K''})$ and $t' \in T(\K'')$ verifying $t' . y = z'$. Let $\Le / \K $ be a field extension which contains $\K'$ and $\K''$. Then, because $z'$ and $z$ are in the same orbit over $\sheafO_{\Le}$, there exists $u \in T(\sheafO_{\Le})$, $t' = u.t$. Since $v_{\Le}(\chi(u)) = 0$, we get:
    $$\frac{v_{\K''}(\chi(t'))}{e(\K'' / \K)} = \frac{v_{\Le}(\chi(t'))}{e(\Le / \K)} = \frac{v_{\Le}(\chi(t))}{e(\Le / \K)} = \frac{v_{\K'}(\chi(t))}{e(\K' / \K)}  .$$
\end{proof}

Using the previous lemma, the following definition has a sense:

\begin{definition}\label{def_age_univ_torsor}
   For $y \in \stackT(\K)$, we define $$\age_T(y,-) \in \left(X^*(T)\right)^{\vee}_{\Q} = \Hom_{\grp}(X^*(T),\Q)$$ such that for any $\chi \in X^*(T)$, $\age_T(y,\chi)$ is the $\Q$-value $\frac{v_{\K'}(\chi(t))}{e(\K' / \K)} $ where $$z = t.y \in \stackT(\sheafO_{\K'})$$ lies in the $T$-orbit of $y$ for some finite extension $\K' / \K$. We shall denote this by $\age_{v,T}(y,-)$ when we want to emphasize the valuation $v$.
\end{definition}

\begin{remark}
Note that the construction of this lift of the age,
$$\age_{T,v},$$
depends not only on $T$ but also on the torsor $q : \stackT \rightarrow X$. This construction will mainly be used with $T = T_{\NS}$ and $[\stackT]$ a universal torsor (when
$\Pic(X)$ is torsion free), or with $T = T_{\NS,\orb}$, where $T_{\NS,\orb}$ denotes the
torus dual to $\Pic_{\orb}(X)$ and $[\stackT]$ is the extended universal torsor (see Definition \ref{definition_extended_universal_torsor}). We shall sometimes abuse notation and write
$\age_v(y,-)$ when the underlying torus $T$ is clear from the context. 

\end{remark}

\begin{Notations}
For the rest of the article, if $T$ is a split torus, we denote by
\[
\log_{T,v} : T(\K) \longrightarrow \bigl(X^*(T)\bigr)^{\vee}
\]
the map defined as follows: for any $t \in T(\K)$ and any $\chi \in X^*(T)$,
\[
\log_{T,v}(t)(\chi) = v\!\left(\chi(t)\right).
\]
  
\end{Notations}
 We have the following useful proposition:

\begin{prop}\label{prop_age_and_logarithm}
    Let $y \in \stackT(\K)$ and $t \in T(\K)$. We have:
    $$\age_v(t.y,-) = \age_v(y,-) - \log_{T,v}(t) .$$
\end{prop}

\begin{example}\label{example_computation_age_weighted_projective_stack}
    We consider the case where $X = \stackP(\w)$ is the weighted projective stack with weight $\w = (\w_0,..,\w_n)$. For $y \in \K^{n+1}-\{0\}$, let us write $$\min\limits_{0 \leqslant i \leqslant n} \frac{v(y_i)}{\w_i} = m + t $$ where $t \in \Q \cap [0,1[$ and $m \in \Z$. Suppose we write $t = \frac{r}{d}$ with $\gcd(r,d) = 1$. Consider $\Le / \K$ a rupture field of $X^d - \pi$. Let $\gamma \in \Le$ denote a root of this polynomial. It is a uniformizer of $\Le$. Then if $u = \pi^{-m} . \gamma^{-r} \in \Le^{\times}$, we have $$u \underset{\w}{.} y \in \left( \affine^{n+1}-\{0\} \right)(\sheafO_{\Le})$$hence $$ \age_v(y,\sheafO(-1)) = - \min\limits_{0 \leqslant i \leqslant n} \frac{v(y_i)}{\w_i}.$$
    
\end{example}

We can also describe the behavior of the lift of the age after a field extension:

\begin{prop}\label{proposition_age_and_field_extension}
    Let $\K' / \K$ be a finite extension. Let $y \in \stackT(\K)$. Then for any $\chi \in X^*(T_{\K})$:
    $$\age_{v_{\K'}}(y_{|\K'},\chi_{|\K'}) = e(\K' / \K) . \age_v(y,\chi) .$$
\end{prop}

In the rest of this article, we shall need the following hypothesis, which is verified by the universal torsor $\stackT_{\Sigma}$ (when $\Pic(X)$ is torsion free) and by the extended universal torsor $\stackT_{\Sigma} \times \G_m^{\twistsector}$ for a split toric stack (see proposition \ref{prop_extended_univ_torsor_zero_age}) :

\begin{hypothesis}\label{hypothesis_zero_of_local_degree}
    For any finite field extension $\K' / \K$, we have:
    $$\stackT(\sheafO_{\K'}) = \{y \in \stackT(\K') \mid \age_{v_{\K'}}(y,-) = 0 \} . $$
\end{hypothesis}

\subsection{Description of the stacky lift using local degrees}\label{paragraph_wild_residue_map}

We continue to fix a $T$-torsor $\stackT \rightarrow X$ defined over $\sheafO$.
We assume that $\stackT$ satisfies Hypotheses~\ref{hypothesis_torsor_is_a_scheme}
and~\ref{hypothesis_zero_of_local_degree}.
We write
\[
X_{\sheafO} = [\stackT_{\sheafO} / T_{\sheafO}] .
\]

We keep the hypothesis that $X_{\K}$ is a nice tame DM stack.
We shall suppose again that $T$ is a split torus, i.e. $\widehat{T}$, the sheaf of characters of $T$, is constant
(hence $X^{*}(T_{\sheafO}) = X^{*}(T_{\K})$) and torsion free. We fix $y \in \stackT(\K)$ and denote by
\[
P = q(y) \in X(\K)
\]
the corresponding rational point. We shall begin by exhibiting a field extension
$\K'/\K$ such that $y$ lies in the orbit of an integral point in
$\stackT(\sheafO_{\K'})$.

\begin{prop}
Let  $N \in \N^*$ such that $$N . \age_v(y,-) \in X^*(T)^{\vee} .$$Let $\K'$ be a rupture field of the irreducible polynomial $X^N - \pi$. It is a totally ramified extension of degree $N$ and $\sheafO_{\K'} = \sheafO[\gamma]$ where $\gamma \in \K'$ is a uniformizer of $\K'$ satisfying $\gamma^N = \pi$.

Then there exists $t \in T(\K')$ such that $t . y \in \stackT(\sheafO_{\K'})$.
\end{prop}

\begin{proof}
Using proposition \ref{proposition_age_and_field_extension}, we have that:
$$\age_{v_{\K'}}(y,-) = N . \age_v(y,-) \in X^*(T)^{\vee} .$$Hence we can consider $t \in T(\K')$ such that $\log_{T,v_{\K'}}(t) = \age_{v_{\K'}}(y,-)$ and we get $z = t . y \in \stackT(\sheafO_{\K'})$ using proposition \ref{prop_age_and_logarithm} and the fact that we have assumed hypothesis \ref{hypothesis_zero_of_local_degree}.
\end{proof}

\begin{Notations}\label{notation_entier_lcm_local_degree}
    Let $d \in \N^*$ be a generator of the ideal given by the $N \in \Z$ such that $N . \age_v(y,-) \in X^*(T)^{\vee}$. If $\base = (\chi_1,..,\chi_r)$ is a basis of $X^*(T)$ and for $1 \leqslant i \leqslant r$, we write $\age_v(y,\chi_i) = \frac{r_i}{d_i}$ with $r_i \in \Z$, $d_i \in \N^*$ and $\gcd(r_i,d_i) = 1$ then $d = \text{lcm}(d_1,..,d_r)$.
\end{Notations}

In the end of this paragraph, we shall need the following lemma:

\begin{lemma}\label{lemma_representabilite_stacky_lift}
    Let $\K'$ be a rupture field of the irreducible polynomial $X^d - \pi$. Then the integers $d, v_{\K'}(\chi_1(t)),..,v_{\K'}(\chi_r(t))$ are coprime where $t \in T(\K')$ is chosen as before.
\end{lemma}

\begin{proof}
    We set $a_i = v_{\K'}(\chi_i(t))$. We have $ a_i . d_i = r_i . d$. Let $p$ a prime integer such that $v_p(d) > 0$ and consider $k \in \{1,..,r\}$ such that $v_p(d_k) = v_p(d)$. Because $\gcd(d_k,r_k) = 1$, we get $v_p(a_k) = 0$ so we have proven the lemma.
\end{proof}

Now we consider again an integer $N \in \N^*$ such that $N . \age_v(y,-) \in X^*(T)^{\vee}$. We shall only later specialize to the case where $N = d$. We write again $\K'$ for a rupture field of the irreducible polynomial $X^N - \pi$. Let us consider the group morphism $$\varphi : \mu_{N,\sheafO} \rightarrow T$$ such that for any $\chi \in X^*(T)$, $$\chi(\varphi(\xi)) = \xi^{v_{\K'}\left(\chi(t)\right)} .$$Recall that we have constructed in the previous paragraph $z = t . y \in \stackT(\sheafO_{\K'})$. We have a natural action of $\mu_{N,\sheafO}$ on $\Spec(\sheafO_{\K'})$ over $\Spec(\sheafO)$.

\begin{prop}
    The morphism $$z : \Spec(\sheafO_{\K'}) \rightarrow \stackT_{\sheafO}$$ is $\varphi$-equivariant and hence induces a morphism $$\Spec(\sheafO[\sqrt[N]{v}]) = [\Spec(\sheafO_{\K'}) / \mu_{N,\sheafO}] \rightarrow X_{\sheafO} .$$
\end{prop}

\begin{proof}
    
To prove the proposition, it suffices to check that the following diagram is
commutative:
\begin{center}
\begin{tikzcd}
{\mu_{N,\sheafO} \times_{\sheafO} \Spec(\sheafO_{\K'})}
\arrow[d, "\varphi \times z"']
\arrow[r, "a"]
&
\Spec(\sheafO_{\K'})
\arrow[d, "z"]
\\
{T_{\sheafO} \times_{\sheafO} \stackT_{\sheafO}}
\arrow[r, "b"]
&
\stackT_{\sheafO}
\end{tikzcd}
\end{center}
where $a$ and $b$ are the maps induced by the natural actions. Since the morphism
\[
\mu_{N,\K} \times_{\K} \Spec(\K')
\longrightarrow
\mu_{N,\sheafO} \times_{\sheafO} \Spec(\sheafO_{\K'})
\]
is dominant, $\mu_{N,\sheafO} \times_{\sheafO} \Spec(\sheafO_{\K'})$ is reduced, and
$\stackT_{\sheafO}$ is separated, it is enough to show that the diagram commutes after
the base change $- \times_{\sheafO} \Spec(\K)$.

This is easier, since for any $\K$-scheme $S$, for any $\xi \in \mu_{N,\K}(S)$ and any
$s \in \Spec(\K')(S)$, we have
\[
z(\xi \cdot s) = t(\xi \cdot s) \cdot y(\xi \cdot s).
\]
As $y \in \stackT(\K)$, we have $y(\xi \cdot s) = y(s)$, so it suffices to show that
\[
t(\xi \cdot s) = \varphi(\xi) \cdot t(s).
\]
This is equivalent to requiring that for any character $\chi \in X^{*}(T)$,
\[
\chi(t)(\xi \cdot s) = \xi^{v_{\K'}(\chi(t))} \cdot \chi(t)(s).
\]
In other words, the following diagram is commutative:
\begin{center}
\begin{tikzcd}
{\mu_{N,\K} \times_{\K} \Spec(\K')}
\arrow[d, "{\varphi \times \chi(t)}"']
\arrow[r, "a"]
&
\Spec(\K')
\arrow[d, "{\chi(t)}"]
\\
{\G_{m,\K} \times_{\K} \G_{m,\K}}
\arrow[r, "m"]
&
{\G_{m,\K}}
\end{tikzcd}
\end{center}
where $m$ denotes the multiplication map. Since all schemes involved are affine, the
commutativity of this diagram is equivalent to the commutativity of the corresponding
diagram in the category of $\K$-algebras:
\begin{center}
\begin{tikzcd}
{\K[U,U^{-1}]}
\arrow[r, "U \mapsto X \otimes Y"]
\arrow[d, "{U \mapsto \chi(t)}"']
&
{\K[X,X^{-1}] \otimes \K[Y,Y^{-1}]}
\arrow[d]
\\
{\K[T]/(T^{n}-\pi)}
\arrow[r, "{\overline{T} \mapsto \overline{T} \otimes \overline{S}}"']
&
{\K[T]/(T^{n}-\pi) \otimes \K[V]/(V^{n}-1)}
\end{tikzcd}
\end{center}
where the vertical morphism on the right is given by
\[
X \longmapsto \chi(t) \otimes 1
\qquad\text{and}\qquad
Y \longmapsto 1 \otimes \overline{V}^{v_{\K'}(\chi(t))}.
\]
The commutativity of the diagram then follows from the fact that $\chi(t)$ is sent to
$\chi(t) \otimes \overline{V}^{v_{\K'}(\chi(t))}$ by the lower horizontal morphism.

\end{proof}

Now using the notation of \ref{notation_entier_lcm_local_degree}, we set $N = d$ and using lemma \ref{lemma_representabilite_stacky_lift}, we get that $\varphi_{| \F}$ is injective so by proposition \ref{proposition_carac_universal_stacky_lift_representability}, we may write the following corollary:

\begin{cor}\label{corollary_description_universal_stacky_lift}
    The above construction gives a stacky lift $$\Spec(\sheafO[\sqrt[d]{v}]) = [\Spec(\sheafO_{\K'}) / \mu_{d,\sheafO}] \rightarrow X_{\sheafO} $$ of $P \in X(\K)$ which is the universal stacky lift up to isomorphism, where $\K'$ is a rupture field of $X^d - \pi$.
\end{cor}

All the definitions written in paragraph \ref{subsubsection_the_residue_map_definition} in the tame case are still valid in the wild case with our hypotheses. That is to say, we have defined a natural reduction map $$- \text{mod} \ \mathfrak{p} : X(\K) \rightarrow \sheafI_{\mu} X(\F)$$ corresponding to a stacky Hensel's lemma as in \cite[lemma 4.3]{loughran_santens_malle_conjecture} and we can naturally define a residue map $$\psi_v : X(\K) \rightarrow \sector$$ even in the case where $X_{\sheafO}$ is not tame.

\subsection{Interpretation of the residue map via torsors under multiplicative group}\label{section_interpretation_residue_map_univ_torsor}

We keep the hypotheses of the previous section and we write again $\stackT \xrightarrow{q} X$ over $\sheafO$ a $T$-torsor over $X$.

\begin{remark}
Many of the constructions that follow can be generalized to the case where $T$ is
not split. Nevertheless, we retain
our standing assumptions on $T$ for the remainder of the article.
\end{remark}

For the remainder of this section, we shall need a slight extension of the age pairing. Let
$\alpha : X \to B T$ denote the associated morphism to the $T$-torsor $$q : \stackT \rightarrow X$$over $\sheafO$. We shall also denote, by a slight abuse of notation, $\alpha = [\stackT] \in H^1(X,T)$ for the class of this $T$-torsor. Let $b \in \sheafI_{\mu} X (S)$, with $S$ connected, corresponding to a morphism
\[
b : B \mu_{N,S} \to X.
\]

\begin{definition}\label{definition_general_age_pairing}
For any character $\chi \in X^*(T)$, we define
\[
\age_{\alpha,T}(b,\chi) = b^*\bigl(\alpha^*(\chi)\bigr) \in \tfrac{1}{N}\Z/\Z.
\]
This defines an age pairing associated with the multiplicative group $T$,
\[
\age_{\alpha,T} : \sector \times X^*(T) \longrightarrow \Q/\Z.
\]
\end{definition}

By Proposition~\ref{proposition_computation_type_torseur}, we obtain the following result:

\begin{prop}\label{prop_computation_general_age_pairing_via_type_torseur}
If $\lambda = \mathrm{type}([\stackT])$, then
$$\age_{\alpha,T}(b,\chi) = b^*\bigl(\lambda(\chi)\bigr).$$
In particular, if $T = T_{\NS}$ and $[\stackT]$ is a universal torsor, then the age pairing associated with this torsor coincides with the usual age pairing (see Definition~\ref{definition_age_pairing}), i.e. in this case
$$\age_{\alpha,T_{\NS}} = \age.$$
\end{prop}

\subsubsection{Interpretation of the age pairing using a torsor}

For any connected $\sheafO$-scheme $S$, $$\Tilde{x} \in \sheafI_{\mu} X[S]$$ is given by the data of $x \in X[S]$ and an injective morphism $\mu_{d,S} \hookrightarrow \underline{\Aut}(x)$ for some $d \in \N^*$. 

With our hypotheses, we may write $X = [\stackT/T]$. Using the description given in \cite[paragraph 4.2]{loughran_santens_malle_conjecture}, the point $x$ can be lifted as $y \in \stackT \overset{T}{\times} Q (S) $ where $Q \rightarrow S$ is a $T$-torsor and hence $$\underline{\Aut}(x) = \underline{\stab}_{T}(y) $$since $\stackT$ is a scheme. We get the following natural injection:
$$ \mu_{d,S} \hookrightarrow  \underline{\stab}_{T}(y) \hookrightarrow T .$$

Now let us consider $P \in X(\K)$. We set
\[
\Tilde{x}
=
\left( P \bmod \mathfrak{p} \right)
\in \sheafI_{\mu} X(\F).
\]Let $\overline{P} \in X(\sheafO[\sqrt[d]{v}])$ be the universal stacky lift of $P$. We denote by $x \in X(\F)$ the image of $\Tilde{x}$. Consider an element $y \in \stackT(\F)$ lifting $x$ via the torsor $\stackT \rightarrow X$. We denote again by $\alpha \in H^{1}(X,T)$ its class, which corresponds to a morphism
\[
X \xrightarrow{\alpha} B T.
\]Recall that by proposition \ref{proposition_cohomology_local_root_stack}, we have the natural isomorphisms
$$H^1(\sheafO[\sqrt[d]{v}],T) = H^1(B \mu_{d,\F},T) = \Hom(\mu_{d,\F},T_{\F}),$$
which allow us to identify
\[
H^{1}\!\left(\sheafO[\sqrt[d]{v}], T\right)
\quad\text{with}\quad
\Hom_{\grp}\!\left(X^{*}(T), \tfrac{1}{d}\Z / \Z\right).
\]We may then establish the following lemma:

\begin{lemma}\label{lemma_formule_lift_age}
Then $\alpha(\overline{P})$ corresponds to the morphism
\[
\chi \in X^{*}(T)
\longmapsto
\age_{\alpha,T}\bigl(\psi(P), \chi\bigr),
\]
where $\psi : X(\K) \rightarrow \sector$ denotes the residue map. It is moreover induced by the group morphism defined earlier:
\[
\mu_{d,\F}
\hookrightarrow
\underline{\stab}_{T}(y)
\hookrightarrow
T_{\F}.
\]
\end{lemma}

\begin{proof}
The pull-back of $\alpha$ along the morphism
\[
\Tilde{x} : B \mu_{d,\F} \longrightarrow X
\]
is given by the isomorphism class of the composition
\[
B \mu_{d,\F} \longrightarrow X \longrightarrow B T,
\]
which sends the trivial point to
\[
[x^* \stackT] = 0 \in H^1(\F,T),
\]
since there exists $y \in \stackT(\F)$ lifting $x$. Hence this composition is determined by the group morphism defined earlier:
\[
\mu_{d,\F}
\hookrightarrow
\underline{\stab}_{T}(y)
\hookrightarrow
T_{\F}.
\]
We conclude the proof using the following commutative diagram, where the horizontal maps are induced by $\alpha$:
\begin{center}
\begin{tikzcd}
{\mor(B \mu_{d,\F},X)} \arrow[r]                               & {H^1(B \mu_{d,\F},T) = \Hom(\mu_{d,\F},T_{\F})}                                    \\
{\overline{P} \in X(\sheafO[\sqrt[d]{v}])} \arrow[u] \arrow[r] & {{\alpha(\overline{P}) \in H^1(\sheafO[\sqrt[d]{v}],T)}  \, .} \arrow[u, "\simeq"]
\end{tikzcd}
\end{center}
\end{proof}

Moreover, we have given an explicit description of the morphism
\[
\mu_{d,\F} \hookrightarrow T_{\F}
\]
via Corollary~\ref{corollary_description_universal_stacky_lift}. It is given by
$\varphi$ such that, for any $\chi \in X^*(T)$,
\[
\chi\bigl(\varphi(\xi)\bigr) = \xi^{v_{\K'}(\chi(t))},
\]
where the notations are the same as in the previous section. Hence, we can deduce from the lemma \ref{lemma_formule_lift_age}, the following crucial theorem:

\begin{theorem}\label{theorem_age_extended_universal_torsor}
For any $\chi \in X^{*}(T)$, we have:
\[
\age_{\alpha,T}\bigl(\psi_v(P),\chi\bigr) \equiv \age_{v,T}(y,\chi) \mod \Z .
\]
\end{theorem}

This yields the following corollary:

\begin{cor}
We have the following commutative diagram:
\begin{center}
\begin{tikzcd}
\stackT(\K) \arrow[r] \arrow[d] & {\Hom_{\grp}(X^{*}(T),\Q)} \arrow[d] \\
X(\K) \arrow[r]                 & {\Hom_{\grp}(X^{*}(T),\Q / \Z)}
\end{tikzcd}
\end{center}
where the horizontal maps are given by
\[
y \in \stackT(\K) \longmapsto \age_{v,T}(y,-)
\]
and by
\[
P \in X(\K) \longmapsto \age_{\alpha,T}\bigl(\psi_v(P),-\bigr).
\]
\end{cor}





Recall that, by Proposition~\ref{prop_sector_over_ring_of_integer}, we may identify $\pi_0(\I_{\mu} X_{\sheafO})$ with $\sector$. The previous theorem will be particularly useful under the following additional assumption, which is satisfied for toric stacks (see Theorem \ref{theorem_sectors_dual_picard_group_toric_stacks}):

\begin{hypothesis}\label{hypothesis_sectors_dual_picard_group}
The map
\begin{align*}
\sector &\longrightarrow \Hom_{\grp}(\Pic(X),\Q/\Z) \\
\stackS &\longmapsto \age(\stackS,-)
\end{align*}
is injective.
\end{hypothesis}


\subsubsection{Functoriality of the age pairing}

Instead of using the classical age pairing, we shall work with the age pairing associated with the extended universal torsor (see Definition~\ref{definition_extended_universal_torsor}), as defined in Definition~\ref{definition_general_age_pairing}. To exploit the fact that Hypothesis~\ref{hypothesis_sectors_dual_picard_group} holds for toric stacks, we shall use the following proposition:

\begin{prop}\label{prop_age_functoriality_torsor}
Let $f : Y \rightarrow X$ be a $T$-torsor, and let
\[
\lambda : \widehat{T} \longrightarrow \R^{1} p_{*}\G_{m,X}
\]
denote the type of the torsor. For any $\chi \in X^{*}(T)$ and any
$\stackS \in \sector$, we have
\[
\age_{\alpha,T}(\stackS,\chi) = \age\bigl(\stackS,\lambda(\chi)\bigr).
\]
\end{prop}

\begin{proof}
This is a reformulation of Proposition~\ref{prop_computation_general_age_pairing_via_type_torseur}.
\end{proof}

We shall also need the following result of functoriality in the sequel.

\begin{prop}\label{prop_functoriality_lif_age}
Let $X_1$ and $X_2$ be nice Deligne--Mumford stacks. For $i \in \{1,2\}$, let
$f_i : Y_i \rightarrow X_i$ be a $T_i$-torsor over $\sheafO$ where $T_i$ are split tori and $Y_i$ are separated schemes.

Assume that we are given a morphism $r : Y_1 \rightarrow Y_2$, which is equivariant
with respect to a group morphism $\phi : T_1 \rightarrow T_2$, induced by a
morphism of character groups
\[
\lambda : X^{*}(T_2) \longrightarrow X^{*}(T_1).
\]
Then, for any $y_1 \in Y_1(\K)$ and any $\chi_2 \in X^{*}(T_2)$, we have
\[
\age_{v,T_1}\bigl(y_1,\lambda(\chi_2)\bigr)
=
\age_{v,T_2}\bigl(r(y_1),\chi_2\bigr).
\]
\end{prop}

\begin{proof}
    The proof is a consequence of the equivariance of $r$ and of the definition of the
lift of the age (see Definition~\ref{def_age_univ_torsor}).
\end{proof}

\subsection{Local heights defined over the universal torsors}\label{subsection_local_height_coarse}

In this section, let $\K$ be a number field. We consider a universal torsor
$\stackT \xrightarrow{q} X$ over $X$, defined over the ring of $S$-integers
$\sheafO_S$, where $M_{K}^{\infty} \subset S$, and we assume that $T_{\NS}$ is a split torus. Let $v \in M_K$ be a
place, and recall that we denote by
\[
p : X \longrightarrow X^{\coarse}
\]
the coarse moduli map.

The aim of this section is to define a local degree
\[
h_{\stackT,v} : \stackT(\K_v) \longrightarrow \Pic(X)^{\vee}_{\Q},
\]
in the spirit of Peyre’s construction
\cite[Construction~4.27]{Peyre_beyond_height}. We fix a system of heights
\[
s : \Pic(X^{\coarse}) \longrightarrow \mathcal{H}(X^{\coarse}),
\]
an integral point $y_0 \in \stackT(\sheafO_S)$, and a place
$v_0 \in M^{\infty}_K$.

For any line bundle $L$, Proposition~\ref{proposition_universal_property_of_pointed_versal_torsor}
provides, up to multiplication by a scalar in $\K^{\times}$, a unique morphism
\[
\Phi_L : \stackT \longrightarrow L^{\times},
\]
which is equivariant with respect to the group morphism
\[
[L] : T_{\NS} \longrightarrow \G_m
\]
associated with $[L] \in X^{*}(T_{\NS}) = \Pic(X)$. For each $[L']$ in $\Pic(X^{\coarse})$, we choose a representative $(||-||_v)$ of the adelic norm associated to $s([L'])$. We write $\Phi_{L'}$ for the composition of the map $$\Phi_{p^* L'} : \stackT(\K_v) \rightarrow (p^* L')^{\times}(\K_v)$$with the map $\left( p^*L' \right)(\K_v) \rightarrow L'(\K_v)$. Then we define for any $v \in M_{\K}$ the map:

\begin{align*}
    ||-||_{L',v} : \stackT&(\K_v) \rightarrow \R \\
    &y \longmapsto \frac{||\phi_{L'}(y))||_v}{||\phi_{L'}(y_0)||_v} \text{  if } v \neq v_0 \\
    &y \longmapsto \frac{||\phi_{L'}(y)||_{v_0}}{||\phi_{L'}(y_0)||_{v_0}} H_{L'}\left(p(q(y_0))\right)^{-1}
\end{align*}where $H_{L'}\left(p(q(y_0))\right)$ is the height of the rational point $P_0 = p(q(y_0)) \in X^{\coarse}(\K) $ in the usual sense. Recall that $p^*$ induces an isomorphism $$\Pic(X^{\coarse})_{\Q} \rightarrow \Pic(X)_{\Q} .$$

\begin{definition}\label{definition_naive_local_degree}
    The local degree $h_{\stackT,v} :\stackT(\K_v) \rightarrow \Pic(X)^{\vee}_{\Q}$ at the place $v \in M_{\K}$ associated to the universal torsor $\stackT$ is defined for $y \in \stackT(\K_v)$ so that for any $L' \in \Pic(X^{\coarse})$ we have:
    $$||y||_{L',v} = q_v^{- \left\langle (p^*)^{\vee}(h_{\stackT,v}(y)) , L' \right\rangle }  .$$

\end{definition}

Let $L \in \Pic(X)$. For any $y \in \stackT(\K)$, setting
$P = q(y) \in X(\K)$, we have, by Definitions \ref{definition_coarse_multi_height_map} and \ref{definition_naive_local_degree}, that:
\begin{equation}\label{equation_local_height_global_height_coarse_space}
h_{\coarse}\bigl(P\bigr)(L)
=
\sum_{v \in M_{\K}}
h_{\stackT,v}(y)(L)\, \log(q_v).
\end{equation}

Let $v \in M_{\K}$. Recall that
\[
\log_{T_{\NS},v} : T_{\NS}(\K_v) \longrightarrow \Pic(X)^{\vee}_{\R}
\]
is defined so that for any $t \in T_{\NS}(\K_v)$ and any $L \in \Pic(X)$,
\[
\log_{T_{\NS},v}(t)(L)
=
\begin{cases}
v\!\left([L](t)\right), & \text{if } v \in M_{\K}^{0}, \\[0.3em]
- \log_v\!\left(\lvert [L](t) \rvert_v\right), & \text{if } v \in M_{\K}^{\infty}.
\end{cases}
\]

We can then state the following lemma.

\begin{lemma}\label{lemma_hauteur_locale_torseur_univ_et_action_de_tns}
For any $t \in T_{\NS}(\K_v)$ and any $y \in \stackT(\K_v)$, we have
\[
h_{\stackT,v}(t \cdot y)
=
h_{\stackT,v}(y) + \log_{T_{\NS},v}(t).
\]
\end{lemma}

We shall also need the following properties of $h_{\stackT,v}$.

\begin{prop}\label{proposition_properties_local_degree}
We have the following properties:
\begin{enumerate}
\item
If $\Le_w / \K_v$ is an extension of $p$-adic fields and $y \in \stackT(\K_v)$, then for any $L \in \Pic(X)$,
\[
h_{\stackT,w}(y) = e(w \mid v)\, h_{\stackT,v}(y),
\]
where $e(w \mid v)$ denotes the ramification index.

\item
If $v \not\in S$, then for any $y \in \stackT(\sheafO_v)$, we have $h_{\stackT,v}(y) = 0$.
\end{enumerate}
\end{prop}

We can now state the following theorem.

\begin{theorem}\label{theorem_local_degree_local_age}
Let $v$ be a finite place not contained in $S$.

Then for any
$y \in \stackT(\K_v)$, we have
\[
h_{\stackT,v}(y) = - \age_{v,T_{\NS}}(y,-).
\]
\end{theorem}

\begin{proof}
Let $y \in \stackT(\K_v)$. By
Proposition~\ref{prop_normalization_method_stacky_lift}, we may consider the fraction
field of a connected étale presentation of the universal stacky lift of
$q(y) = P \in X(\K_v)$. Hence there exists a finite extension
$\Le_w / \K_v$ of $p$-adic fields and an element $t \in T_{\NS}(\Le_w)$ such that
\[
t \cdot y \in \stackT(\sheafO_w).
\]
Using Lemma~\ref{lemma_hauteur_locale_torseur_univ_et_action_de_tns} and
Proposition~\ref{proposition_properties_local_degree}, we obtain
\[
h_{\stackT,v}(y)
=
\frac{- 1}{e(w \mid v)}\, \log_{T_{\NS},w}(t),
\]
which is precisely $- \age_{v,T_{\NS}}(y,-)$.
\end{proof}

\section{Introduction to toric stacks}

\subsection{Preliminaries}

In \cite{borisov_toric_stack}, Borisov, Chen and Smith introduced toric stacks which was later characterized in different ways in \cite{fantechi_toric_stack} by Fantechi and Mann. Let us first define the notion of \textit{stacky fan}.

\begin{definition}\label{definition_stacky_fan_toric_stack}
    A stacky fan is a triple $$\mathbf{\Sigma} = (\Sigma,\Z^{\Sigma(1)} \xrightarrow{\beta} N)$$ where $N$ is a finitely generated abelian group and $\Sigma$ is a rational complete simplicial fan of $N_{\Q} = N \otimes_{\Z} \Q $ whose set of rays is denoted by $\Sigma(1)$. We suppose moreover that if we denote by $\overline{\beta}$ the composition of $\beta$ with the natural map $N \rightarrow N_{\Q}$, then for any edge $\rho \in \Sigma(1)$, $\overline{\beta}(e_{\rho}) \in \rho$.
\end{definition}

\begin{remark}
    With this definition, $\coker(\beta)$ is finite.
\end{remark}

We define a quasi-affine scheme $\stackT_{\Sigma}$ over $\Z$ as an open subset of $\affine^{\Sigma(1)}$ obtained by cutting out a certain closed subset $Z(\Sigma)$. As in \cite[Chapter 5, §1]{CoxLittleSchenck2011}, we write for a cone of the fan $\sigma \in \Sigma$, the monomial:
$$ X^{\Tilde{\sigma}} = \prod\limits_{\rho \not\in \sigma(1)} X_{\rho} $$where $\sigma(1)$ is the set of rays of the cone $\sigma$. Then we shall use the following definition of the exceptional set $Z(\Sigma)$:

\begin{definition}
    We denote by $B(\Sigma)$ the ideal of $\Z[X_{\rho},\rho \in \Sigma(1)]$ generated by the monomials $X^{\Tilde{\sigma}} = \prod\limits_{\rho \not\in \sigma(1)} X_{\rho}$ where $\sigma$ belongs to $\Sigma_{\max}$, the set of cones of maximal dimension of the fan $\Sigma$. Then we define the closed subset of $\affine^{\Sigma(1)}$ as:
    $$Z(\Sigma) =  V(B(\Sigma)) = \bigcap\limits_{\sigma \in \Sigma_{\max}} \left( \bigcup\limits_{\rho \not\in \sigma(1)} \left( X_{\rho} = 0 \right) \right).$$We define $\stackT_{\Sigma} = \affine^{\Sigma(1)} - Z(\Sigma)$.
\end{definition}

We shall give here another description of $\stackT_{\Sigma}$ as an open subset of $\affine^{\Sigma(1)}$ obtained by cutting out a certain closed subset which can be described in terms of primitive collections which are defined as follows (see \cite[definition 5.1.5]{CoxLittleSchenck2011}):


\begin{definition}\label{def_primitive_collection}
    A subset $C \subset \Sigma(1)$ is a primitive collection if:
\begin{itemize}
    \item $C \not\subset \sigma(1)$ for any $\sigma \in \Sigma$
    \item for every proper subset $C' \subsetneq C$, there exists $\sigma \in \Sigma$ such that $C' \subset \sigma(1)$
\end{itemize}
\end{definition}

Hence $Z(\Sigma)$ can be described as a union of irreducible components as in \cite[proposition 5.1.6]{CoxLittleSchenck2011}:
\begin{prop}\label{description_torseur_univ_primitive_collection}
    $$Z(\Sigma) = \bigcup\limits_{C} V( x_{\rho} \mid \rho \in C ) $$where $C$ goes over the set of primitive collections.
\end{prop}

Now let us consider the following exact triangle in $\mathcal{D}(\text{Ab})$:
$$ \Z^{\Sigma(1)} \xrightarrow{\beta} N \rightarrow \coneform(\beta) .$$ We apply the functor $\RHom(-,\Z)$, so we get the following long exact sequence:

$$ 0 \rightarrow N^* \xrightarrow{\beta^*} (\Z^{\Sigma(1)})^* \rightarrow \mathbf{R^1 Hom}(\coneform(\beta),\Z) \rightarrow \Ext^1(N,\Z) \rightarrow 0 .$$If we write $DG(\beta) = \mathbf{R^1 Hom}(\coneform(\beta),\Z)$, we shall denote by $G_{\beta}$ the associated group of multiplicative type, which acts over $\stackT_{\Sigma}$ via the dual morphism of the exact sequence. Then we have the following definition:

\begin{definition}
    The toric stack $X$ associated to the stacky fan $\mathbf{\Sigma}$ is defined over $\Z$ as the global quotient $[\stackT_{\Sigma} / G_{\beta} ]$. 
\end{definition}

\begin{remark}
    In \cite[definition 3.1]{fantechi_toric_stack}, a smooth toric Deligne-Mumford stack is introduced as a smooth separated Deligne-Mumford stack $\stackX$ together with an open immersion of the stacky torus $T$ $$\iota : T \hookrightarrow \stackX$$with dense image such that the action of $T$ on itself extends to an action $$a : T \times \stackX \rightarrow \stackX.$$The article \cite{fantechi_toric_stack} shows that it is equivalent to the characterization as a global quotient, using stacky fan.
\end{remark}

By \cite[proposition 3.2, 3.7]{borisov_toric_stack} and by \cite[section 7]{fantechi_toric_stack}:

\begin{prop}
    The toric stack $X$ is a Deligne-Mumford stack, its Picard group is isomorphic to $DG(\beta)$ and the map $\Div_T(X) \rightarrow \Pic(X)$ naturally identifies with the map $(\Z^{\Sigma(1)})^* \rightarrow \mathbf{R^1 Hom}(\coneform(\beta),\Z)$. From now on, we shall also denote by $T_{\NS}$ the multiplicative group $G_{\beta}$.

    If $X^{\coarse}$ denotes the proper toric variety defined by the fan $\Sigma$, then we have a natural coarse map $X \rightarrow X^{\coarse}$. 
\end{prop}

We now explain how to construct the universal torsor of the toric stack $X$ from the previous quotient description.

\begin{theorem}\label{theorem_cox_ring_universal_torsor}
The $T_{\NS}$-torsor given by the morphism
$$\stackT_{\Sigma} \rightarrow X$$
constructed above has type
$$- \mathrm{id}_{\Pic(X)}.$$
In other words, the universal torsor of $X$ is given by the quotient morphism
$$\stackT_{\Sigma} \rightarrow X = [\stackT_{\Sigma} / T_{\NS}],$$
where $T_{\NS}$ acts on $\stackT_{\Sigma}$ via the opposite action. That is, an element $t$ of $T_{\NS}$ acts on the coordinate associated with $\rho \in \Sigma(1)$ by multiplication by $[D_{\rho}](t)^{-1}$.
\end{theorem}

\begin{remark}
In the case of smooth, proper, and split toric varieties over $\Q$, the theorem is a special case of the construction of the universal torsor via the Cox ring. The result of \cite[Proposition~1.6.1.7]{ArzhantsevDerenthalHausenLaface2015} computes the type of the torsor obtained in this way. However, the final step of the proof appears to contain an inaccuracy. Since this point is relevant in our setting, we provide a complete argument here.
\end{remark}

\begin{proof}
Assume that $\stackT_{\Sigma}$ is endowed with the opposite action of $T_{\NS}$
as defined in the statement of the theorem. For $\rho \in \Sigma(1)$, by the
definition of the contracted product (see equation~\ref{equation_definition_contracted_product}) and the remark that follows,
we have
\[
L \;=\;
\stackT_{\Sigma} \overset{T_{\NS},[D_{\rho}]}{\times} \affine^1
\;=\;
\bigl[\, \stackT_{\Sigma} \times \affine^1 \,/\, T_{\NS} \bigr],
\]
where the action of $T_{\NS}$ is given by
\[
t \cdot (y,\lambda)
=
\left(
\bigl([D_{\rho}](t)^{-1} \cdot y_{\rho}\bigr)_{\rho \in \Sigma(1)},
\;
[D_{\rho}](t)^{-1} \cdot \lambda
\right).
\]

Observe that the map
\begin{align*}
\stackT_{\Sigma} &\longrightarrow \stackT_{\Sigma}\times \affine^1 \\
y &\longmapsto (y,y_{\rho})
\end{align*}
is $T_{\NS}$-equivariant and therefore induces by quotient a section of the line bundle $$L = \bigl[\, \stackT_{\Sigma} \times \affine^1 \,/\, T_{\NS} \bigr] $$over $X = [\stackT_{\Sigma} / T_{\NS} ]$ given by $y_{\rho}$. Consequently,
\[
\mathrm{type}([\stackT_{\Sigma}])([D_{\rho}])
=
[\mathrm{div}(y_{\rho})]
=
[D_{\rho}] \, .
\]

This proves the claim when $N$ is torsion free or $\Pic(X)$ is torsion free. We conclude for the general case using the description of toric stack over their rigidification (see \cite[Corollary 6.26]{fantechi_toric_stack}) and by lemma \ref{lemma_root_stack_type_torseur}.

\end{proof}

\begin{lemma}\label{lemma_root_stack_type_torseur}
Let $X = [\stackT / T_{\NS}]$ be a toric stack such that
\[
\mathrm{type}([\stackT]) = \mathrm{id}.
\]
Let $k \in \N^*$, let $\mathrm L = (L_1,\dots,L_k) \in \Pic(X)^k$, and let
$n = (n_1,\dots,n_k) \in \N^{*\,k}$. Let $Y$ be the toric stack defined by
\[
Y = \sqrt[n]{\mathrm L / X}.
\]
Denote by $\widetilde{T_{\NS}}$ the group of multiplicative type defined by the
following cartesian diagram:
\begin{center}
\begin{tikzcd}
\widetilde{T_{\NS}} \arrow[r] \arrow[d] & \G_m^k \arrow[d, "{\wedge n}"] \\
T_{\NS} \arrow[r, "{[\mathrm L]}"] & \G_m^k
\end{tikzcd}
\end{center}
Then $\widetilde{T_{\NS}} = T_{\NS,Y}$ is the Néron--Severi torus of $Y$.
Moreover, the upper horizontal arrow is given by the characters
\[
[\mathrm{L}^{\frac{1}{n}}]
=
\left(
[L_1^{\frac{1}{n_1}}],\dots,[L_k^{\frac{1}{n_k}}]
\right)
\in X^*(\widetilde{T_{\NS}})^k = \Pic(Y)^k,
\]
and the left vertical arrow is induced, by duality, from the pullback
\[
\Pic(X) \longrightarrow \Pic(Y).
\]
Finally, we have the quotient presentation
\[
Y = [\stackT / \widetilde{T_{\NS}}],
\]
and the torsor given by the associated quotient morphism
\[
\stackT \longrightarrow Y
\]
is still universal.
\end{lemma}
\begin{proof}
All assertions in the statement except the last one are proved in \cite[Lemma 7.1 (2)]{fantechi_toric_stack}. To prove the last assertion, consider the following 2-commutative diagram
\begin{center}
\begin{tikzcd}
{Y = \sqrt[n]{\mathrm{L} / X}} \arrow[d] \arrow[r] & B \widetilde{T_{\NS}} \arrow[r, "{[ \mathrm{L}^{\frac{1}{n}} ]}"] \arrow[d] & B \G_m^k \arrow[d, "\wedge n"] \\
X \arrow[r]                               & B T_{\NS} \arrow[r, "{[\mathrm L]}"]                                       & B \G_m^k                  
\end{tikzcd}    
\end{center}

We now consider the commutative diagram obtained by pull-backs on Picard groups, using moreover Proposition~\ref{proposition_computation_type_torseur} and our hypothesis:
\begin{center}
\begin{tikzcd}
{\Pic\left( \sqrt[n]{\mathrm{L} / X} \right)} & {\Pic\left( \sqrt[n]{\mathrm{L}  / X} \right)} \arrow[l, "{\mathrm{type}([\stackT])}"'] & \Z^k \arrow[l, "{[ \mathrm{L} ^{\frac{1}{n}} ]}"']     \\
\Pic(X) \arrow[u]                    & \Pic(X) \arrow[u] \arrow[l, "\mathrm{id}_{\Pic(X)}"']                 & \Z^k \arrow[u, "{\mathrm{diag}(n_1,\dots,n_k)}"'] \arrow[l, "{[\mathrm{L} ]}"']
\end{tikzcd}
\end{center}By \cite[Remark~7.4]{fantechi_toric_stack}, the right square is cocartesian. 
If for every $1 \leqslant i \leqslant k$ we have
\[
\mathrm{type}([\stackT])\!\left([L_i^{\frac{1}{n_i}}]\right)
= [L_i^{\frac{1}{n_i}}],
\]
then, again by \cite[Remark~7.4]{fantechi_toric_stack}, the outer rectangle is cocartesian. 
It follows that the left square is cocartesian, and therefore
\[
\mathrm{type}([\stackT]) = \mathrm{id}_{\Pic(Y)}.
\]

For $i \in \{1,..,k\}$, to prove that $\mathrm{type}([\stackT])\left([L_i^{\frac{1}{n_i}}]\right) = [L_i^{\frac{1}{n_i}}]$, it suffices to exhibit a $\G_m$-equivariant morphism over $Y$,
$$
[\stackT \times \affine^1 / \widetilde{T_{\NS}}] = \stackT \overset{\widetilde{T_{\NS}},[L_i^{\frac{1}{n_i}}]} {\times} \affine^1 \longrightarrow L_i^{\frac{1}{n_i}} .
$$
By \cite[Appendix B]{abramovich2008gromovwittentheorydelignemumfordstacks}, the line bundle
$$
L_i^{\frac{1}{n_i}} \rightarrow Y = \prod_{j=1}^k \sqrt[n_j]{L_j/X}
$$
is equal to the morphism of quotient stacks
$$
 [L_i^{\times} \times \affine^1 / \G_m] \underset{X}{\times} \prod_{j \neq i} \sqrt[n_j]{L_j/X} \rightarrow [L_i^{\times} / \G_m ] \underset{X}{\times} \prod_{j \neq i} \sqrt[n_j]{L_j/X} = Y,
$$
where the symbol $\prod$ denotes the fiber product over $X$, $\G_m$ acts on $L_i^{\times}$ with weight $n_i$ and on $\affine^1$ with weight $- 1$.

Under our hypothesis,
$$
L_i^{\times} = \stackT \overset{T_{\NS},[L_i]} {\times} \G_m
= [\stackT \times \G_m / T_{\NS} ] .
$$
Consider the morphism
\begin{align*}
   &\stackT \times \affine^1 \rightarrow L_i^{\times} \times \affine^1 \\
   &( y , \lambda ) \longmapsto \left( [(y,1)],\lambda \right) \,.
\end{align*}
It is invariant under the morphism $t \in \widetilde{T_{\NS}} \mapsto [L_i^{\frac{1}{n}}](t) \in \G_m$, where $\G_m$ acts on $L_i^{\times} \times \affine^1$ as defined above. Thus we obtain the desired morphism
$$
[\stackT \times \affine^1 / \widetilde{T_{\NS}}]
= \stackT \overset{\widetilde{T_{\NS}},[L_i^{\frac{1}{n_i}}]} {\times} \affine^1
\longrightarrow L_i^{\frac{1}{n_i}},
$$
which concludes the proof.
\end{proof}

\subsection{Rigidification and canonical stack}

We refer to \cite[section 1.4, section 4.1]{fantechi_toric_stack} for the notion of (universal) rigidification of $X$ and the universal canonical stack associated to $X$. Let us just say that the "rigidification" $r : X \rightarrow X^{\rig}$ makes in a certain way $X^{\rig}$ the universal orbifold associated to $X$ and $r$ makes $X$ a gerbe over $X^{\rig}$ along the generic stabilizer of $X$. While $X^{\can}$ is the "universal" orbifold resolution of the coarse space $X^{\coarse}$. In particular, if $X^{\coarse}$ is smooth, we have $X^{\can} = X^{\coarse}$.

\begin{Notations}\label{notation_stacky_fan_beta}
    For $\rho \in \Sigma(1)$, we write $b_ {\rho} = \beta(e_{\rho})$ and $\overline{N}$ for  the lattice $N / N_{\tor}$, which is naturally embedded in $N_{\Q}$. We can see $\Sigma$ as a fan associated to this lattice. Let us denote by $u_{\rho} \in N_{\Q}$ the minimal generator of the edge $\rho \in \Sigma(1)$. Let $a_{\rho} \in \N^*$ be such that $\overline{\beta}(e_{\rho}) = \overline{b_{\rho}} = a_{\rho} . u_{\rho}$. We denote by $$\beta^{\rig} : \Z^{\Sigma(1)} \rightarrow \overline{N}$$ the endomorphism which maps $e_{\rho}$ to $a_{\rho} . u_{\rho}$ and by $$\beta^{\can} : \Z^{\Sigma(1)} \rightarrow \overline{N}$$ the one which maps $e_{\rho}$ to $ u_{\rho}$.
\end{Notations}

With this notation, we get the following theorem from \cite[Paragraph 7.2]{fantechi_toric_stack}:

\begin{theorem}
    The stacky fan $\mathbf{\Sigma^{\rig}} = (\Sigma,\beta^{\rig})$ corresponds to the toric orbifold $X^{\rig}$ obtained by rigidification along the generic stabilizer of $X$. The stacky fan $\mathbf{\Sigma^{\can}} = (\Sigma,\beta^{\can})$ corresponds to the canonical stack associated to $X^{\coarse}$. Moreover the morphisms $$X \rightarrow X^{\rig} \rightarrow X^{\can} $$ come from the natural morphisms of stacky fans.
\end{theorem}

Now let us give a quotient description of the morphism $X \rightarrow X^{\rig}$ in the case of toric stacks:

\begin{cor}\label{cor_description_map_stack_to_rigidification}
    The natural morphism of stacky fans $(\Sigma,\beta) \rightarrow (\Sigma,\beta^{\rig})$ gives the following exact sequence:
    $$0 \rightarrow \Pic(X^{\rig}) \rightarrow \Pic(X) \rightarrow \Ext^1(N_{\tor},\Z) \rightarrow 0 .$$Hence if $G$ denotes the multiplicative group associated to $\Ext^1(N_{\tor},\Z)$, $G$ identifies with the kernel of the natural map $T_{\NS,X} \rightarrow T_{\NS,X^{\rig}}$. The universal morphism $$X \xrightarrow{p_{\rig}} X^{\rig}$$ is obtained by quotient from the identity morphism $\id_{\stackT}$ which is equivariant for the group morphism $T_{\NS,X} \rightarrow T_{\NS,X^{\rig}}$.
\end{cor}

Now we may describe the natural map $X \rightarrow X^{\can}$ when $X$ is a toric orbifold (that is to say the map $X^{\rig} \rightarrow X^{\can}$ in the general case).

\begin{cor}\label{cor_description_map_stack_to_canonical}
Suppose that $X$ is a toric orbifold. The natural map $X \xrightarrow{p_{\can}} X^{\can}$ is obtained by quotient from the map 
\begin{align*}
\Phi_a : \  &\stackT \rightarrow \stackT \\
&y \longmapsto (y_{\rho}^{a_{\rho}}) . 
\end{align*}which is $T_{\NS,X} \rightarrow T_{\NS,X^{\can}}$-equivariant. Moreover if $\Cl(X^{\coarse})$ is the class group of $X^{\coarse}$, $\Pic(X^{\can}) \simeq \Cl(X^{\coarse})$.
\end{cor}

\begin{remark}\label{remark_picard_group_canonical_stack}
    By \cite[Theorem 4.1.3]{CoxLittleSchenck2011}, this means that the Picard group of the canonical stack $X^{\can}$ of a toric stack $X$ is always torsion free.
\end{remark}

\subsection{Sectors of toric stacks}

By \cite[proposition 4.3, remark 4.4]{borisov_toric_stack}, we can give an open covering of the toric stack $X$ defined via the stacky fan $$\mathbf{\Sigma} = (\Sigma,\Z^{\Sigma(1)} \xrightarrow{\beta} N).$$Let us set $d = \dim_{\Q} N_{\Q}$. For $\sigma \in \Sigma_{\max}$, we denote by $N(\sigma) = N / N_{\sigma}$ where $N_{\sigma}$ is the subgroup generated by $\{ b_{\rho} \mid \rho \in \sigma(1) \}$. We denote by $D(X^{\tilde{\sigma}})$ the open subscheme of $\stackT_{\Sigma}$ obtained by cutting out $V(X^{\tilde{\sigma}}) = \bigcup\limits_{\rho \not\in \sigma(1)} \left( X_{\rho} = 0 \right) $. Note that in this subsection, $$\stackT_{\Sigma} \rightarrow X$$still denotes the classical quotient presentation of the toric stack. By Theorem~\ref{theorem_cox_ring_universal_torsor}, we showed that it corresponds to the $T_{\NS}$-torsor of type $-\mathrm{id}_{\Pic(X)}$.

We also define by $\G_{\sigma}$ the multiplicative group associated to 
\begin{equation}\label{equation_identification_ext_1}
    \Ext^1(N(\sigma),\Z) = \Hom_{\Q/\Z}(N(\sigma),\Q/\Z).
\end{equation}Applying \cite[proposition 4.3]{borisov_toric_stack}, we get:

\begin{theorem}
    The stack $U_{\sigma} = [\affine^d / \G_{\sigma} ]$ is an open substack of $X$ such that we have the following cartesian and commutative diagram:
    \begin{center}
\begin{tikzcd}
D(X^{\tilde{\sigma}}) \arrow[r] \arrow[d] & \stackT_{\Sigma} \arrow[d] \\
U_{\sigma} \arrow[r]                      & X                         
\end{tikzcd}
    \end{center}
\end{theorem}

\begin{remark}
    The stack $U_{\sigma}$ has a description as a toric stack with stacky fan $$\mathbf{\sigma} = (\sigma, \beta_{\sigma} : \Z^{\sigma(1)} \rightarrow N) $$ where $\beta_{\sigma}$ is defined such that if $i : \Z^{\sigma(1)} \hookrightarrow \Z^{\Sigma(1)}$ is the natural injection, $\beta_{\sigma} = \beta \circ i$. With this notation, note that $\mathbf{R^1 Hom}(\coneform(\beta_{\sigma}),\Z) = \Ext^1(N(\sigma),\Z)$.
\end{remark}

Now for $\sigma \in \Sigma$, we define $\Boxe(\sigma)$ as follows:
\begin{equation}
    \Boxe(\sigma) = \{ n \in N \mid \exists \ (q_{\rho})_{\rho \in \sigma(1)} \in \left( [0,1[ \cap \Q \right)^{\sigma(1)}, \overline{n} = \sum\limits_{\rho \in \Sigma(1)} q_{\rho} . \overline{b_{\rho}} \in \overline{N} \}
\end{equation}We set $$\Boxe(\Sigma) = \bigcup\limits_{\sigma \in \Sigma_{\max}} \Boxe(\sigma) .$$Let us explain now how this set describe the sectors of $X$.

\begin{theorem}\label{theorem_description_sector_box_element}
    The set of sectors of $X$, $\sector$, is in one and one correspondence with $\Boxe(\Sigma)$. This correspondence is defined as follows: for $b \in \Boxe(\sigma)$ its class in $$N(\sigma) = \Hom_{\grp}(\Ext^1(N(\sigma),\Z),\Q/\Z)$$defines an injective morphism $$\mu_{r_b} \hookrightarrow \G_{\sigma} .$$The sector $\stackY_b$ associated to $b$ corresponds to the connected component of $\sheafI_{\mu} X$ which contains the morphism $$B \mu_{r_b} \rightarrow U_{\sigma} \rightarrow X$$associated to the pair $(0_{\affine^d},\mu_{r_b} \hookrightarrow \G_{\sigma})$.
\end{theorem}

\begin{remark}
The isomorphism
\[
N(\sigma) \simeq \Hom_{\grp}\!\left(\Ext^1(N(\sigma),\Z),\Q/\Z\right)
\]
is the one induced by the short exact sequence
\[
0 \longrightarrow \Z \longrightarrow \Q \longrightarrow \Q/\Z \longrightarrow 0,
\]
to which we apply the functor $\mathbf{RHom}(N(\sigma),-)$.
\end{remark}

\begin{proof}
    It is a consequence of \cite[lemma 4.6]{borisov_toric_stack} and \cite[lemma 3.1.4]{darda_yasuda_toric_stacks_batyrev}.
\end{proof}

\begin{Notations}\label{notation_cone_minimal_sector}
For a sector $\stackS \in \sector$ and the corresponding element
$$b \in \Boxe(\Sigma),$$we denote by either $r_{\stackS}$ or $r_b$ the associated exponent and  we denote by $\sigma_{\stackS}$ or $\sigma_b$ the minimal
cone of $\Sigma$ containing $\overline{b} \in N_{\Q}$.  Note that $\sigma_0 = \{0\}$.
\end{Notations}

Recall that the toric morphism of toric stacks $U_{\sigma} \rightarrow X$ is given by the natural morphism of stacky fan $$(\sigma, \beta_{\sigma}) \rightarrow (\Sigma,\beta) .$$We have the following morphism of exact triangles:
\begin{equation}\label{equation_morphism_triangle_distingue_associe_open_immersion}
\begin{tikzcd}
\Z^{\sigma(1)} \arrow[r, "\beta_{\sigma}"] \arrow[d, "i"] & N \arrow[r] \arrow[d, equal] & \coneform(\beta_{\sigma}) \arrow[d] \\
\Z^{\Sigma(1)} \arrow[r, "\beta"']                        & N \arrow[r]                       & \coneform(\beta)                   
\end{tikzcd}
\end{equation}which gives by duality via $\RHom(-,\Z)$ a natural map $$ \Pic(X) \xrightarrow{p_{\sigma}} \Ext^1(N(\sigma),\Z) .$$Again this map can be interpreted by duality of groups of multiplicative type as the map $\G_{\sigma} \rightarrow T_{\NS}$. We can deduce the following expression of the age pairing for toric stacks:

\begin{prop}\label{prop_expression_age_sector_toric_stack}
    Let $b \in \Boxe(\sigma)$. Denote by $$\phi_b : \Ext^1(N(\sigma),\Z) \rightarrow \frac{1}{r_b} \Z / \Z \subset \Q / \Z $$ the associated morphism. If $\stackS_b$ is the corresponding sector then the composition $$ \Pic(X) \xrightarrow{p_{\sigma}} \Ext^1(N(\sigma),\Z) \xrightarrow{\phi_b}  \frac{1}{r_b} \Z / \Z$$ is the map $$- \age(\stackS_b,-) : L \in \Pic(X) \longmapsto -\age(\stackS_b,L) .$$
\end{prop}

\begin{proof}
Recall that to an element $b \in \Boxe(\Sigma)$ we associate its class in
$$
N(\sigma) = \Hom_{\grp}(\Ext^1(N(\sigma),\Z),\Q/\Z)
$$
which defines an injective morphism
$$
\mu_{r_b} \hookrightarrow \G_{\sigma}.
$$
By Theorem~\ref{theorem_description_sector_box_element}, this corresponds to a unique sector $\stackS \in \sector$, given bijectively as the connected component of the element $b_{\stackS}$
$$
(0_{\affine^d},\mu_{r_b} \hookrightarrow \G_{\sigma}).
$$Moreover, following the statement of this theorem, we described the dual of the morphism
$$ \G_{\sigma} \rightarrow T_{\NS}. $$We have the following $2$-commutative diagram:
\begin{center}
\begin{tikzcd}
B \mu_{r_b} \arrow[r] \arrow[d, equal] \arrow[rr, "b_{\stackS}", bend left] & U_{\sigma} \arrow[r] \arrow[d] & X \arrow[d] \\
B \mu_{r_b} \arrow[r]                                                            & B G_{\sigma} \arrow[r]         & B T_{\NS}  \, .
\end{tikzcd}
\end{center}
By considering the pull-back on Picard groups and applying Proposition~\ref{prop_computation_general_age_pairing_via_type_torseur},
which relates the age pairing to the type of a torsor, together with
Theorem~\ref{theorem_cox_ring_universal_torsor}, which determines the type of
$\stackT_{\Sigma} \rightarrow X$, we can conclude the proof.
\end{proof}

\begin{remark}
    The map $p_{\sigma} : \Pic(X) \rightarrow \Ext^1(N(\sigma),Z)$ is surjective. Indeed consider the exact triangle $$\coneform(\beta_{\sigma}) \rightarrow \coneform(\beta) \rightarrow W .$$ We have $\mathbf{R^2 Hom}(W,\Z) = 0$ because $W$ is an element of $\mathcal{D}_{[0,1]}(\Z)$ with $H^0(W)$ and $H^1(W)$ abelian groups of finite type.
\end{remark}

\begin{prop}\label{prop_age_divisor_associated_to_an_edge}
For $\sigma \in \Sigma_{\max}$ and for any $b \in \Boxe(\sigma) \subset \sector$ such that $$\overline{b} = \sum\limits_{\rho \in \sigma(1)} q_{\rho} . \overline{b_{\rho}} \in N_{\Q},$$ we have for any $\rho \in \sigma(1)$:
$$\age_{\num}(\stackY_b,-[D_{\rho}]) = q_{\rho} .$$Moreover, if $\rho \not\in \sigma(1)$, $\age(\stackY_b,[D_{\rho}]) = 0 \mod \Z$.
\end{prop}

\begin{proof}
    Applying the functor $\RHom(-,\Z)$ to the morphism of exact triangles from the diagram \ref{equation_morphism_triangle_distingue_associe_open_immersion}, we get the following commutative diagram:
    \begin{center}
\begin{tikzcd}
\Z^{\Sigma(1)} \arrow[r] \arrow[d] & \Pic(X) \arrow[d, "p_{\sigma}"] \\
\Z^{\sigma(1)} \arrow[r]           & {\Ext^1(N(\sigma),\Z) .}         
\end{tikzcd} 
    \end{center}The top horizontal map associates to a vector of the canonical basis $e_{\rho}$ the class $[D_{\rho}]$. If $\Ext^1(N(\sigma),\Z)$ is identified with $\Hom_{\grp}(N(\sigma),\Q / \Z)$ seen as a subgroup of $$\text{Map}(\Boxe(\sigma),\Q / \Z) ,$$ the whole map is $$b \in \Boxe(\sigma) \longmapsto - \age(\stackY_b,[D_{\rho}]) $$by proposition \ref{prop_expression_age_sector_toric_stack}. For $\rho \not\in \sigma(1)$, $e_{\rho}$ is sent to $0$ via the left vertical map. Hence  $\age(\stackY_b,[D_{\rho}]) = 0 \mod \Z$ for $\rho \not\in \sigma(1)$.

    Now to conclude the proof for $\rho \in \sigma(1)$, it is enough to see that the bottom horizontal map sends a vector of the canonical basis $e'_{\rho}$ to the morphism $$\overline n \in N(\sigma) \longmapsto \langle b_{\rho}^* , n \rangle \mod \Z $$where $(b_{\rho}^*)_{\rho \in \sigma(1)}$ is the dual basis of the basis $(\overline{b_{\rho}})_{\rho \in \sigma(1)}$ of $N_{\Q}$ and $n \in N$ is any lift of $\overline n \in N(\sigma)$. This is proved in Proposition \ref{proposition_diagramme_chasing_appendix}.
\end{proof}

\begin{theorem}\label{theorem_sectors_dual_picard_group_toric_stacks}
    For $\stackY_1, \stackY_2 \in \sector$, if
\[
\age(\stackY_1,-) = \age(\stackY_2,-),
\]
then $\stackY_1 = \stackY_2$. In other words, Hypothesis~\ref{hypothesis_sectors_dual_picard_group}
is satisfied by toric stacks.

\end{theorem}

\begin{proof}
    First by proposition \ref{prop_age_divisor_associated_to_an_edge}, we have that if $\stackY_1, \stackY_2 \in \sector$ verify $$\age(\stackY_1,-) = \age(\stackY_2,-)$$ then $\stackY_1, \stackY_2 \in \Boxe(\sigma) $ for a same $\sigma \in \Sigma_{\max}$. Then the sector $\stackY_i \in \sector$ corresponds bijectively to a morphism $\varphi_i : \Ext^1(N(\sigma),\Z) \rightarrow \Q / \Z$. By proposition \ref{prop_expression_age_sector_toric_stack}, the fact that $\age(\stackY_1,-) = \age(\stackY_2,-)$ means that $\varphi_1 \circ p_{\sigma} = \varphi_2 \circ p_{\sigma}$ where $p_{\sigma}$ is the map $ \Pic(X) \rightarrow \Ext^1(N(\sigma),Z)$ defined earlier. Because $p_{\sigma}$ is surjective, we get that $\varphi_1 = \varphi_2$.
\end{proof}

\subsection{The orbifold Picard group of toric stacks}

\subsubsection{Extended quotient description for toric stacks}

Let us consider $\mathcal{C} \subset \twistsector$. In \cite{Jiang2005TheOC}, Jiang gave a new quotient description of the toric stack $$X = [\stackT_{\Sigma} / T_{\NS} ].$$ This description is based on the introduction of an extended stacky fan $$\mathbf{\Sigma}_{\mathcal{C}} = \left(\Sigma, \beta_{\mathcal{C}} : \Z^{\Sigma(1)} \oplus \Z^{\mathcal{C} } \rightarrow N \right)$$such that for any $\stackS \in \mathcal{C}$, we set $$\beta_{\mathcal{C}}(e_{\stackS}) = b_{\stackS}$$ where $b_{\stackS} \in \Boxe(\Sigma)$ corresponds to $\stackS$ and $\beta_{\mathcal{C}}$ and $\beta$ coincides over $\Z^{\Sigma(1)}$. We have the following exact sequence:

$$ 0 \rightarrow N^* \xrightarrow{\beta_{\mathcal{C}}^*} (\Z^{\Sigma(1)} \oplus \Z^{\mathcal{C} })^* \xrightarrow{\pi} \mathbf{R^1 Hom}(\coneform(\beta_{\mathcal{C}}),\Z) \rightarrow \Ext^1(N,\Z) \rightarrow 0 .$$The map $\pi$ induces by duality an action of $G_{\mathcal{C}}$, the multiplicative group associated to $\mathbf{R^1 Hom}(\coneform(\beta_{\mathcal{C}}),\Z)$  on $\stackT_{\Sigma} \times \G_m^{\mathcal{C}}$.

\begin{Notations}\label{notation_orbifold_T_invariant_divisor}
To each $\rho \in \Sigma(1)$ we associate an orbifold line bundle $[D_{\rho}]_0$
defined by the pair
$$([D_\rho],\varphi_{\rho}),$$
where
$$\varphi_{\rho}(\stackS) = - \age_{\num}(\stackS,-[D_{\rho}]) \quad \text{for all } \stackS \in \twistsector.$$   
\end{Notations}
We can then state the following result:

\begin{theorem}\label{theorem_action_orbi_neron_severi_torus}
    We have an isomorphism:
    $$ \Pic_{\mathcal{C}}(X) \simeq \mathbf{R^1 Hom}(\coneform(\beta_{\mathcal{C}}),\Z).$$Moreover, the corresponding map $$ \Z^{\Sigma(1)} \oplus \Z^{\mathcal{C} } \xrightarrow{\pi} \Pic_{\mathcal{C}}(X)$$is obtained by mapping $e_{\rho}$ to $[D_{\rho}]_{0}$ for $\rho \in \Sigma(1)$ and $e_{\stackS}$ to $[\stackS]$ for $\stackS \in \mathcal{C}$.
\end{theorem}

\begin{remark}\label{remark_exact_sequence_dual_orbifold_picard_group}
    In particular, in the case where $\mathcal{C} = \twistsector$, we shall write $$\beta_{\twistsector} = \beta^{\ext} .$$Since $\beta^{\ext}$ is surjective, we obtain the exact sequence
\[
0 \longrightarrow \Pic_{\orb}(X)^{\vee} \xrightarrow{\pi^{\vee}} \Z^{\Sigma(1)} \oplus \Z^{\twistsector} \xrightarrow{\beta^{\ext}} N \longrightarrow 0 \, .
\]This also proves that $\Pic_{\orb}(X)$ is torsion free (see Corollary \ref{corollaire_orbifold_picard_group_torsion_free} for an other proof). 

\end{remark}

This theorem was stated by Coates, Corti, Iritani, and Tseng in \cite[Proposition 21]{Coates_Corti_Iritani_Tseng_miror_theorem_toric_stack}. Since our definition of the orbifold Picard group (see Definition \ref{definition_orbifold_picard_group}) differs from the one given in \cite[Definition 15]{Coates_Corti_Iritani_Tseng_miror_theorem_toric_stack}, we give a complete proof of the previous theorem in appendix (see Theorem \ref{theorem_computation_orbifold_picard_group_toric_stack_appendix}).

\begin{Notations}
    From now on, we shall denote by $T_{\NS,\mathcal{C}}$ the multiplicative group associated to $\Pic_{\mathcal{C}}(X).$ When $\mathcal{C} = \twistsector$, it shall be denoted $T_{\NS,\orb}$.
\end{Notations}

There is a natural map $\stackT \rightarrow \stackT \times \G_m^{\mathcal{C}}$ given by $y \mapsto (y,1_{\mathcal{C}})$, which is equivariant with respect to $T_{\NS} \rightarrow T_{\NS,\mathcal{C}} $, where $T_{\NS} \rightarrow T_{\NS,\mathcal{C}} $ is the group morphism dual to the natural map $\Pic_{\mathcal{C}}(X) \rightarrow \Pic(X)$. By \cite[proposition 2.3]{Jiang2005TheOC}, we have the following theorem:

\begin{theorem}\label{theorem_extended_universal_torsor}
    The map described earlier $\stackT \rightarrow \stackT \times \G_m^{\mathcal{C}}$ gives by quotient an isomorphism:
    $$ X = [\stackT_{\Sigma} \times \G_m^{\mathcal{C}} / T_{\NS,\mathcal{C}} ].$$
\end{theorem}

\begin{remark}\label{remark_computation_type_extended_universal_torsor}
From the construction of the isomorphism
\[
[\stackT_{\Sigma} / T_{\NS}]
\simeq
[\stackT_{\Sigma} \times \G_m^{\mathcal{C}} / T_{\NS,\mathcal{C}}]
\]
we get that the type of the $T_{\NS,\mathcal{C}}$-torsor given by the quotient map
\[
\stackT \times \G_m^{\mathcal{C}} \rightarrow X
\]
is given using lemma \ref{lemma_computation_type_morphism} and theorem \ref{theorem_cox_ring_universal_torsor} by:
\begin{align*}
    &\Pic_{\mathcal{C}}(X) \longrightarrow \Pic(X) \\
    &(L,\varphi) \longmapsto - L \, .
\end{align*}
\end{remark}

\begin{Notations}\label{notation_extended_torus_fan}
Throughout this text, we will mainly consider the case $\mathcal{C} = \twistsector$.  
For any field $\K$, we write $\stackT^{\ext}(\K)$ for $\stackT(\K) \times \G_m(\K)^{\twistsector}$, and we set
\[
T^{\ext} = \G_m^{\Sigma(1)} \times \G_m^{\twistsector}
\]
to be the open torus in $\stackT \times \G_m^{\twistsector}$. We shall write $\Sigma_{\ext}(1)$ the set $\Sigma(1) \cup \twistsector$.
\end{Notations}

We now define the extended universal torsor, which will be a key ingredient of this work.

\begin{definition}\label{definition_extended_universal_torsor}
The quotient morphism
\[
q^{\ext} : \stackT \times \G_m^{\twistsector} \longrightarrow X,
\]
arising from the quotient presentation of $X$, where $T_{\NS,\orb}$ acts on
$\stackT \times \G_m^{\twistsector}$ via the inverse of the action defined at the
beginning of this paragraph, defines a $T_{\NS,\orb}$-torsor. We call it the
\emph{extended universal torsor}.

For $\stackS \in \sector$ and $(L,\varphi) \in \Pic_{\orb}(X)$, we denote by
\[
\age\bigl(\stackS,(L,\varphi)\bigr)
\]
the age pairing defined via this torsor (see Definition~\ref{definition_general_age_pairing}).
\end{definition}

The extended universal torsor satisfies the following property.

\begin{prop}\label{proposition_functoriality_age}
The type of the extended universal torsor is given by the natural morphism
\[
(L,\varphi) \longmapsto L.
\]
Moreover, for any $\stackS \in \sector$ and any $(L,\varphi) \in \Pic_{\orb}(X)$, we have
\[
\age\bigl(\stackS,(L,\varphi)\bigr) = \age(\stackS,L).
\]
\end{prop}

\begin{proof}
The first assertion follows from Remark~\ref{remark_computation_type_extended_universal_torsor}
together with the definition of the extended universal torsor. The second statement
is a consequence of Proposition~\ref{prop_age_functoriality_torsor}.
\end{proof}

We conclude this paragraph by checking that Hypothesis~\ref{hypothesis_zero_of_local_degree} is indeed satisfied by
\[
q : \stackT \longrightarrow X,
\]
the universal torsor, and by
\[
q^{\ext} : \stackT \times \G_m^{\twistsector} \longrightarrow X,
\]
the extended universal torsor, in the case of toric stacks.

\begin{prop}\label{prop_extended_univ_torsor_zero_age}
    Let $Y \rightarrow X$ a $T$-torsor other the toric stack $X$, which is either:
    \begin{itemize}
        \item the $T_{\NS}$-torsor given by $Y = \stackT_{\Sigma}$ when $\Pic(X)$ is torsion free;
        \item the $T_{\NS,\orb}$-torsor given by $Y = \stackT_{\Sigma} \times \G_m^{\twistsector}$.
    \end{itemize}
     Then hypothesis \ref{hypothesis_zero_of_local_degree} is verified: for any finite field extension $\K' / \K$, we have:
    $$\stackT(\sheafO_{\K'}) = \{y \in \stackT(\K') \mid \age_{v_{\K'}}(y,-) = 0 \} . $$
\end{prop}

\begin{proof}
    We can assume without loss of generality that $\K' = \K$. For any $y \in Y(\sheafO)$, we have $\age_v(y,-) = 0$. Now let $y \in Y(\K)$ be such that $\age_v(y,-) = 0$. We know there exists a finite field extension $\Le/ \K$ of local fields and $t \in T(\Le)$ such that $t.y \in Y(\sheafO_w)$ where $w$ is the valuation of $\Le$. Then by definition, $$\age_v(y,-) = \frac{1}{e(w|v)} \log_{T,w}(t) .$$We get because $T$ is split, that $t \in T(\sheafO_w)$ hence $$y \in Y(\sheafO_w) \cap Y(\K) .$$With our hypothesis on $Y$, we have $Y(\sheafO_w) \cap Y(\K) = Y(\sheafO)$ which concludes the proof.
\end{proof}

\subsubsection{The cone of orbifold effective divisors of toric stacks}

Darda and Yasuda showed the following result on the cone of orbifold effective divisors of toric stacks (see \cite[Theorem 1.1]{darda2025orbifoldpseudoeffectiveconestoric}):

\begin{prop}\label{proposition_description_cone_orbifold_effective_divisor_toric_stack}
    The cone of  orbifold effective divisor of $X$ is the cone generated by the set of classes $$\left\{ [D_{\rho}]_{0},[\stackY] \mid \rho \in \Sigma(1), \stackY \in \twistsector \right\}.$$ Moreover, we have the following description for the orbifold anti-canonical line bundle:
$$\w_{X,\orb}^{-1} = \sum\limits_{\rho \in \Sigma(1)} [D_{\rho}]_{0} + \sum\limits_{\stackY \in \twistsector} [\stackY] .$$
\end{prop}

\begin{proof}
    To prove the first assertion, it is enough to see that the rational number $a_{\rho}(\stackY)$ in \cite[Theorem 1.1]{darda2025orbifoldpseudoeffectiveconestoric} corresponds to $\age_{\num}(\stackY,-[D_{\rho}])$. This is a consequence of Proposition \ref{prop_age_divisor_associated_to_an_edge}.  The second assertion also follows from Proposition \ref{prop_age_divisor_associated_to_an_edge} together with \cite[Lemma 3.3.7 (2)]{darda_yasuda_toric_stacks_batyrev}.
\end{proof}

\begin{Notations}\label{notation_class_orbifold_picard_group}
When there is no ambiguity, we shall write for $\rho \in \Sigma_{\ext}(1)$,  $[D_{\rho}]$ the element of $\Pic_{\orb}(X)$ such that $[D_{\rho}] = [D_{\rho}]_{0}$ if $\rho \in \Sigma(1)$ and $[D_{\rho}] = [\stackS] $ if $\rho = \stackS \in \twistsector$. In particular, with these notations, 
$$\w_{X,\orb}^{-1} = \sum\limits_{\rho \in \Sigma_{\ext}(1)} [D_{\rho}] .$$
\end{Notations}


\subsection{Computing the residue map in the case of toric stacks}

A fundamental question when studying rational points on a stack using the formalism
introduced by Darda and Yasuda in \cite{darda2024batyrevmanin} is how to compute, for a
rational point $P \in X(\K_v)$, the sector given by the residue map
$\psi_v(P) \in \sector$. In their work on toric stacks, Darda and Yasuda addressed this
problem using the logarithm map (see \cite[Lemma~3.1.6]{darda_yasuda_toric_stacks_batyrev}).
Our approach is entirely different, and we present it in this subsection. To this end,
we will need the following proposition.

\begin{prop}\label{corollary_sectors_dual_picard_group_toric_stacks}
For $\stackY_1, \stackY_2 \in \sector$, if for every $(L,\varphi) \in \Pic_{\orb}(X)$,
\[
\age(\stackY_1,(L,\varphi)) = \age(\stackY_2,(L,\varphi)),
\]
then $\stackY_1 = \stackY_2$. 

\end{prop}

\begin{proof}
This is a consequence of theorem
\ref{theorem_sectors_dual_picard_group_toric_stacks} and proposition~\ref{proposition_functoriality_age}.
\end{proof}

This proposition will be used repeatedly. Indeed, for
$y \in \stackT^{\ext}(\K_v)$, the element
\[
\age_v(y,-) \in \Pic_{\orb}(X)^{\vee}_{\Q}
\]
determines
\[
\age\bigl(\psi_v(P),-\bigr) \in \Hom_{\grp}\!\left(\Pic_{\orb}(X), \Q / \Z\right)
\]
by Theorem~\ref{theorem_age_extended_universal_torsor}. By
Proposition~\ref{corollary_sectors_dual_picard_group_toric_stacks}, this uniquely
determines $\psi_v(P)$.

Recall that we denote by
\[
\pi : \Z^{\Sigma_{\ext}(1)} \longrightarrow \Pic_{\orb}(X)
\]
the map appearing in Theorem~\ref{theorem_action_orbi_neron_severi_torus}. We shall use
its injective dual map $\pi^{\vee}$. We write
$\{e_{\rho}\}_{\rho \in \Sigma_{\ext}(1)}$ for the canonical basis of
\[
(\Z^{\Sigma_{\ext}(1)})^{*} = \Z^{\Sigma_{\ext}(1)}.
\]
The following lemma will be particularly useful.

\begin{lemma}\label{lemma_computation_residue_map}
Let $y \in \stackT^{\ext}(\K_v)$, set $P = q^{\ext}(y)$, and let
$\stackS \in \twistsector$. Assume that
\[
\pi^{\vee}(\age_v(y,-))
= e_{\stackS} -
\sum_{\rho \in \Sigma(1)}
\age_{\num}\!\left(\stackS,-[D_{\rho}]\right)\, e_{\rho} \, .
\]
Then
\[
\psi_v(P) = \stackS .
\]
\end{lemma}

\begin{proof}
Let $(L,\varphi) \in \Pic_{\orb}(X)$. The family of classes of $T$-invariant divisors
$\{[D_{\rho}]\}_{\rho \in \Sigma(1)}$ generates $\Pic(X)_{\Q}$, hence there exist
rational numbers $m_{\rho} \in \Q$ such that
\[
L = \sum_{\rho \in \Sigma(1)} m_{\rho}\,[D_{\rho}].
\]
It follows using notation \ref{notation_orbifold_T_invariant_divisor} that, in $\Pic_{\orb}(X)_{\Q}$,
\[
(L,\varphi)
=
\sum_{\rho \in \Sigma(1)} m_{\rho}\,[D_{\rho}]_{0}
+
\sum_{\stackZ \in \twistsector} m_{\stackZ}\,[\stackZ],
\]
where
\[
m_{\stackZ}
=
\sum_{\rho \in \Sigma(1)}
m_{\rho}\,
\age_{\num}\!\left(\stackZ,[-D_{\rho}]\right)
+
\varphi(\stackZ).
\]
Using $\Q$-linearity, we therefore obtain
\begin{equation}\label{equation_lemma_computation_residue_map}
\age_v\bigl(y,(L,\varphi)\bigr) = \varphi(\stackS).    
\end{equation}Passing to classes modulo $\Z$ and applying
Theorem~\ref{theorem_age_extended_universal_torsor}, we obtain
\[
\age\bigl(\psi_v(P),(L,\varphi)\bigr)
=
\age\bigl(\psi_v(P),L\bigr)
=
\age(\stackS,L),
\]
where the first equality follows from Proposition~\ref{proposition_functoriality_age}, and the third term comes from
Definition~\ref{definition_orbifold_picard_group}, since
\[
\varphi(\stackS) \equiv \age(\stackS,L) \mod \Z.
\]
Applying Theorem~\ref{theorem_sectors_dual_picard_group_toric_stacks}, we conclude
that $\psi_v(P) = \stackS$.

\end{proof}

\section{Integral parametrization of rational points of toric stacks }

For the remainder of this article, we work over $\Q$ for simplicity. Recall that for a
toric stack, Hypotheses~\ref{hypothesis_torsor_is_a_scheme},
\ref{hypothesis_zero_of_local_degree}, and
\ref{hypothesis_sectors_dual_picard_group} are satisfied. We equip the proper and $\Q$-factorial toric variety $X^{\coarse}$ with its natural
system of heights (see \cite[Definition~9.2]{salberger_torsor} and
\cite[Proposition~2.1.2]{Batyrev1990}). The choice of an adelic stacky data (see
Definition~\ref{definition_adelic_stacky_data}) is given, for any $(L,\varphi)$, by the
collection $\{\varphi_v\}_{v \in M_{\Q}^0}$ defined by
\[
\varphi_v(P) = \varphi\bigl(\psi_v(P)\bigr).
\]

With these conventions, for any $P \in X(\Q)$ and any $\stackS \in \twistsector$, we have
\[
H_{\stackS}(P)
=
\prod_{\psi_v(P) = \stackS} p_v .
\]


\subsection{Definition of the integral parametrization}

We now aim to define an integral parametrization $\stackT^{\ext}(\Z)$ of the rational points $X(\Q)$ of the toric stack, which is invariant under the action of $T_{\NS,\orb}(\Z)$, using the extended quotient description (see Definition \ref{definition_extended_universal_torsor}) given by
\[
q^{\ext} : \stackT^{\ext}(\Q) \longrightarrow X(\Q).
\]
The various examples studied by the author, together with the geometry of the stack, naturally lead to the following definition.

\begin{definition}\label{definition_orb_universal_torsor_gcd_cond}
    For $R$ a principal ring with fraction field $\K$, we define $\stackT^{\ext}(R)$ the set of $(y',k) \in \stackT^{\ext}(\K) \cap R^{\Sigma_{\ext}(1)}$ such that:
    \begin{itemize}
        \item for any $\stackS \in \twistsector $, $k_{\stackS}$ is square free,
        \item for $\stackS \neq \stackS'$, $\gcd(k_{\stackS},k_{\stackS'}) = 1$,
        \item $y' \in \stackT(R)$,
        \item for any $\stackS \in \twistsector $ and for any primitive collection $C$ (see definition \ref{def_primitive_collection}) which contains edges of $\sigma(\stackS)$, we have:
        $$\gcd\left(k_{\stackS},y'_{\rho} \mid \rho \in C - \sigma_{\stackS}(1) \right) = 1 .$$
    \end{itemize}
\end{definition}

\begin{remark}
The definition of the previous integral parametrization is explained by the following geometric interpretation. For $b \in \Boxe(\Sigma)$, to which we associate the corresponding sector $\stackS_b \in \sector$, following \cite[Section~4]{borisov_toric_stack}, one can
associate a closed substack $$X\bigl(\Sigma / \sigma_b\bigr)$$of the toric stack
$X$ isomorphic to $\stackS_b \subset \I_{\mu} X$. Geometrically, this substack is given by
\[
X\bigl(\Sigma / \sigma_b\bigr)
= \bigcap_{\rho \in \sigma_b} D_{\rho}.
\]
In particular, for any primitive collection $C$, one has
\[
X\bigl(\Sigma / \sigma_b\bigr)
\cap
\bigcap_{\rho \in C \setminus \sigma_b} D_{\rho}
= \varnothing.
\]
\end{remark}

\begin{remark}
The set $\stackT^{\ext}(R)$ should not be confused with the set of $R$-points of the quasi-affine scheme $\stackT \times \G_m^{\twistsector}$. Indeed, for $R$ a principal ring, we have a strict inclusion
\[
\stackT(R) \times \G_m(R)^{\twistsector} \subsetneq \stackT^{\ext}(R).
\]
\end{remark}

We first show that the above subset is invariant under the action of $T_{\NS,\orb}$.

\begin{prop}\label{proposition_invariance_ext_param_action_t}
Let $R$ be a principal ring. For any $y \in \stackT^{\ext}(R)$ and any $t \in T_{\NS,\orb}(R)$, we have
\[
t \cdot y \in \stackT^{\ext}(R).
\]
\end{prop}

\begin{proof}
An element $t \in T_{\NS,\orb}(R)$ acts on the coordinate of $y$ associated with $\rho \in \Sigma_{\ext}(1)$ by multiplication by
\[
[D_{\rho}](t)^{-1} \in R^{\times},
\]
which proves the claim.
\end{proof}

\subsection{Local description}


We fix a finite place $v \in M^0_{\Q}$. In this paragraph, we show that $\stackT^{\ext}(\Z_v)$ parametrizes $X(\Q_v)$. We begin by introducing some notations:

\begin{Notations}
For any finite extension $\Le/\Q_v$, with valuation $v_{\Le}$, we also denote by
$v_{\Le}$, or simply by $v$ when $\Le=\Q_v$, the map
\[
\stackT^{\ext}(\Le) \longrightarrow \left(\Z \cup \{+\infty\}\right)^{\Sigma_{\ext}(1)},
\qquad
z \longmapsto \bigl(v_{\Le}(z_{\rho})\bigr)_{\rho \in \Sigma_{\ext}(1)} .
\]We also write
\[
\pi : \Z^{\Sigma_{\ext}(1)} \longrightarrow \Pic_{\orb}(X)
\]
the map appearing in Theorem~\ref{theorem_action_orbi_neron_severi_torus}, and we
write $\{e_{\rho}\}_{\rho \in \Sigma_{\ext}(1)}$ for the canonical basis of
$(\Z^{\Sigma_{\ext}(1)})^* = \Z^{\Sigma_{\ext}(1)}$.
We will simplify our notations by writing $\log_{v_{\Le}}$ instead of $\log_{T_{\NS,\orb},v_{\Le}}$ and $\age_v$ instead of $\age_{v,T_{\NS,\orb}}$.
Note that for any $t \in T_{\NS,\orb}(\Le)$ and any $y \in \stackT^{\ext}(\Le)$, we have
\begin{equation}\label{equation_formula_equivariance_valuation}
v_{\Le}(t \cdot y) = v_{\Le}(y) - \pi^{\vee}\left(\log_{v_{\Le}}(t)\right)    
\end{equation}by Definition \ref{definition_extended_universal_torsor}.
\end{Notations}

We first show that the computation of the residue map is simplified when working with the integral parametrization.

\begin{prop}\label{proposition_computation_residue_map_extended_parametrisation}
Let $y \in \stackT^{\ext}(\Z_v)$ and set $P = q^{\ext}(y)$. Then, for $\stackS \in \twistsector$, we have the equivalence
\[
p_v \mid y_{\stackS}
\;\Longleftrightarrow\;
\psi_v(P) = \stackS \, .
\]
In particular, for $y \in \stackT^{\ext}(\Z_v)$ such that $p_v \mid y_{\stackS}$, we have
\[
\pi^{\vee}\!\left( \age_v(y,-) \right)
= e_{\stackS}
-
\sum_{\rho \in \sigma_\stackS(1)}
\age_{\num}\!\left(\stackS,-[D_{\rho}]\right)\, e_{\rho} \, .
\]
\end{prop}

\begin{remark}
It follows that for $y \in \stackT^{\ext}(\Z_v)$, if for every $\stackS \in \twistsector$ we have
$p_v \nmid y_{\stackS}$, then $\psi_v(P) = 0$, and conversely. Note that in this case we also have $\age_v(y,-) = 0$.
\end{remark}

\begin{proof}
If for every $\stackS \in \twistsector$ we have $p_v \nmid y_{\stackS}$, then
\[
y \in \stackT(\Z_v) \times \G_m(\Z_v)^{\twistsector},
\]
and therefore $\psi_v(P) = 0$. To deduce the equivalence, it remains to show that
if $p_v \mid y_{\stackS}$, then $\psi_v(P) = \stackS$. We denote by $b_{\stackS}$ the element of $\Boxe(\Sigma)$ corresponding to the twisted sector $\stackS$. Recall that by proposition \ref{prop_age_divisor_associated_to_an_edge}, we have the equality
\[
\beta_{\Q}^{\ext}(e_{\stackS}) = \overline b_{\stackS}
=
\sum_{\rho \in \sigma_\stackS(1)}
\age_{\num}\!\left(\stackS,-[D_{\rho}]\right)\, \beta^{\ext}(e_{\rho})
\]
in $N_{\Q}$. There exists an integer $d \in \N^{*}$ such that
\[
d \cdot \sum_{\rho \in \sigma_\stackS(1)}
\age_{\num}\!\left(\stackS,-[D_{\rho}]\right)\, e_{\rho}
\in \Z^{\Sigma_{\ext}(1)}
\]
and since the torsion subgroup of $N$ is finite, such that
\[
\beta^{\ext}\!\left( d \cdot e_{\stackS}
-
\sum_{\rho \in \sigma_\stackS(1)}
d \cdot \age_{\num}\!\left(\stackS,-[D_{\rho}]\right)\, e_{\rho}
\right)
= 0 
\]in $N$. We may therefore consider an element $\alpha \in \Pic_{\orb}(X)^{\vee}$ such that
\[
\pi^{\vee}(\alpha)
= 
d \cdot e_{\stackS} -
\sum_{\rho \in \sigma_\stackS(1)}
d \cdot \age_{\num}\!\left(\stackS,-[D_{\rho}]\right)\, e_{\rho} \, .
\]
Let $\Le_d$ be a rupture field of the irreducible polynomial $X^{d}- p_v$ (it is a totally ramified extension of $\Q_v$ of degree $d$), and let
$t \in T_{\NS,\orb}(\Le_d)$ be such that
\[
\log_{v_{\Le_d}}(t) = \alpha .
\]
We set $z = t \cdot y \in \stackT^{\ext}(\Le_d)$. We have
\[
v_{\Le_d}(z)
=
d \cdot v(y)
-
d \cdot e_{\stackS}
+
\sum_{\rho \in \sigma_\stackS(1)}
d \cdot \age_{\num}\!\left(\stackS,-[D_{\rho}]\right)\, e_{\rho}.
\]
Thus, for any primitive collection $C$ of $\Sigma(1)$, we obtain
\[
\min_{\rho \in C}\bigl(v_{\Le_d}(z_{\rho})\bigr)
=
d \cdot
\left(
\min_{\rho \in C}
\left(
v(y_{\rho}) + \age_{\num}\!\left(\stackS,-[D_{\rho}]\right)
\right)
\right).
\]
Since for every $\rho \in \sigma_\stackS(1)$ we have
$\age_{\num}\!\left(\stackS,-[D_{\rho}]\right) > 0$, and for
$\rho \notin \sigma_\stackS(1)$ we have
$\age_{\num}\!\left(\stackS,-[D_{\rho}]\right) = 0$, it follows that
\[
\min_{\rho \in C}\bigl(v_{\Le_d}(z_{\rho})\bigr)
=
d \cdot
\left(
\min_{\rho \in C \setminus \sigma_\stackS(1)}
v(y_{\rho})
\right)
= 0,
\]
where the last equality follows from the definition of
$\stackT^{\ext}(\Z_v)$. Moreover, for every $\stackS' \in \twistsector$, we have
\[
v_{\Le}\bigl(z_{\stackS'}\bigr) = 0,
\]
and therefore
\[
z \in \stackT(\sheafO_{\Le}) \times \G_m(\sheafO_{\Le})^{\twistsector}.
\]We get that, by Definition~\ref{def_age_univ_torsor},
\[
\pi^{\vee}\!\left( \age_v(y,-) \right)
=
e_{\stackS} -
\sum_{\rho \in \sigma_\stackS(1)}
\age_{\num}\!\left(\stackS,-[D_{\rho}]\right)\, e_{\rho} \, .
\]Hence applying lemma \ref{lemma_computation_residue_map}, we can conclude the proof, $\psi_v(P) = \stackS$.

\end{proof}

For $\stackS \in \sector$, we define
\begin{equation}\label{equation_definition_notation_t_stacks}
    \stackT_{\stackS}(\Z_v)=
    \begin{cases}
        \stackT(\Z_v)\times \G_m(\Z_v)^{\twistsector} & \text{if } \stackS=0, \\[10pt]
        \left\{y\in\stackT^{\ext}(\Z_v)\mid p_v\mid y_{\stackS}\right\}
        & \text{if } \stackS\in\twistsector .
    \end{cases}
\end{equation}
We remark that
\[
\stackT^{\ext}(\Z_v)=\bigsqcup_{\stackS\in\sector}\stackT_{\stackS}(\Z_v) 
\]and by Proposition~\ref{proposition_computation_residue_map_extended_parametrisation}, for every $\stackS \in \sector$, we have
\begin{equation}\label{equation_description_2_stackT_ext_local}
    \stackT_{\stackS}(\Z_v) = \left\{y\in\stackT^{\ext}(\Z_v) \mid \psi_v\left(q^{\ext}(y)\right) = \stackS \right\} \, \, .
\end{equation}

For every $(L,\varphi)\in\Pic_{\orb}(X)$, we extend $\varphi$ by setting
$\varphi(0)=0$. From Proposition~\ref{proposition_computation_residue_map_extended_parametrisation}
and the equation \eqref{equation_lemma_computation_residue_map} in the proof of Lemma~\ref{lemma_computation_residue_map}, we deduce a first step toward another description of $\stackT^{\ext}(\Z_v)$.

\begin{lemma}\label{lemma_description_first_inclusion_extended_parametrisation}
For $\stackS\in\sector$, we have
\[
\stackT_{\stackS}(\Z_v)\subset
\left\{
y\in\stackT^{\ext}(\Q_v)\;\middle|\;
\age_v(y,(L,\varphi))=\varphi(\stackS)
\text{ for all } (L,\varphi)\in\Pic_{\orb}(X)
\right\}.
\]
\end{lemma}

From the previous lemma, we can deduce the injectivity
\[
\stackT^{\ext}(\Z_v) / T_{\NS,\orb}(\Z_v) \hookrightarrow X(\Q_v).
\]

\begin{prop}\label{prop_injectivite_local_extended_parametrisation}
Let $y_1, y_2 \in \stackT^{\ext}(\Z_v)$ be two elements such that
$$q^{\ext}(y_1) = q^{\ext}(y_2) .$$Then there exists $t \in T_{\NS,\orb}(\Z_v)$ such that
\[
y_2 = t \cdot y_1 .
\]
\end{prop}

\begin{proof}
Let $y_1, y_2 \in \stackT^{\ext}(\Z_v)$ be such that
$q^{\ext}(y_1) = q^{\ext}(y_2)$. Then there exists $t \in T_{\NS,\orb}(\Q_v)$ such that
\[
y_2 = t \cdot y_1 .
\]By
Lemma~\ref{lemma_description_first_inclusion_extended_parametrisation}, we have
\[
\age_v(y_2,-) = \age_v(y_1,-) \, .
\]
Proposition~\ref{prop_age_and_logarithm} then yields
\[
\log_v(t) = 0 \, .
\]
Hence $t \in T_{\NS,\orb}(\Z_v)$, which proves injectivity.

\end{proof}

We now show that $\stackT^{\ext}(\Z_v)$ is a natural candidate for defining an integral parametrization of $X(\Q_v)$. Namely, every element of $\stackT^{\ext}(\Q_v)$ lies in the same orbit as an element of $\stackT^{\ext}(\Z_v)$, that is to say the map \[
\stackT^{\ext}(\Z_v) / T_{\NS,\orb}(\Z_v) \rightarrow X(\Q_v).
\]is surjective. We will need the following lemma.

\begin{lemma}\label{lemma_inequality_valuation_plus_age}
Let $y \in \stackT^{\ext}(\Q_v)$. Then for every $\rho \in \Sigma(1)$, we have the inequality
\[
v(y_{\rho}) + \lfloor \age_v(y,-[D_{\rho}]) \rfloor \geqslant 0 .
\]
\end{lemma}

\begin{proof}
Let $\Le/\Q_v$ be a finite extension such that there exists
\[
z \in \stackT(\sheafO_{\Le}) \times \G_m(\sheafO_{\Le})^{\twistsector}
\]
lying in the same $T_{\NS,\orb}(\Le)$-orbit as $y_{|\Le}$. Using the definition of the lift of the age associated with the $T_{\NS,\orb}$-torsor (see Definition~\ref{def_age_univ_torsor}) and the definition of the extended universal torsor (see Definition \ref{definition_extended_universal_torsor}), we obtain the equality for any $\rho \in \Sigma(1)$:
\[
\frac{v_{\Le}(z_{\rho})}{e(\Le/\Q_v)}
=
v(y_{\rho}) + \age_v(y,-[D_{\rho}]) \geqslant 0
\]
in $\Z \cup \{+\infty\} $. Passing to the lower integer part, we obtain the lemma.

\end{proof}

We may now state the key proposition of this paragraph.

\begin{prop}\label{prop_surjectivite_local_extended_quotient_description}
Let $y \in \stackT^{\ext}(\Q_v)$. Then there exists an element
$\alpha \in \Pic_{\orb}(X)^{\vee}$ such that, if
$t \in T_{\NS,\orb}(\Q_v)$ satisfies $\log_v(t) = \alpha$, then
\[
t \cdot y \in \stackT^{\ext}(\Z_v).
\]
\end{prop}

\begin{proof}
Let $P = q^{\ext}(y)$. Assume first that $\psi_v(P) = 0$. In this case, by Theorem~\ref{theorem_age_extended_universal_torsor}, we have
\[
\age_v(y,-) \in \Pic_{\orb}(X)^{\vee},
\]
so we may consider an element $t \in T_{\NS,\orb}(\Q_v)$ such that
\[
\log_v(t) = \age_v(y,-).
\]
Then, by Proposition~\ref{prop_age_and_logarithm}, we obtain
\[
\age_v(t \cdot y,-) = 0.
\]
It therefore follows from Proposition~\ref{prop_extended_univ_torsor_zero_age} that
\[
t \cdot y \in \stackT(\Z_v) \times \G_m(\Z_v)^{\twistsector} \subset \stackT^{\ext}(\Z_v).
\]

Assume now that $\psi_v(P) = \stackS \in \twistsector$. Let $\Le/\Q_v$ be a finite extension such that there exists
\[
z \in \stackT(\sheafO_{\Le}) \times \G_m(\sheafO_{\Le})^{\twistsector}
\]
lying in the same $T_{\NS,\orb}(\Le)$-orbit as $y_{|\Le}$. Using the definition of the lift of the age associated with the $T_{\NS,\orb}$-torsor (see Definition~\ref{def_age_univ_torsor}), we obtain the equality
\[
\frac{v_{\Le}(z)}{e(\Le/\Q_v)}
=
v(y) - \pi^{\vee}(\age_v(y,-)) 
\]
in $\left( \Z \cup \{+\infty\} \right)^{\Sigma_{\ext}(1)}$.

By
Theorem~\ref{theorem_age_extended_universal_torsor}, Proposition \ref{proposition_functoriality_age} and
Proposition~\ref{prop_age_divisor_associated_to_an_edge}, we further have
\begin{equation}\label{equation_computation_value_n}
\pi^{\vee}\!\left( \age_v(y,-) \right)
=
n - \sum_{\rho \in \sigma_\stackS(1)} q_{\rho}\, e_{\rho} \, ,   
\end{equation}where
$q_{\rho} = \age_{\num}\!\left(\stackS,-[D_{\rho}]\right) \in [0,1[$ and
$n \in \Z^{\Sigma_{\ext}(1)}$.

We denote by $b_{\stackS} \in \Boxe(\Sigma)$ the element associated with $\stackS$.
Recall that
\[
\overline b_{\stackS}
=
\sum_{\rho \in \sigma_\stackS(1)} q_{\rho}\, \beta^{\ext}_{\Q}(e_{\rho}).
\]
By Remark~\ref{remark_exact_sequence_dual_orbifold_picard_group}, we have the following
commutative diagram with exact rows:
\begin{center}
\begin{tikzcd}
0 \arrow[r]
& \Pic_{\orb}(X)^{\vee}
\arrow[r, "\pi^{\vee}"]
\arrow[d]
& \Z^{\Sigma_{\ext}(1)}
\arrow[r, "\beta^{\ext}"]
\arrow[d]
& N
\arrow[r]
\arrow[d]
& 0
\\
0 \arrow[r]
& \Pic_{\orb}(X)^{\vee}_{\Q}
\arrow[r]
& \Q^{\Sigma_{\ext}(1)}
\arrow[r]
& N_{\Q}
\arrow[r]
& 0
\end{tikzcd}
\end{center}

From the previous equality, we obtain
\[
\beta^{\ext}_{\Q}(n) = \overline b_{\stackS}.
\]
Therefore, there exists $\stackS' \in \twistsector$ such that
$\sigma_\stackS' = \sigma_\stackS$,
$\overline b_{\stackS'} = \overline b_{\stackS}$, and
\[
\beta^{\ext}(n) = b_{\stackS'}.
\]
It follows that $n - e_{\stackS'} \in \pi^{\vee}\left(\Pic_{\orb}(X)^{\vee}\right)$, thus there exists $\alpha \in \Pic_{\orb}(X)^{\vee}$ such that
\[
n - e_{\stackS'} = \pi^{\vee}\left(\alpha\right).
\]

We then consider $t \in T_{\NS,\orb}(\Q_v)$ such that
\[
\log_v(t) = \alpha.
\]
Using \eqref{equation_formula_equivariance_valuation}, we obtain the equality
\[
\frac{v_{\Le}(z)}{e(\Le/\Q_v)}
=
v(t \cdot y)
+
\sum_{\rho \in \sigma_\stackS(1)} q_{\rho}\, e_{\rho}
-
e_{\stackS'}
\]
in $\left( \Z \cup \{+\infty\} \right)^{\Sigma_{\ext}(1)}$. We now show that $y' = t \cdot y \in \stackT^{\ext}(\Z_v)$. First, we observe that for
every $\stackZ \neq \stackS'$, we have $v(y'_{\stackZ}) = 0$, and that
$v(y'_{\stackS'}) = 1$. It therefore remains to prove that for any primitive
collection $C$, we have
\[
\min_{\rho \in C \setminus \sigma_{\stackS'}(1)} v(y'_{\rho}) = 0 .
\]

To begin with, since for every $\rho \in \Sigma(1)$ we have by \eqref{equation_computation_value_n}
\[
n(e_{\rho}) = - \lfloor \age_v(y,-[D_{\rho}]) \rfloor,
\]
it follows that
\[
v(y'_{\rho})
=
v(y_{\rho}) + \lfloor \age_v(y,-[D_{\rho}]) \rfloor
\geqslant 0
\]
by Lemma~\ref{lemma_inequality_valuation_plus_age}. Now, using the fact that
\[
z \in \stackT(\sheafO_{\Le}) \times \G_m(\sheafO_{\Le})^{\twistsector},
\]
we have
\[
0
=
\min_{\rho \in C}
\frac{v_{\Le}(z_{\rho})}{e(\Le/\Q_v)}
=
\min_{\rho \in C}
\left(
v(y'_{\rho}) + \age_{\num}\!\left(\stackS',-[D_{\rho}]\right)
\right).
\]
Since for every $\rho \notin \sigma_{\stackS'}(1)$ we have
$\age_{\num}\!\left(\stackS',-[D_{\rho}]\right) = 0$, while for
$\rho \in \sigma_{\stackS'}(1)$ we have
$\age_{\num}\!\left(\stackS',-[D_{\rho}]\right) \in ]0,1[$, and since moreover
$v(y'_{\rho}) \geqslant 0$, it follows that
\[
\min_{\rho \in C \setminus \sigma_{\stackS'}(1)} v(y'_{\rho}) = 0 .
\]Note that, by Proposition~\ref{proposition_computation_residue_map_extended_parametrisation},
we obtain $\stackS' = \stackS$.

\end{proof}

We deduce the following corollary, which gives another description of the set $\stackT_{\stackS}(\Z_v)$ for $\stackS \in \sector$.

\begin{cor}\label{corollary_computation_extended_parametrization}
For $\stackS\in\sector$, we have
\[
\stackT_{\stackS}(\Z_v) =
\left\{
y\in\stackT^{\ext}(\Q_v)\;\middle|\;
\age_v(y,(L,\varphi))=\varphi(\stackS)
\text{ for all } (L,\varphi)\in\Pic_{\orb}(X)
\right\}.
\]In particular, we have:
\[
\stackT^{\ext}(\Z_v) =
\left\{
y\in\stackT^{\ext}(\Q_v)\;\middle|\;
\age_v(y,(L,\varphi))=\varphi\left(\psi_v\left(q^{\ext}(y)\right)\right)
\text{ for all } (L,\varphi)\in\Pic_{\orb}(X)
\right\}.
\]
\end{cor}

\begin{proof}
The inclusion $\subset$ is Lemma~\ref{lemma_description_first_inclusion_extended_parametrisation}. 
For $\stackS=0$, the reverse inclusion follows from 
Proposition~\ref{prop_extended_univ_torsor_zero_age}. Finally, let $\stackS\in\twistsector$ and let 
$y\in\stackT^{\ext}(\Q_v)$ satisfy
\[
\age_v(y,(L,\varphi))=\varphi(\stackS)
\quad \text{for all } (L,\varphi)\in\Pic_{\orb}(X).
\]
Then Proposition~\ref{prop_surjectivite_local_extended_quotient_description}
and its proof ensure that there exists 
$t\in T_{\NS,\orb}(\Q_v)$ such that 
$t.y\in\stackT_{\stackS}(\Z_v)$. Again by Lemma~\ref{lemma_description_first_inclusion_extended_parametrisation}, we get
$$
\age_v(t \cdot y,-) = \age_v(y,-).
$$Hence, by Proposition~\ref{prop_age_and_logarithm}, we deduce that $t \in T_{\NS,\orb}(\Z_v)$, and therefore
$$y \in \stackT_{\stackS}(\Z_v).$$The remainder of the corollary follows from the other description of $\stackT_{\stackS}(\Z_v)$ (see equation~\ref{equation_description_2_stackT_ext_local}) and from the fact that
$$
\stackT^{\ext}(\Z_v)
=
\bigsqcup\limits_{\stackS \in \twistsector}
\stackT_{\stackS}(\Z_v).
$$
\end{proof}

By Propositions~\ref{prop_injectivite_local_extended_parametrisation} and
\ref{prop_surjectivite_local_extended_quotient_description}, we have shown that the
map
\[
q^{\ext} : \stackT^{\ext}(\Z_v) \longrightarrow X(\Q_v)
\]
induces a bijection
\[
\stackT^{\ext}(\Z_v) / T_{\NS,\orb}(\Z_v)
\;\simeq\;
X(\Q_v).
\]We have also seen that for $\stackS \in \sector$, the restriction of the map
$q^{\ext}$ to $\stackT_{\stackS}(\Z_v)$ induces a bijection
\[
\stackT_{\stackS}(\Z_v) / T_{\NS,\orb}(\Z_v)
\;\simeq\;
\psi_v^{-1}(\stackS).
\]We now turn to the study of this parametrization in the global setting.

\subsection{Global description}

\begin{definition}
    An adelic family of sectors $\underline{\stackS} = (\stackS_v)_{v \in M_{\Q}^0} \in \sector^{(M_{\Q}^0)}$ is defined such that $\{ v \in M_{\Q}^0 \mid \stackS_v \neq 0 \} $ is finite. The set of adelic family of sectors will be written as the restricted product $\adelicprod \sector$.
\end{definition}

\begin{definition}
    We define the adelic residue map $$\psi : X(\Q) \rightarrow \adelicprod \sector$$ which maps a rational point $P \in X(\Q)$ to the associated adelic family $$\underline{\stackS} = \left( \psi_v(P) \right)_{v \in M_{\Q}^0} .$$
\end{definition}

The results obtained in the previous paragraph allow us to deduce the following
theorem.

\begin{theorem}\label{theorem_description_of_the_extended_parmetrization}
The quotient map
\[
q^{\ext} : \stackT^{\ext}(\Z) \longrightarrow X(\Q)
\]
induces a bijection
\[
\stackT^{\ext}(\Z) / T_{\NS,\orb}(\Z)
\;\simeq\;
X(\Q).
\]
Moreover, for $y \in \stackT^{\ext}(\Z)$ and $P = q^{\ext}(y)$, we have the equivalence
\[
\psi(P) = \underline{\stackS}
\quad\Longleftrightarrow\quad
\forall\, v \in M_{\Q}^{0},\
\begin{cases}
p_v \mid y{}_{\substack{\stackS_v}} & \text{if } \stackS_v \neq 0, \\[4pt]
p_v \nmid y{}_{\substack{\stackZ}} \text{ for all } \stackZ \in \twistsector & \text{if } \stackS_v = 0 .
\end{cases}
\]
\end{theorem}

\begin{proof}
  We first prove injectivity. Let $y_1, y_2 \in \stackT^{\ext}(\Z)$ be such that
$$q^{\ext}(y_1) = q^{\ext}(y_2) .$$Then there exists $t \in T_{\NS,\orb}(\Q)$ such that
\[
y_2 = t \cdot y_1 .
\]
Note that for every $v \in M_{\Q}^{0}$, we have
$y_i{}_{|\Q_v} \in \stackT^{\ext}(\Z_v)$. By Corollary \ref{corollary_description_extended_parametrization}
\[
\age_v(y_2,-) = \age_v(y_1,-)
\]
for all $v \in M_{\Q}^{0}$. Proposition~\ref{prop_age_and_logarithm} then yields
\[
\log_v(t) = 0
\qquad \text{for all } v \in M_{\Q}^{0}.
\]
Hence $t \in T_{\NS,\orb}(\Z)$, which proves injectivity. We now prove surjectivity. Let $y \in \stackT^{\ext}(\Q)$. For each
$v \in M_{\Q}^{0}$, by Proposition~\ref{prop_surjectivite_local_extended_quotient_description},
there exists an element $\alpha_v \in \Pic_{\orb}(X)^{\vee}$ such that, if
$t \in T_{\NS,\orb}(\Q_v)$ satisfies $\log_v(t) = \alpha_v$, then
\[
t \cdot y \in \stackT^{\ext}(\Z_v).
\]

Since the set $\{\, v \in M_{\Q}^0 \mid \age_v(y,-) \neq 0 \,\}$ is finite, it follows that
$$\{\, v \in M_{\Q}^0 \mid \alpha_v \neq 0 \,\}$$is also finite. Hence the family
$(\alpha_v)_{v \in M_{\Q}^0}$ defines an element of the direct sum
\[
\bigoplus_{v \in M_{\Q}^{0}} \Pic_{\orb}(X)^{\vee}.
\]
We may therefore use the surjectivity of the map
\[
T_{\NS,\orb}(\Q)
\xrightarrow{\ \oplus\,\log_v\ }
\bigoplus_{v \in M_{\Q}^{0}} \Pic_{\orb}(X)^{\vee}
\]
to choose an element $t \in T_{\NS,\orb}(\Q)$ whose image is precisely
$(\alpha_v)_v$. We then have
\[
t \cdot y \in
\stackT^{\ext}(\Q)
\cap
\bigcap_{v \in M_{\Q}^{0}} \stackT^{\ext}(\Z_v)
=
\stackT^{\ext}(\Z).
\]
This proves surjectivity. The final assertion is a consequence of
Proposition~\ref{proposition_computation_residue_map_extended_parametrisation}.

\end{proof}

We can now state a corollary of the previous result, which will allow us to define
local heights on the extended universal torsor.

\begin{cor}\label{corollary_description_extended_parametrization}
Let $y \in \stackT^{\ext}(\Z)$ and let $P \in X(\Q)$ be the corresponding rational
point. Then for every $\stackS \in \twistsector$, we have
\[
|y{}_{\stackS}| = H_{[\stackS]}(P) .
\]
\end{cor}

\begin{proof}
    Recall that
\[
H_{[\stackS]}(P) = \prod_{\psi_v(P) = \stackS} p_v .
\]
The corollary therefore follows from the final assertion of
Theorem~\ref{theorem_description_of_the_extended_parmetrization} and the fact that $y{}_{\stackS}$ is square free by the definition of $\stackT^{\ext}(\Z)$.

\end{proof}

\section{A closer look at the case of toric stacks with torsion-free Picard group}\label{appendix_B}

The goal of this section and of the next one is to define a lift of the stacky height to the extended universal torsor. Here, we first consider the simpler case of toric stacks with torsion-free Picard group. We begin by presenting the key example of weighted projective stacks. We then introduce a more general formalism that explains this example, and finally we use this formalism to construct a naive local height over the extended universal torsor to lift the coarse multi-height map $h_{\coarse}$ (see Definition~\ref{definition_coarse_multi_height_map}).

\subsection{The case of the weighted projective stacks}

Let us consider the example of the weighted projective stack $\stackP(\w)$ with weight $\w = (w_0,..,w_n) \in (\N^*)^{n + 1}$. In this particular case, the map $$\age_{\num}(-,\sheafO(1)) : \pi_0(\I_{\mu} \stackP(\w)) \rightarrow \bigcup\limits_{0 \leqslant i \leqslant n} \frac{1}{w_i} \Z \cap [0,1[$$ is a bijection (see Theorem \ref{theorem_sectors_dual_picard_group_toric_stacks}). We denote by $\stackS_t$ the sector such that $$\age_{\num}(\stackS_t,\sheafO(1)) = t \, .$$In what follows, we shall describe both a naive and the extended parametrization of $\stackP(\w)(\Q)$ in the spirit of the theory we develop this paper.

\begin{definition}
    For $(x_0,..,x_n) \in \K^{n+1} \setminus \{0\}$, we define a fractional ideal of $\sheafO_{\K}$, $I_w(x) = \bigcap\limits_{\forall i \ x_i \in \mathrm{b}^{w_i} } \mathrm{b}$. It is easier to understand this fractional ideal through its inverse:
    $$ I_w(x)^{-1} = \{ \alpha \in \K \mid \alpha \underset{w}{.} x \in \sheafO_{\K} \} .$$
\end{definition}

As we work over $\Q$, we introduce here some nice arithmetic notations. For $(x_0,..,x_n) \in \Z^{n+1} \setminus \{0\}$, we denote by $\gcd_w(x)$ the element of $\N^*$, such that $$I_w(x) = \left( \text{gcd}_w (x) \right).$$ Here are a few characterization of $\gcd_w(x)$.

\begin{prop}
    $\gcd_w(x)$ is the unique non negative integer $\delta \in \N^*$ such that the following assertions are equivalent:
    \begin{itemize}
        \item for all $0 \leqslant i \leqslant n$, $d^{w_i} \mid x_i$
        \item $d \mid \delta$
    \end{itemize}
    It is also characterized such that for any prime $p$, $v_p (\gcd_w (x)) = \lfloor \min\limits_{0 \leqslant i \leqslant n} \big( \frac{v_p(x_i)}{w_i} \big) \rfloor$. Moreover we have the property that $\gcd_w (d \underset{w}{.} x) = d . \gcd_w (x)$.
\end{prop}

A naive parametrization of $\stackP(\w)(\Q)$ is given by:
\begin{equation}\label{equation_example_parametrisation_naive_pw}
    \stackT^{\stack}(\Z) = \{ y \in \Q^{n+1}-\{0\} \mid \text{gcd}_w(y) = 1 \}  .
\end{equation}One may verify that we get a new quotient description of $\stackP(\w)(\Q)$ as $$ \stackT^{\stack}(\Z) / \G_m(\Z) $$where $\G_m$ continue to act with weight $\w$. In example \ref{example_computation_age_weighted_projective_stack}, we have seen that for any $v \in M_{\Q}^0$ and for any $y \in \Q_v^{n+1}-\{0\}$, $$\age_v(y,\sheafO(1)) = \min\limits_{0 \leqslant i \leqslant n} \frac{v(y_i)}{w_i} .$$If $P = q(y) \in \stackP(\w)(\Q_v)$, we get by theorem \ref{theorem_age_extended_universal_torsor}:
\begin{equation}\label{equation_computation_residue_map_weighted_projective_stack}
    \psi_v(P) = \stackY_t  \Leftrightarrow  \age_v(y,\sheafO(1)) \equiv t \mod \Z \, \, .
\end{equation}

For $t \in \bigcup\limits_{0 \leqslant i \leqslant n} \frac{1}{w_i} \Z \cap [0,1[$ and $v \in M_{\Q}^0$, we define:
\begin{equation}
    \stackT^{\stack}_{\stackS_t}(\Z) = \{ y \in \Q_v^{n+1}-\{0\} \mid \min\left(\frac{v(y_i)}{\w_i} \right) = t \}  .
\end{equation}We also introduce the following notation:
\begin{Notations}
We denote by $[\sheafO(1)]_{\stack}$ the class in $\Pic_{\orb}(\stackP(\w))$ given by
$(\sheafO(1),\varphi)$, where the map
\[
\varphi : \pi_0^*(\I_{\mu} \stackP(\w)) \longrightarrow \Q
\]
is defined by the condition that, for any $\stackS_t \in \pi_0^*(\I_{\mu} \stackP(\w))$,
\[
\varphi(\stackS_t) = \age_{\num}(\stackS_t,\sheafO(1)) = t.
\]  
\end{Notations}
We can compute a family of heights for any rational point of $\stackP(\w)$:

\begin{prop}\label{proposition_height_weighted_projective_stack}
    Let $y \in \Q^{n + 1}-\{0\}$ and consider $P = q(y) \in \stackP(\w)(\Q)$. We have:
    \begin{enumerate}
        \item For any $t \in \bigcup\limits_{0 \leqslant i \leqslant n} \frac{1}{w_i} \Z \cap ]0,1[$, $$H_{\stackS_t}(P) = \prod\limits_{\left\{ \min\limits_{0 \leqslant i \leqslant n} \frac{v(y_i)}{w_i} \right\} = t} p_v ;$$
        \item $$H_{[\sheafO(1)]_{\stack}}(P) = \frac{\max\limits_{0 \leqslant i \leqslant n }|y_i|^{\frac{1}{w_i}}}{\gcd_w(y)} .$$
    \end{enumerate}
\end{prop}

\begin{proof}
The first assertion is a consequence of the equivalence given in equation~\eqref{equation_computation_residue_map_weighted_projective_stack}. Now for any $v \in M_{\Q}^0$, if we write $t_v = \min\limits_{0 \leqslant i \leqslant n} \left( \frac{v(y_i)}{w_i} \right)$, we have $$\max\limits_{0 \leqslant i \leqslant n }|y_i|_v^{\frac{1}{w_i}} = p_v^{-t_v}$$and $\age_{\num}\left({\stackS_{\{t_v\}}},\sheafO(1)\right) = \{t_v\}$. Hence we have by Definition \ref{definition_stacky_height}:
    $$
H_{[\sheafO(1)]_{\stack}}(P)
=
\max\limits_{0 \leqslant i \leqslant n} |y_i|^{\frac{1}{w_i}}
\cdot
\prod\limits_{v \in M_{\Q}^0} p_v^{-t_v + \{t_v\}}
=
\frac{\max\limits_{0 \leqslant i \leqslant n} |y_i|^{\frac{1}{w_i}}}{\gcd_w(y)}.
$$
\end{proof}

\begin{remark}\label{remark_computation_weighted_stack_height}
    In particular, we see that if $y \in \stackT^{\stack}(\Z)$ and $P = q(y)$, then we have
$$
H_{[\sheafO(1)]_{\stack}}(P) = \max\limits_{0 \leqslant i \leqslant n }|y_i|^{\frac{1}{w_i}}.
$$This is the phenomenon that we will exploit in the sequel to define a lift of the stacky height.
\end{remark}

We now describe the relation between this naive parametrisation (see equation~\ref{equation_example_parametrisation_naive_pw}) and the extended parametrisation (see Definition~\ref{definition_orb_universal_torsor_gcd_cond}) of $\stackP(w)(\Q)$. Recall that we proved in Theorem~\ref{theorem_description_of_the_extended_parmetrization} that the map $q^{\ext} : \stackT^{\ext}(\Q) \rightarrow \stackP(w)(\Q)$ induces a bijection:
$$
\stackP(\w)(\Q) = \stackT^{\ext}(\Z) / T_{\NS,\orb}(\Z).
$$

\begin{theorem}\label{theorem_first_example_orbifold_universal_torsor}
 Let $(y',k) \in \Q^{n + 1}-\{0\} \times \left( \Q^{\times} \right)^{\bigcup\limits_{0 \leqslant i \leqslant n} \frac{1}{w_i} \Z \cap ]0,1[}$ be a tuple of integers such that:
\begin{enumerate}
    \item for any $t \in \bigcup\limits_{0 \leqslant i \leqslant n} \frac{1}{w_i} \Z \cap ]0,1[$, the $k_t$ are square-free non-negative integers and are pairwise coprime;
    \item $\gcd(y'_0,..,y'_n) = 1$;
    \item $\gcd\left( k_t, y'_j, j \in J_t \right) = 1 \text{ where } J_t = \left\{ j \mid \w_j t \in \Z \right\} \quad \forall t \in \bigcup\limits_{0 \leqslant i \leqslant n} \frac{1}{w_i} \Z \cap ]0,1[$.
\end{enumerate}
That is, $(y',k) \in \stackT^{\ext}(\Z)$.

Then the element $y \in \Q^{n + 1}-\{0\}$ defined, for every $j$, by
$$
y_j = \left(\prod\limits_{t \in \bigcup\limits_{0 \leqslant i \leqslant n} \frac{1}{w_i} \Z \cap ]0,1[} k_t^{m_t(j)} \right) \cdot y'_j,
\qquad \text{where } m_t(j) = \lceil \w_j t \rceil,
$$
satisfies
$$
y = (y_0,..,y_n) \in \stackT^{\stack}(\Z).
$$
    
\end{theorem}

\begin{proof}
This is a particular case of Theorem~\ref{theorem_lien_extended_parametrisation_naive_parametrisation}, using Remark~\ref{remark_choice_section_weighted_proj_stack}.
\end{proof}

Since the family
$$
\left\{ [\sheafO(1)]_{\stack} \right\}
\cup
\left\{ [\stackS_t] \mid t \in \bigcup\limits_{0 \leqslant i \leqslant n} \frac{1}{w_i}\Z \cap ]0,1[ \right\}
$$
is a basis of $\Pic_{\orb}(\stackP(\w))$ (see Proposition~\ref{prop_basis_of_orbifold_picard_group}), we can construct a lift of the stacky height. We refer the reader to Definitions~\ref{definition_lift_free_picard_stacky_height}, \ref{definition_naive_local_height_extended_universal_torsor} and \ref{definition_local_height_extended_universal_torsor} for the definition of $h_{\stackT,\infty}^{\ext} = h_{\stackT,\infty}^{\stack}$ in the general case.

For $(y',k) \in \stackT^{\ext}(\R)$, we define a lift of the stacky height
$$
h_{\stackT,\infty}^{\ext} : \stackT^{\ext}(\R) \longrightarrow \Pic_{\orb}(X)^{\vee}_{\R}
$$
by setting, for $t \in \bigcup\limits_{0 \leqslant i \leqslant n} \frac{1}{w_i}\Z \cap ]0,1[$,
$$
h_{\stackT,\infty}^{\ext}(y',k)\bigl([\stackS_t]\bigr)
=
\log \left| k_t \right|
$$
and
$$
h_{\stackT,\infty}^{\ext}(y',k)\bigl([\sheafO(1)]_{\stack}\bigr)
=
\log \left(
\max\limits_{0 \leqslant j \leqslant n}
\left|
\left(
\prod\limits_{t \in \bigcup\limits_{0 \leqslant i \leqslant n} \frac{1}{w_i}\Z \cap ]0,1[}
k_t^{m_t(j)}
\right)
y'_j
\right|^{\frac{1}{w_j}}
\right).
$$

Thus, by Theorem \ref{theorem_lift_height_extend_universal_torsor}, for every $y \in \stackT^{\ext}(\Z)$, if $P = q^{\ext}(y)$ we have
$$
h(P) = h_{\stackT,\infty}^{\ext}(y).
$$

\subsection{A naive parametrisation of rational points of toric stacks}

Throughout this appendix, the Picard group of the toric stack $X$, $\Pic(X)$, is assumed to be torsion-free. Let us fix a basis $\base = \{ L_1,\ldots,L_t \}$ of $\Pic(X)$. We shall consider a set-theoretic section $s$ of
$$
\Hom_{\grp}(\Pic(X),\Q) \longrightarrow \Hom_{\grp}(\Pic(X),\Q / \Z) \, .
$$We introduce the following notation.

\begin{Notations}
For $\stackS \in \twistsector$, we define a lift of $\age(\stackS,-)$, denoted by $\widetilde{\age}(\stackS,-)$, as the unique group homomorphism such that for every $1 \leqslant i \leqslant t$,
$$
\widetilde{\age}(\stackS,L_i)
=
s\!\left(\age(\stackS,-)\right)(L_i).
$$By convention, we take $\widetilde{\age}(0,-)$ to be the zero morphism. Observe that by definition for every $\stackS \in \sector$ and every $L \in \Pic(X)$,
$$
\age(\stackS,L) \equiv \widetilde{\age}(\stackS,L) \mod \Z .
$$
\end{Notations}

\begin{remark}\label{remark_choice_section_weighted_proj_stack}
The case of the weighted projective stack $\stackP(\w)$ presented in the previous paragraph can be interpreted as a particular case of the construction that follows if we choose $\base = \{ \sheafO(1) \}$ as a basis of $\Pic(\stackP(\w))$ and the set-theoritic section $s$ such that for every $\varphi \in \Hom_{\grp}(\Pic(X),\Q / \Z)$ we have
$$
s(\varphi)(\sheafO(1)) \in [0,1[ \, .
$$
\end{remark}

For $v \in M_{\Q}^0$, we define a naive parametrisation of $X(\Q_v)$ as follows.

\begin{definition}
For every $\stackS \in \sector$, we define
$$
\stackT^{\stack}_{\stackS}(\Z_v)
=
\left\{
y \in \stackT(\Q_v)
\;\middle|\;
\age_v(y,-) = \widetilde{\age}(\stackS,-)
\right\}
$$
and more generally
$$
\stackT^{\stack}(\Z_v)
=
\bigsqcup\limits_{\stackS \in \sector}
\stackT^{\stack}_{\stackS}(\Z_v) \, .
$$
\end{definition}

\begin{remark}
From Theorem~\ref{theorem_age_extended_universal_torsor} and Theorem~\ref{theorem_sectors_dual_picard_group_toric_stacks}, we deduce that if $y \in \stackT^{\stack}_{\stackS}(\Z_v)$ and $P = q(y)$, then $\psi_v(P) = \stackS$.
\end{remark}

In the global case, we set
\begin{equation}\label{equation_global_naive_parametrisation}
\stackT^{\stack}(\Z)
=
\bigcap\limits_{v \in M_{\Q}^0}
\stackT^{\stack}(\Z_v)
\cap
\stackT(\Q) \, .
\end{equation}

Associated with our set-theoretic section $s$ and our choice of basis $\base$, we define the following group homomorphism
\begin{equation}\label{equation_appendix_B_definition_lambda_map}
\begin{aligned}
\lambda :\;& \Pic(X) \longrightarrow \Pic_{\orb}(X) \\
&L \longmapsto (L,\widetilde{\age}(-,L)) .
\end{aligned}
\end{equation}

We then obtain the following description of $\Pic_{\orb}(X)$.

\begin{prop}\label{prop_basis_of_orbifold_picard_group}
The family
$$
\{ [\lambda(L_i)] \mid 1 \leqslant i \leqslant t \}
\cup
\{ [\stackS] \mid \stackS \in \twistsector \}
$$
is a basis of $\Pic_{\orb}(X)$.
\end{prop}

\begin{proof}
This follows from the fact that the map $\lambda$ defines a splitting of the exact sequence of Proposition~\ref{prop_exact_sequence_extended_picard_group}.
\end{proof}

We now consider the map
$$
f : \Z^{\Sigma(1)} \longrightarrow \Z^{\Sigma(1) \cup \twistsector}
$$
defined, for every $\rho \in \Sigma(1)$, by
$$
f(e_{\rho})
=
e_{\rho}
+
\sum\limits_{\stackS \in \twistsector}
m_{\stackS}(\rho) e_{\stackS},
$$
where $m_{\stackS}(\rho)$ denotes the integer
$$
\widetilde{\age}(\stackS,[D_{\rho}])
+
\age_{\num}(\stackS,-[D_{\rho}]) \, .
$$

By construction, we obtain the following lemma.

\begin{lemma}\label{lemma_appendix_B_commutative_diagram}
The following diagram is commutative:
\begin{center}
\begin{tikzcd}
\Z^{\Sigma(1)} \arrow[d] \arrow[r, "f"] & \Z^{\Sigma(1) \cup \twistsector} \arrow[d] \\
\Pic(X) \arrow[r, "\lambda"]            & \Pic_{\orb}(X)
\end{tikzcd}
\end{center}
where the vertical maps associate to an element its class as a line bundle (respectively orbifold line bundle).
\end{lemma}

We define the map
$$
\Phi : \stackT \times \G_m^{\twistsector} \longrightarrow \stackT
$$
to be the morphism induced by the dual of $f$, in the sense that for every $\rho \in \Sigma(1)$, the coordinate of $\Phi(y)$ associated with $\rho$ is
\begin{equation}\label{equation_definition_retraction_appendix_B}
\Phi(y)_{\rho}
=
y_{\rho}
\prod\limits_{\stackS \in \twistsector}
y_{\stackS}^{m_{\stackS}(\rho)} \, .
\end{equation}

By Lemma~\ref{lemma_appendix_B_commutative_diagram}, the morphism $\Phi$ is equivariant with respect to the morphism
$$
\varphi : T_{\NS,\orb} \longrightarrow T_{\NS}
$$
defined as the dual of $\lambda$. We obtain the following result.

\begin{prop}\label{prop_appendix_B_retraction_torsion_free_case}
The quotient morphism induced by $\Phi$,
$$
[\stackT \times \G_m^{\twistsector} / T_{\NS,\orb}]
\xrightarrow{\overline \Phi}
[\stackT / T_{\NS}],
$$
is $2$-isomorphic to the identity.
\end{prop}

\begin{proof}
The map
$$
\stackT \longrightarrow \stackT \times \G_m^{\twistsector}
$$
defined by $y \mapsto (y,1)$ defines a section of $\Phi$. In particular, we obtain the following commutative diagram:
\begin{center}
\begin{tikzcd}
\stackT \arrow[r] \arrow[rd, "\mathrm{id}_{\stackT}"'] & \stackT \times \G_m^{\twistsector} \arrow[d, "\Phi"] \\
                                                      & \stackT \, .
\end{tikzcd}
\end{center}

Passing to the quotient, we obtain the following $2$-commutative diagram:
\begin{center}
\begin{tikzcd}
X \arrow[r] \arrow[rd, "\mathrm{id}_{X}"'] & X \arrow[d, "\overline{\Phi}"] \\
                                                 & X
\end{tikzcd}
\end{center}where the horizontal map corresponds to the morphism of Theorem~\ref{theorem_extended_universal_torsor}, which is $2$-isomorphic to $\mathrm{id}_{X}$. This concludes the proof.
\end{proof}

We now show that the map $\Phi$ allows us to relate the extended parametrisation defined in Definition~\ref{definition_orb_universal_torsor_gcd_cond} with the naive parametrisation (see Equation~\eqref{equation_global_naive_parametrisation}):

\begin{theorem}\label{theorem_lien_extended_parametrisation_naive_parametrisation}
For $y \in \stackT^{\ext}(\Z)$, we have
$$
\Phi(y) \in \stackT^{\stack}(\Z).
$$
\end{theorem}

\begin{proof}
It suffices to show that for $v \in M_{\Q}^0$, if $y \in \stackT^{\ext}(\Z_v)$ then $\Phi(y) \in \stackT^{\stack}(\Z_v)$. By Proposition~\ref{prop_functoriality_lif_age}, for every $L \in \Pic(X)$ we have
$$
\age_v(\Phi(y),L)
=
\age_v(y,\lambda(L))
=
\widetilde{\age}(\psi_v(P),L),
$$
where the last equality follows from the description of $\stackT^{\ext}(\Z_v)$ given in Corollary~\ref{corollary_computation_extended_parametrization} and Equation \eqref{equation_appendix_B_definition_lambda_map}.
\end{proof}

\subsection{A naive local height over the extended universal torsor}

In this paragraph, we define for $v \in M_{\Q}$ a local height
\[
    h_{\stackT,v}^{\stack} : \stackT^{\ext}(\Q_v) \longrightarrow \Pic_{\orb}(X)^{\vee}_{\R}
\]
which will allow us to lift the coarse multi-height $h_{\coarse}$ (see Definition~\ref{definition_coarse_multi_height_map}) to toric stacks. Recall that the toric variety $X^{\coarse}$ is always equipped with its natural system of heights (see \cite[Definition~9.2]{salberger_torsor} and \cite[Proposition~2.1.2]{Batyrev1990}).

To define a lift of the coarse height to the extended universal torsor, we use the fact that, by Proposition~\ref{prop_basis_of_orbifold_picard_group}, we have the direct sum decomposition 
$$
\Pic_{\orb}(X)_{\R}
=
\lambda\left(\Pic(X)_{\R}\right)
\oplus
\bigoplus_{\stackS \in \twistsector} \R [\stackS].
$$
We can therefore introduce the following definition.

\begin{definition}\label{definition_lift_free_picard_stacky_height}
Let $v \in M_{\Q}$. We define a local height on the extended universal torsor
\[
h_{\stackT,v}^{\stack} : \stackT^{\ext}(\Q_v) \longrightarrow \Pic_{\orb}(X)^{\vee}_{\R}
\]
by requiring that, for any $y \in \stackT^{\ext}(\Q_v)$, the following conditions hold:
\begin{itemize}
\item
for any $\stackS \in \twistsector$,
\[
h_{\stackT,v}^{\stack}(y)\bigl([\stackS]\bigr)
=
\log_{p_v} \left| y_{\stackS} \right|_v;
\]
\item
for any $L \in \Pic(X)_{\R}$,
\[
h_{\stackT,v}^{\stack}(y)\bigl(\lambda(L)\bigr)
=
h_{\stackT,v}\bigl(\Phi(y)\bigr)(L)
\]where $h_{\stackT,v}$ is the local height from Definition~\ref{definition_naive_local_degree}, with the choice \\ $y_0 = 1 \in \stackT(\Z)$.
\end{itemize}
\end{definition}

We will need the following invariance property under the action of $T_{\NS,\orb}(\Q_v)$.

\begin{prop}\label{proposition_appendix_B_property_equivariant_extended_height}
    For any $t \in T_{\NS,\orb}(\Q_v)$ and any $y \in \stackT^{\ext}(\Q_v)$, $$h_{\stackT,v}^{\stack}(t.y) = h_{\stackT,v}^{\stack}\left(y\right) + \log_{T_{\NS,\orb},v}(t)$$
\end{prop}

\begin{proof}
    By the description of the action of $T_{\NS,\orb}$ on the extended universal torsor (see Definition \ref{definition_extended_universal_torsor}), we have for any $\stackS \in \twistsector$,
    \begin{align*}
        h_{\stackT,v}^{\stack}(t.y)([\stackS]) &=  \log_{p_v} | y{}_{\stackS} |_v + \log_{p_v} |[\stackS](t)^{-1}|_v \\
        &= h_{\stackT,v}^{\stack}\left(y\right)([\stackS]) + \log_{T_{\NS,\orb},v}(t)([\stackS]).
    \end{align*}
    Moreover, for any $L \in \Pic(X)_{\R}$,
    \begin{align*}
        h_{\stackT,v}^{\stack}(t.y)(\lambda(L)) &= h_{\stackT,v}(\Phi(t.y))(L) \\
        &= h_{\stackT,v}\left(\varphi(t).\Phi(y)\right)(L)  \\
        &= h_{\stackT,v}\left(\Phi(y)\right)(L) + \log_{T_{\NS},v}(\varphi(t))(L) \\
        &= h_{\stackT,v}^{\stack}\left(y\right)(\lambda(L)) + \log_{T_{\NS,\orb},v}(t)(\lambda(L)) \, .
    \end{align*}
\end{proof}
We now justify the notation $h_{\stackT,v}^{\stack}$ by showing that this construction does not depend on the choice of the section $s$ or of the basis $\base$.

\begin{prop}\label{proposition_extended_height_do_not_depend_on_the_base}
    The definition of $h_{\stackT,v}^{\stack}$ does not depend on the choice of $s$ or $\base$.
\end{prop}

\begin{proof}
Let $\base_1$ and $\base_2$ be two bases of $\Pic(X)$, and let $s_1$ and $s_2$ be two choices of set-theoretic sections. We denote by $g_i$ the map $h_{\stackT,v}^{\stack}$ obtained from the choice of the basis $\base_i$ and the section $s_i$, and we write
$$
\widetilde{\age}_i(-,-)
$$
for the corresponding lift of the age pairing. This choices also give retractions
$$
\Phi_i : \stackT^{\ext} \longrightarrow \stackT
$$
as constructed in Equation~\eqref{equation_definition_retraction_appendix_B}. Even if this means replacing $y$ by $t \cdot y$, with $t \in T_{\NS,\orb}(\Z_v)$ for $v \in M_{\Q}^0$ or $t \in T_{\NS,\orb}(\Z)$ for $v = \infty$, Proposition~\ref{prop_basis_of_orbifold_picard_group} allows us to assume the following.

For $v \in M_{\Q}^0$, we can assume that for any $\stackS \in \twistsector$,
\[
    (t \cdot y)_{\stackS} = p_v^{n_\stackS}
\]for some $n_\stackS \in \Z$, while for $v = \infty$, we can assume that
\[
    (t \cdot y)_{\stackS} = |y_{\stackS}| \geqslant 0.
\]Moreover, this normalization does not change $g_i(y)$ by Proposition~\ref{proposition_appendix_B_property_equivariant_extended_height}.

We want to show that
$$
g_1(y) = g_2(y).
$$
For any $\stackS \in \twistsector$, we already have
$$
g_1(y)([\stackS]) = g_2(y)([\stackS]).
$$

Let
$$
\lambda_i : \Pic(X) \longrightarrow \Pic_{\orb}(X)
$$
be the section defined by the choice of $s_i$ and $\base_i$. For $L \in \Pic(X)$, we have
$$
\lambda_1(L)
=
\lambda_2(L)
+
\sum_{\stackS \in \twistsector}
\left(
\widetilde{\age}_1(\stackS,L)
-
\widetilde{\age}_2(\stackS,L)
\right)
[\stackS].
$$

Hence we obtain
\begin{align*}
g_2(y)(\lambda_1(L))
&=
g_2(y)(\lambda_2(L))
+
\sum_{\stackS \in \twistsector}
\left(
\widetilde{\age}_1(\stackS,L)
-
\widetilde{\age}_2(\stackS,L)
\right)
\log |y_{\stackS}|_v \\
&= h_{\stackT,v}(\Phi_2(y))(L) + \log_{T_{\NS},v}(t) \\
&=
h_{\stackT,v}(t \cdot \Phi_2(y))(L),
\end{align*}
where we choose $t \in T_{\NS}(\Q_v)$ such that
$$
\log_{T_{\NS},v}(t)
=
\alpha
=
\sum_{\stackS \in \twistsector}
\left(
\widetilde{\age}_2(\stackS,-)
-
\widetilde{\age}_1(\stackS,-)
\right)
\log_{p_v} |y_{\stackS}|_v.
$$

Now, using equation~\eqref{equation_definition_retraction_appendix_B} for $\Phi_1$ and $\Phi_2$, we have for every $\rho \in \Sigma(1)$
$$
\Phi_1(y)_{\rho}
=
\Phi_2(y)_{\rho}
\prod_{\stackS \in \twistsector}
y_{\stackS}^{\widetilde{\age}_1(\stackS,[D_{\rho}])-\widetilde{\age}_2(\stackS,[D_{\rho}])}.
$$

Using the exponential map, we can therefore choose such $t = p_v^{-\alpha} \in T_{\NS}(\Q_v)$ which satisfies
$$
t \cdot \Phi_2(y) = \Phi_1(y)
$$
and
$$
\log_{T_{\NS},v}(t) = \alpha .
$$

It follows that
$$
g_2(y)(\lambda_1(L))
=
h_{\stackT,v}(\Phi_1(y))(L)
=
g_1(y)(\lambda_1(L)),
$$
which proves the claim.
\end{proof}

We have the following adelic formula relating these naive local heights $h_{\stackT,v}^{\stack}$ and the coarse multi-height map $h_{\coarse}$:

\begin{prop}\label{proposition_appendix_b_coarse_height_naive_local_height}
    Let $y \in \stackT^{\ext}(\Q)$ and let $P = q^{\ext}(y) \in X(\Q)$ be its image. Then we have for any $(L,\varphi) \in \Pic_{\orb}(X)_{\R}$
    \[
        h_{\coarse}(P)(L) = \sum\limits_{v \in M_{\Q}} \log(p_v) \cdot h_{\stackT,v}^{\stack}(y)(L,\varphi) \, .
    \]
\end{prop}

\begin{proof}
For $\stackS \in \twistsector$, the equality $$h_{\coarse}(P)(0) = \sum\limits_{v \in M_{\Q}} \log(p_v) \cdot h_{\stackT,v}^{\stack}(y)([\stackS]) $$is a consequence of the product formula. For $L \in \Pic(X)_{\R}$, by Proposition~\ref{prop_appendix_B_retraction_torsion_free_case} and Equation~\eqref{equation_local_height_global_height_coarse_space}, we have the equality
\begin{align*}
h_{\coarse}(P)(L) &= \sum\limits_{v \in M_{\Q}} \log(p_v) \cdot h_{\stackT,v}(\Phi(y))(L) \\
&= \sum\limits_{v \in M_{\Q}} \log(p_v) \cdot  h_{\stackT,v}^{\stack}(y)(\lambda(L)) 
\end{align*}from which the proposition follows since $$\Pic_{\orb}(X)_{\R}
=
\lambda\left(\Pic(X)_{\R}\right)
\oplus
\bigoplus_{\stackS \in \twistsector} \R [\stackS].$$
\end{proof}

We conclude this paragraph with the following lemma, which compute the naive local height $h_{\stackT,v}^{\stack}$ for $v \in M_{\Q}^0$.

\begin{lemma}\label{lemma_appendix_b_computation_naive_height_finite_place}
Let  $v \in M_{\Q}^0$. For any $y \in \stackT^{\ext}(\Q_v)$, we have:
$$h_{\stackT,v}^{\stack}(y) = - \age_v(y,-) $$where $\age_v$ is the lift of the age associated to the extended universal torsor (see Definitions \ref{def_age_univ_torsor} and \ref{definition_extended_universal_torsor}). 
\end{lemma}

\begin{proof}
First, by Propositions~\ref{prop_age_and_logarithm} and \ref{proposition_appendix_B_property_equivariant_extended_height}, the map
\[
    y \mapsto h_{\stackT,v}^{\stack}(y) + \age_v(y,-)
\]is $T_{\NS,\orb}(\Q_v)$-invariant. Hence, by Proposition~\ref{prop_surjectivite_local_extended_quotient_description}, it suffices to prove the equality for $y \in \stackT^{\ext}(\Z_v)$.

For $\stackS \in \twistsector$, we have
\[
    h_{\stackT,v}^{\stack}(y)([\stackS]) = \log_{p_v} |y_{\stackS}|_v = - \age_v(y,[\stackS])
\]by Definition~\ref{definition_lift_free_picard_stacky_height} and Proposition~\ref{proposition_computation_residue_map_extended_parametrisation}. Next, for $L \in \Pic(X)$, by Theorem~\ref{theorem_local_degree_local_age} and since $\Phi(y) \in \stackT^{\stack}(\Z)$ by Theorem~\ref{theorem_lien_extended_parametrisation_naive_parametrisation}, we have for any $v \in M_{\Q}^{0}$,
\begin{align*}
h_{\stackT,v}^{\stack}(y)(\lambda(L))&= h_{\stackT,v}\bigl(\Phi(y)\bigr)(L) \\
&= - \age_{v,T_{\NS}}\bigl(\Phi(y),L\bigr) \\
&= - \widetilde{\age}\bigl(\psi_v(P),L\bigr).   
\end{align*}Moreover by Equation \eqref{equation_appendix_B_definition_lambda_map} and Corollary \ref{corollary_computation_extended_parametrization}, we get $\age_v(y,\lambda(L)) = \widetilde{\age}\bigl(\psi_v(P),L\bigr)$ so we can conclude the proof.

\end{proof}

\subsection{Boundedness result for this local height}\label{section_appendix_B_boundedness}

We shall also need a boundedness result, stating that if $\mathrm{D} \subset
\Pic_{\orb}(X)^{\vee}_{\R}$ is a compact subset, then
$h^{\stack}_{\stackT,\infty}{}^{-1}(\mathrm{D})$ is bounded. To this end, we first need to study
the real points of our stack. 

\subsubsection{Compactness of the real points of our stack}

In this paragraph only, $X$ denotes a general toric stack, the assumption that $\Pic(X)$ is torsion-free is not required. We equip $X(\R)$ with its real analytic topology (see
\cite[Definition~7.5]{de_gaay_definition_real_topology}).

\begin{lemma}\label{lemma_compacity_real_points_stacks}
The space $X(\R)$ is compact for the real analytic topology.
\end{lemma}

\begin{proof}
By \cite[Proposition~2.1]{ambrosi2025topologyrealalgebraicstacks}, the coarse moduli map
on real points
\[
p : X(\R) \longrightarrow X^{\coarse}(\R)
\]
satisfies that $p^{-1}(P)$ is quasi-compact for any
$P \in X^{\coarse}(\R)$, and moreover $p$ is closed, since this holds locally on a
covering by open subsets. Indeed, for any $\sigma \in \Sigma_{\max}$, the map
\[
[\affine^{d}/N(\sigma)](\R)
\longrightarrow
\affine^{d}/N(\sigma)(\R)
\]
is closed by \cite[Lemma~4.8]{ambrosi2025topologyrealalgebraicstacks}. Hence $p$ is
proper. Since $X^{\coarse}(\R)$ is compact for the real analytic topology, it follows
that $X(\R)$ is compact.
\end{proof}

\subsubsection{The boundedness result}

We again assume that $\Pic(X)$ is torsion-free. Let $\mathrm{C}$ be a compact subset of
$\Pic(X)^{\vee}_{\R}$. We then have the following proposition.

\begin{prop}\label{prop_appendix_B_upper_bound_non_smooth_case}
The set $h_{\stackT,\infty}^{-1}(\mathrm{C})$ is compact.
\end{prop}

\begin{proof}
The arguments used in the proofs of
\cite[Proposition 3.20, Corollary~3.21, Lemma~4.3]{bongiorno2024multiheightanalysisrationalpoints}
apply verbatim in the present setting, using the previous lemma.
\end{proof}

We can now state the boundedness result announced above:

\begin{theorem}\label{theorem_appendix_B_compacity_pull-back-compact_ext_height}
    Let $\mathrm{D} \subset \Pic_{\orb}(X)^{\vee}_{\R}$ be a compact set. Then $(h^{\stack}_{\stackT,\infty})^{-1}(\mathrm{D})$ is bounded.
\end{theorem}

\begin{proof}
Let $K > 0$ be a constant such that $\mathrm{D}$ is contained in the ball of radius $K$
centered at the origin for the supremum norm. Then, for any
$\stackS \in \twistsector$, we have
\[
\bigl| h^{\stack}_{\stackT,\infty}(y)\bigl([\stackS]\bigr) \bigr| \leqslant K,
\]
and therefore
\[
|y_{\stackS}| \leqslant e^{K}.
\]
Moreover, by Proposition~\ref{prop_appendix_B_upper_bound_non_smooth_case}, the coordinates
$\Phi(y)_{\rho}$ for $\rho \in \Sigma(1)$ are uniformly bounded. Using the explicit
description of the map $\Phi$ (see Equation~\eqref{equation_definition_retraction_appendix_B}), together with
the fact that for any $\stackS \in \twistsector$ we have $$1 \leqslant |y_{\stackS}| \leqslant e^K$$
we deduce that the coordinates $y_{\rho}$ for $\rho \in \Sigma(1)$ are also uniformly
bounded. It follows that $(h^{\stack}_{\stackT,\infty})^{-1}(\mathrm{D})$ is bounded.
\end{proof}

\section{Lift of the stacky height to the extended universal torsor}

We now return to the general case where $X$ is a toric stack without any restrictive assumption on $\Pic(X)$. Recall that its coarse moduli space $X^{\coarse}$ is still endowed with its natural system of heights, while the choice of an adelic stacky data (see Definition~\ref{definition_adelic_stacky_data}) is given, for any $(L,\varphi)$, by the collection $\{\varphi_v\}_{v \in M_{\Q}^0}$ defined by
\[
    \varphi_v(P) = \varphi\bigl(\psi_v(P)\bigr).
\]In the previous Section~\ref{appendix_B}, we explained how to construct a lift of the coarse height to the extended universal torsor in the case of a toric stack with torsion-free Picard group. We will use the construction developed in that section (see Definition~\ref{definition_lift_free_picard_stacky_height}) to define a lift of the stacky height to the extended universal torsor.

\subsection{Lifting the morphism from the toric stack to its canonical stack to the extended universal torsor}

Let $v \in M_{\Q}$. To define part of the extended local height, we need to use the canonical stack
of $X$, denoted by $X^{\can}$. By Remark~\ref{remark_picard_group_canonical_stack}, we
know that the Picard group of $X^{\can}$ is torsion free. We can therefore apply the results obtained in Section~\ref{appendix_B}. We shall use the naive local
height
\[
h^{\stack,\can}_{\stackT,v} : \stackT^{\ext,\can}(\R) \longrightarrow \Pic_{\orb}(X^{\can})^{\vee}_{\R}
\]
defined in Definition~\ref{definition_lift_free_picard_stacky_height} in order to define part of
the extended local height
\[
h_{\stackT,v}^{\ext} : \stackT^{\ext}(\R) \longrightarrow \Pic_{\orb}(X)^{\vee}_{\R}.
\]

The aim of this paragraph is to explain how the natural morphism
\[
X \xrightarrow{p_{\can}} X^{\can}
\]
lifts to a morphism between the extended universal torsor of $X$ and $X^{\can}$
\[
\stackT \times \G_m^{\twistsector}
\xrightarrow{r}
\stackT \times \G_m^{\pi_0^*(\I_{\mu} X^{\can})},
\]
which is equivariant with respect to the action of $T_{\NS,\orb}$ on
$\stackT \times \G_m^{\twistsector}$ and the action of $T^{\can}_{\NS,\orb}$, the orbifold Néron--Severi
torus of $X^{\can}$, on $\stackT \times \G_m^{\pi_0^*(\I_{\mu} X^{\can})}$.

We have seen in Corollaries~\ref{cor_description_map_stack_to_rigidification} and
\ref{cor_description_map_stack_to_canonical} that the morphism
from $X$ to its canonical stack $X^{\can}$ lifts to the universal torsor via a map
\[
\stackT \longrightarrow \stackT
\]
given by
\[
y \longmapsto \bigl( y_{\rho}^{a_{\rho}} \bigr)_{\rho \in \Sigma(1)},
\]
where the integers $a_{\rho} \in \N^{*}$ are defined by the condition that, for every
$\rho \in \Sigma(1)$,
\[
\overline{\beta(e_{\rho})} = a_{\rho} \cdot \beta^{\can}(e_{\rho}) = a_{\rho} u_{\rho}
\]where $u_{\rho}$ is the minimal vector of the edge $\rho \in \Sigma(1)$ (see Notations \ref{notation_stacky_fan_beta}). We also denote by $p_{\can}$ the induced map
$$
\sector \longrightarrow \pi_0(\I_{\mu} X^{\can}) \, .
$$We will need the following proposition.

\begin{prop}\label{prop_computation_box_element_canonicalisation}
For every $\stackS \in \twistsector$, let $b_{\stackS} \in \Boxe(\Sigma)$ denote the corresponding element. Then we have
$$
\overline{b}_{\stackS}
=
\sum_{\rho \in \sigma_\stackS(1)}
\left\lfloor
a_{\rho}\,\age_{\num}(\stackS,-[D_{\rho}])
\right\rfloor
u_{\rho}
+
b_{p_{\can}(\stackS)}
$$
in $N / N_{\tor}$, where $b_{p_{\can}(\stackS)}$ denotes the element of $\Boxe^{\can}(\Sigma)$ corresponding to $$p_{\can}(\stackS) \in \pi_0(\I_{\mu} X^{\can}).$$
\end{prop}

We denote by $D_{\rho}^{\can}$ the boundary divisor of $X^{\can}$ associated with $\rho \in \Sigma(1)$. Recall that we have
$$
p_{\can}^*[D^{\can}_{\rho}] = a_{\rho}[D_{\rho}] .
$$

\begin{proof}
We need to show that for every $\rho \in \Sigma(1)$,
$$
\{a_{\rho}\,\age_{\num}(\stackS,-[D_{\rho}]) \}
=
\age_{\num}(p_{\can}(\stackS),-[D^{\can}_{\rho}]) \, .
$$
For this, it suffices to prove the equality in $\Q/\Z$. We compute
\begin{align*}
a_{\rho}\,\age(\stackS,-[D_{\rho}])
&=
\age(\stackS,-a_{\rho}[D_{\rho}]) \\
&=
\age(\stackS,p_{\can}^*(-[D^{\can}_{\rho}])) \\
&=
\age(p_{\can}(\stackS),-[D^{\can}_{\rho}]),
\end{align*}
where the last equality follows from the functoriality property of the age pairing (see Proposition~\ref{proposition_age_pairing_morphism_functoriality}).
\end{proof}

We define the morphism
$$
f : \Z^{\Sigma(1) \cup \twistsector}
\longrightarrow
\Z^{\Sigma(1) \cup \pi_0^*(\I_{\mu} X^{\can})}
$$
by setting, for $\rho \in \Sigma(1)$,
\[
f(e_{\rho}) = a_{\rho}\, e_{\rho},
\]
and for $\stackS \in \twistsector$,
\[
f(e_{\stackS}) =
\begin{cases}
\displaystyle
\sum_{\rho \in \sigma_\stackS(1)}
a_{\rho}\,
\age_{\num}\!\left(\stackS,-[D_{\rho}]\right)\, e_{\rho},
& \text{if } p_{\can}(\stackS)=0, \\[1.6em]
\displaystyle
\sum_{\rho \in \sigma_\stackS(1)}
\left\lfloor
a_{\rho}\,
\age_{\num}\!\left(\stackS,-[D_{\rho}]\right)
\right\rfloor e_{\rho}
+
e_{p_{\can}(\stackS)},
& \text{if } p_{\can}(\stackS)\in \pi_0^*(\I_{\mu} X^{\can}).
\end{cases}
\]
By Proposition~\ref{prop_computation_box_element_canonicalisation}, we obtain the following commutative diagram:
\begin{equation}\label{equation_diagram_lift_canonical_map_extended_torsor}
\begin{tikzcd}
\Z^{\Sigma(1) \cup \twistsector} \arrow[r, "\beta^{\ext}"] \arrow[d, "f"']      & N \arrow[r] \arrow[d, two heads] & \coneform(\beta^{\ext}) \arrow[d] \\
\Z^{\Sigma(1) \cup \pi_0^*(\I_{\mu} X^{\can})} \arrow[r, "\beta^{\ext}_{\can}"] & N / N_{\tor} \arrow[r]           & \coneform(\beta^{\ext}_{\can}) \, .   
\end{tikzcd}
\end{equation}

The map $f$ defines a morphism
\[
\stackT \times \G_m^{\twistsector}
\xrightarrow{r}
\stackT \times \G_m^{\pi_0^*(\I_{\mu} X^{\can})},
\]equivariant for the morphism of tori
\[
T_{\NS,\orb} \xrightarrow{\phi} T^{\can}_{\NS,\orb},
\]induced by the previous diagram, such that the coordinates of $r(y)$ are given by
\begin{equation}\label{equation_description_r}
\begin{cases}
\displaystyle
r(y)_{\rho}
=
y_{\rho}^{a_{\rho}}
\prod_{\stackS \in \twistsector}
y_{\stackS}^{\left\lfloor
a_{\rho}\,\age_{\num}\!\left(\stackS,-[D_{\rho}]\right)
\right\rfloor},
& \text{if } \rho \in \Sigma(1), \\[1.6em]

\displaystyle
r(y)_{\stackZ}
=
\prod_{p_{\can}(\stackS)=\stackZ}
y_{\stackS},
& \text{if } \stackZ \in \pi_0^*(\I_{\mu} X^{\can}).
\end{cases}
\end{equation}

Let $\lambda_{\phi}$ denotes the morphism of character groups induced by $\phi$, that is: \[
\Pic_{\orb}(X^{\can}) = \mathbf{R^1Hom}(\coneform(\beta^{\ext}_{\can}) ,\Z)
\xrightarrow{\lambda_{\phi}}
\Pic_{\orb}(X) = \mathbf{R^1Hom}(\coneform(\beta^{\ext}) ,\Z) \, .
\]We have the following lemma.

\begin{lemma}
The map $\lambda_{\phi}$ is given by the pullback on orbifold Picard groups induced by $p_{\can}$, that is to say
$$
p_{\can}^* : \Pic_{\orb}(X^{\can}) \longrightarrow \Pic_{\orb}(X),
$$
such that for every $(L,\varphi) \in \Pic_{\orb}(X^{\can})$,
$$
p_{\can}^*(L,\varphi) = \left(p_{\can}^*L,\varphi \circ p_{\can}\right).
$$
\end{lemma}

\begin{proof}
Since $\beta^{\ext}$ and $\beta^{\ext}_{\can}$ are surjective, the morphism $\lambda_{\phi}$ is simply the horizontal arrow in the following commutative diagram induced by $f$:
\begin{center}
\begin{tikzcd}
\Z^{\Sigma(1) \cup \pi_0^*(\I_{\mu} X^{\can})} \arrow[r, "f^{\vee}"] \arrow[d] & \Z^{\Sigma(1) \cup \twistsector} \arrow[d] \\
\Pic_{\orb}(X^{\can}) = \ker(\beta^{\ext}_{\can})^{\vee} \arrow[r, "\lambda_{\phi}"] & \Pic_{\orb}(X) = \ker(\beta^{\ext})^{\vee}
\end{tikzcd}
\end{center}
where the vertical arrows are the usual projections.

Moreover, since $X^{\can}$ is a toric orbifold, the canonical map
$$
\Z^{\Sigma(1) \cup \pi_0^*(\I_{\mu} X^{\can})}
\longrightarrow
\Pic_{\orb}(X^{\can})
$$
is surjective. It therefore suffices to show that $\lambda_{\phi}$ and $p_{\can}^*$ coincide on
$$
\{ [D_{\rho}^{\can}]_0 \mid \rho \in \Sigma(1) \}
\cup
\{ [\stackZ] \mid \stackZ \in \pi_0^*(\I_{\mu} X^{\can}) \}
$$where $[D_{\rho}^{\can}]_0$ is the orbifold line bundle defined in Notations \ref{notation_orbifold_T_invariant_divisor}.

Now $f^{\vee}$ is computed as follows:
\[
f^{\vee}(e_{\rho}) =
\begin{cases}
\displaystyle
a_{\rho} e_{\rho}
+
\sum_{\stackS \in \twistsector}
\left\lfloor
a_{\rho}\,\age_{\num}\!\left(\stackS,-[D_{\rho}]\right)
\right\rfloor e_{\stackS},
& \text{if } \rho \in \Sigma(1), \\[1.6em]

\displaystyle
\sum_{p_{\can}(\stackS)=\stackZ}
e_{\stackS},
& \text{if } \rho = \stackZ \in \pi_0^*(\I_{\mu} X^{\can}).
\end{cases}
\]

Hence, for every $\rho \in \Sigma(1)$, we obtain
$$
\lambda_{\phi}([D_{\rho}^{\can}]_0)
=
\left(
a_{\rho}[D_{\rho}],
-
\{a_{\rho}\,\age_{\num}\!\left(-,-[D_{\rho}]\right)\}
\right)
=
p_{\can}^*([D_{\rho}^{\can}]_0),
$$where the last equality is a consequence of the proof of Proposition~\ref{prop_computation_box_element_canonicalisation}
and for every $\stackZ \in \pi_0^*(\I_{\mu} X^{\can})$,
$$
\lambda_{\phi}([\stackZ])
=
\sum_{p_{\can}(\stackS)=\stackZ} [\stackS]
=
p_{\can}^*([\stackZ]) \, .
$$
    
\end{proof}

From now on, we will write $p_{\can}^*$ instead of $\lambda_{\phi}$. We now explain how the map $r$ lifts the map $p_{\can}$ to the extended universal torsor.

\begin{prop}\label{proposition_lift_of_canonical_map_to_extended_univ_torsor}
The following diagram is $2$-commutative:
\begin{center}
\begin{tikzcd}
\stackT \times \G_m^{\twistsector} \arrow[r, "r"] \arrow[d, "q^{\ext}"] 
& \stackT \times \G_m^{\pi_0^*(\I_{\mu} X^{\can})} \arrow[d, "q^{\ext}_{\can}"] \\
X \arrow[r, "p_{\can}"]                                                 
& X^{\can}
\end{tikzcd}
\end{center}
\end{prop}

\begin{proof}
We have the following commutative diagram:
\begin{center}
\begin{tikzcd}
\stackT \times \G_m^{\twistsector} \arrow[r, "r"] & \stackT \times \G_m^{\pi_0^*(\I_{\mu} X^{\can})} \\
\stackT \arrow[r, "\wedge a"] \arrow[u]           & \stackT \arrow[u]
\end{tikzcd}
\end{center}
where the vertical arrows are given by the inclusion $y \mapsto (y,1)$. Passing to the quotient, and using Corollaries~\ref{cor_description_map_stack_to_rigidification} and
\ref{cor_description_map_stack_to_canonical} together with Theorem~\ref{theorem_extended_universal_torsor}, we obtain the following $2$-commutative diagram:
\begin{center}
\begin{tikzcd}
{[ \stackT \times \G_m^{\twistsector} / T_{\NS,\orb} ]} \arrow[r, "\overline r"] 
& {[ \stackT \times \G_m^{\pi_0^*(\I_{\mu} X^{\can})} / T_{\NS,\orb}^{\can} ]} \\
{[ \stackT / T_{\NS} ]} \arrow[r, "p_{\can}"] \arrow[u, equal]             
& {[ \stackT / T_{\NS,\can} ]} \arrow[u, equal]
\end{tikzcd}
\end{center}
where $\overline r$ denotes the quotient morphism induced by $r$. This completes the proof.
\end{proof}

%
%
%

We now explain why the construction of $h_{\stackT,v}^{\stack}$ naturally proceeds via
this lifting.

\subsection{The naive local height in the general case}

To define the naive local height $h_{\stackT,v}^{\stack}$ in the general case, we use the morphism $r$ from the previous
paragraph together with the fact that the real vector space generated by
$\Pic_{\orb}(X)$ decomposes as
\[
\Pic_{\orb}(X)_{\R}
=
\Pic(X)_{\R} \oplus \R^{\twistsector}.
\]
Moreover, the pullback induced by
\[
p_{\can} : X \longrightarrow X^{\can}
\]
defines an isomorphism of real vector spaces
\[
p_{\can}^{*} : \Pic(X^{\can})_{\R} \longrightarrow \Pic(X)_{\R}.
\]Thus, we will actually replace $\Pic(X)_{\R}$ by
$p_{\can}^{*}\!\left(\Pic(X^{\can})_{\R}\right)$ and use $$\Pic_{\orb}(X)_{\R}
=
p_{\can}^{*}\!\left(\Pic(X^{\can})_{\R}\right) \oplus \R^{\twistsector} .$$

\begin{definition}\label{definition_naive_local_height_extended_universal_torsor}
For $v \in M_{\Q}$, we define a naive local height on the extended universal torsor
\[
h_{\stackT,v}^{\stack} : \stackT^{\ext}(\Q_v) \longrightarrow \Pic_{\orb}(X)^{\vee}_{\R}
\]by requiring that the following conditions hold:
\begin{itemize}
\item
for any $\stackS \in \twistsector$,
\[
h_{\stackT,v}^{\stack}(y)\bigl([\stackS]\bigr)
=
\log_{p_v} |y_{\stackS}|_v \, ,
\]

\item
for any $L \in \Pic(X^{\can})_{\R}$,
\[
h_{\stackT,v}^{\stack}(y)\bigl((p_{\can}^*(L),0)\bigr)
=
h_{\stackT,v}^{\stack,\can}\!\left(r(y)\right)((L,0))
\]
\end{itemize}where $h_{\stackT,v}^{\stack,\can}$ corresponds to the naive local height for $X^{\can}$ defined in \ref{definition_lift_free_picard_stacky_height}.
\end{definition}

\begin{remark}
    The previous definition of $h_{\stackT,v}^{\stack}$ coincides with that of Definition~\ref{definition_lift_free_picard_stacky_height} in the case where $\Pic(X)$ is torsion free. For $v \in M_{\Q}^0$, this is a consequence of Lemma~\ref{lemma_appendix_b_computation_naive_height_finite_place} and Lemma~\ref{lemma_general_computation_naive_height_finite_place}. For $v = \infty$, this follows from the previous assertion and from Propositions~\ref{proposition_appendix_b_coarse_height_naive_local_height} and \ref{proposition_coarse_height_naive_local_height}.
\end{remark}

Before studying the properties of this naive local height, we first show that it satisfies a transitivity property with respect to the morphism $r$.

\begin{prop}\label{proposition_functoriality_naive_canonical_stack}
Let
\[
p_{\can}^* : \Pic_{\orb}(X^{\can}) \longrightarrow \Pic_{\orb}(X)
\]
denote the pullback on orbifold Picard groups. Then for every $y \in \stackT^{\ext}(\Q_v)$ and for every $(L,\varphi) \in \Pic_{\orb}(X^{\can})_{\R}$ we have
\[
h_{\stackT,v}^{\stack}(y)\bigl(p_{\can}^*(L,\varphi)\bigr)
=
h_{\stackT,v}^{\stack,\can}\!\left(r(y)\right)(L,\varphi) \, .
\]
\end{prop}

\begin{proof}
Indeed, we have
\begin{align*}
h_{\stackT,v}^{\stack}(y)\bigl(p_{\can}^*(L,\varphi)\bigr)
&=
h_{\stackT,v}^{\stack}(y)\bigl((p_{\can}^*(L),0)\bigr)
+
\sum_{\stackS \in \twistsector}
\varphi(p_{\can}(\stackS))
\, h_{\stackT,v}^{\stack}(y)\bigl([\stackS]\bigr) \\
&=
h_{\stackT,v}^{\stack,\can}\!\left(r(y)\right)((L,0))
+
\sum_{\stackZ \in \pi_0^*(\I_{\mu} X^{\can})}
\sum_{p_{\can}(\stackS)=\stackZ}
\varphi(\stackZ)
\, h_{\stackT,v}^{\stack}(y)\bigl([\stackS]\bigr) \\
&=
h_{\stackT,v}^{\stack,\can}\!\left(r(y)\right)((L,0))
+
\sum_{\stackZ \in \pi_0^*(\I_{\mu} X^{\can})}
\varphi(\stackZ)
\sum_{p_{\can}(\stackS)=\stackZ}
\log_{p_v} |y_{\stackS}|_v \\
&=
h_{\stackT,v}^{\stack,\can}\!\left(r(y)\right)((L,0))
+
\sum_{\stackZ \in \pi_0^*(\I_{\mu} X^{\can})}
\varphi(\stackZ)
\, h_{\stackT,v}^{\stack}(r(y))\bigl([\stackZ]\bigr) \\
&= h_{\stackT,v}^{\stack,\can}\!\left(r(y)\right)(L,\varphi) \, ,
\end{align*}
where the penultimate equality follows from equation~\eqref{equation_description_r}.
\end{proof}

We can now compute $h_{\stackT,v}^{\stack}$ for $v \in M_{\Q}^0$, a finite place.

\begin{lemma}\label{lemma_general_computation_naive_height_finite_place}
For any $y \in \stackT^{\ext}(\Q_v)$, we have:
$$h_{\stackT,v}^{\stack}(y) = - \age_v(y,-) $$where $\age_v$ is the lift of the age associated to the extended universal torsor. 
\end{lemma}

\begin{proof}
First, by Propositions~\ref{prop_age_and_logarithm} and \ref{proposition_appendix_B_property_equivariant_extended_height}, the map
\[
    y \mapsto h_{\stackT,v}^{\stack}(y) + \age_v\left(y,-\right)
\]
is $T_{\NS,\orb}(\Q_v)$-invariant. Hence, by Proposition~\ref{prop_surjectivite_local_extended_quotient_description}, it suffices to prove the equality for $y \in \stackT^{\ext}(\Z_v)$.

For $\stackS \in \twistsector$, we have
\[
    h_{\stackT,v}^{\stack}(y)\left([\stackS]\right) = \log_{p_v} \left|y_{\stackS}\right|_v = - \age_v\left(y,[\stackS]\right)
\]
by Definition~\ref{definition_naive_local_height_extended_universal_torsor} and Proposition~\ref{proposition_computation_residue_map_extended_parametrisation}. Next, for $\left(L,\varphi\right) \in \Pic_{\orb}(X^{\can})_{\R}$, by  Proposition \ref{proposition_functoriality_naive_canonical_stack} and Lemma~\ref{lemma_appendix_b_computation_naive_height_finite_place} we have
\begin{align*}
h_{\stackT,v}^{\stack}(y)\left(p_{\can}^*\left(L,\varphi\right)\right)
&= h_{\stackT,v}^{\stack,\can}\left(r(y)\right)\left(L,\varphi\right) \\
&= - \age_{v}\left(r(y),\left(L,\varphi\right)\right) \\
&= - \age_{v}\left(y,p_{\can}^*\left(L,\varphi\right)\right).
\end{align*}
where the last equality is a consequence of Proposition~\ref{prop_age_functoriality_torsor}. We can conclude the proof.

\end{proof}

We now turn to a property of equivariance of this naive local heights $h_{\stackT,v}^{\stack}$ with respect to the action of $T_{\NS,\orb}(\Q_v)$.

\begin{prop}\label{proposition_property_equivariant_naive_height}
Let $v \in M_{\Q}$.

For any $t \in T_{\NS,\orb}(\Q_v)$ and any $y \in \stackT^{\ext}(\Q_v)$, we have
\[
h_{\stackT,v}^{\stack}(t \cdot y)
=
h_{\stackT,v}^{\stack}(y)
+
\log_{T_{\NS,\orb},v}(t).
\]
\end{prop}

\begin{proof}
By the description of the action of $T_{\NS,\orb}$ on the extended universal torsor (see Definition \ref{definition_extended_universal_torsor}), we have for any $\stackS \in \twistsector$,
    \begin{align*}
        h_{\stackT,v}^{\stack}(t.y)([\stackS]) &=  \log_{p_v} | y{}_{\stackS} |_v + \log_{p_v} |[\stackS](t)^{-1}|_v \\
        &= h_{\stackT,v}^{\stack}\left(y\right)([\stackS]) + \log_{T_{\NS,\orb},v}(t)([\stackS])
    \end{align*}Moreover for any $(L,\varphi) \in \Pic_{\orb}(X^{\can})_{\R}$,
    \begin{align*}
        h_{\stackT,v}^{\stack}(t.y)(p_{\can}^*(L,\varphi)) &= h_{\stackT,v}^{\stack,\can}(r(t.y))(L,\varphi) \\
        &= h_{\stackT,v}^{\stack,\can}\left(\phi(t).r(y)\right)(L,\varphi)  \\
        &= h_{\stackT,v}^{\stack,\can}\left(r(y)\right)(L,\varphi) + \log_{T_{\NS,\orb}^{\can},v}(\phi(t))(L,\varphi) \\
        &= h_{\stackT,v}^{\stack}\left(y\right)(p_{\can}^*(L,\varphi)) + \log_{T_{\NS,\orb},v}(t)(p_{\can}^*(L,\varphi)) 
    \end{align*}where the penultimate equality is a consequence of Proposition \ref{proposition_appendix_B_property_equivariant_extended_height}.
\end{proof}

To conclude this paragraph, we relate these naive local heights $h_{\stackT,v}$ to the coarse multi-height map $h_{\coarse}$:

\begin{prop}\label{proposition_coarse_height_naive_local_height}
    Let $y \in \stackT^{\ext}(\Q)$ and let $P = q^{\ext}(y) \in X(\Q)$ be its image. Then we have for any $(L,\varphi) \in \Pic_{\orb}(X)_{\R}$
    \[
        h_{\coarse}(P)(L) = \sum\limits_{v \in M_{\Q}} \log(p_v) \cdot h_{\stackT,v}^{\stack}(y)(L,\varphi) \, .
    \]
\end{prop}

\begin{proof}
For $\stackS \in \twistsector$, the equality $$h_{\coarse}(P)(0) = \sum\limits_{v \in M_{\Q}} \log(p_v) \cdot h_{\stackT,v}^{\stack}(y)([\stackS]) $$is a consequence of the product formula. For $(L,\varphi) \in \Pic_{\orb}(X^{\can})_{\R}$, by Definition \ref{definition_coarse_multi_height_map} of the coarse height, we have $$h_{\coarse}(P)(p_{\can}^*L) = h_{\coarse}^{\can}(p_{\can}(P))(L) .$$Then by Propositions~\ref{proposition_lift_of_canonical_map_to_extended_univ_torsor} and \ref{proposition_appendix_b_coarse_height_naive_local_height}, we have the equality
\begin{align*}
h_{\coarse}(P)(p_{\can}^*L) &= h_{\coarse}^{\can}(p_{\can}(P))(L) \\
&= \sum\limits_{v \in M_{\Q}} \log(p_v) \cdot h_{\stackT,v}^{\stack,\can}(r(y))(L,\varphi) \\
&= \sum\limits_{v \in M_{\Q}} \log(p_v) \cdot  h_{\stackT,v}^{\stack}(y)(p_{\can}^*(L,\varphi)) 
\end{align*}from which the proposition follows since $\Pic_{\orb}(X)_{\R}$ is the sum of the subspaces $p_{\can}^*\left(\Pic_{\orb}(X^{\can})_{\R}\right) $ and $\bigoplus\limits_{\stackS \in \twistsector} \R [\stackS] $.

\end{proof}

\subsection{The extended local height}

We can now define the extended local height $h_{\stackT,v}^{\ext}$ using the naive local height defined in Definition~\ref{definition_naive_local_height_extended_universal_torsor}.

\begin{definition}\label{definition_local_height_extended_universal_torsor}
For $v \in M_{\Q}$, we define a local height on the extended universal torsor
\[
h_{\stackT,v}^{\ext} : \stackT^{\ext}(\Q_v) \longrightarrow \Pic_{\orb}(X)^{\vee}_{\R}
\]such that for any $y \in \stackT^{\ext}(\Q_v)$, we set
\[
h_{\stackT,v}^{\ext}(y)(L,\varphi)
=
\begin{cases}
h_{\stackT,v}^{\stack}(y)(L,\varphi)
+
\varphi\!\left(\psi_v\!\left(q^{\ext}(y)\right)\right)
& \text{if } v \in M_{\Q}^0, \\[0.3em]
h_{\stackT,\infty}^{\stack}(y)(L,\varphi)
& \text{if } v = \infty.
\end{cases}
\]The associated exponential height $H^{\ext}_{\stackT,v}$ is defined by the condition that, for any $\theta \in \Pic_{\orb}(X)_{\R}$,
\[
H^{\ext}_{\stackT,v}(y,\theta)
=
q_v^{\,h^{\ext}_{\stackT,v}(y)(\theta)}.
\]
\end{definition}

\begin{remark}\label{remark_forme_classique_extended_local_height_finite_place}
    By Lemma \ref{lemma_general_computation_naive_height_finite_place}, we get that for a finite place $v \in M_{\Q}^0$ and for any $y \in \stackT^{\ext}(\Q_v)$,
    $$h_{\stackT,v}^{\ext}(y)(L,\varphi) = - \age_v\left(y,(L,\varphi)\right) + \varphi\!\left(\psi_v\!\left(q^{\ext}(y)\right)\right) \, .$$
\end{remark}

In particular, we recover the following property of local heights in the classical case (see \cite[Lemma~4.30 (ii)]{Peyre_beyond_height}):

\begin{prop}\label{proposition_forme_classique_parametrisation_entiere_locale}
For every $v \in M_{\Q}^0$ we have
$$
\stackT^{\ext}(\Z_v)
=
\{ y \in \stackT^{\ext}(\Q_v) \mid h_{\stackT,v}^{\ext}(y) = 0 \} \, .
$$
\end{prop}

\begin{proof}
    This is a consequence of Remark \ref{remark_forme_classique_extended_local_height_finite_place} and Corollary~\ref{corollary_computation_extended_parametrization}.
\end{proof}

By Proposition \ref{proposition_property_equivariant_naive_height}, this extended local height $h_{\stackT,v}^{\ext}$ also have a property of equivariance with respect to the action of $T_{\NS,\orb}(\Q_v)$.

\begin{prop}\label{proposition_property_equivariant_extended_height}
Let $v \in M_{\Q}$.

For any $t \in T_{\NS,\orb}(\Q_v)$ and any $y \in \stackT^{\ext}(\Q_v)$, we have
\[
h_{\stackT,v}^{\ext}(t \cdot y)
=
h_{\stackT,v}^{\ext}(y)
+
\log_{T_{\NS,\orb},v}(t).
\]
\end{prop}

We now define an adelic space $\stackT^{\ext}(\affine_{\Q})$ together with an adelic height on it as follows.

\begin{definition}
We define $\stackT^{\ext}(\affine_{\Q})$ to be the set of elements
$$
(y_v) \in \prod_{v \in M_{\Q}} \stackT^{\ext}(\Q_v)
$$
such that the set
$$
\{\, v \in M_{\Q}^0 \mid y_v \notin \stackT^{\ext}(\Z_v) \,\}
$$
is finite.

We then define
$$
h_{\stackT}^{\ext} : \stackT^{\ext}(\affine_{\Q}) \longrightarrow \Pic_{\orb}(X)^{\vee}_{\R}
$$
by setting, for $y \in \stackT^{\ext}(\affine_{\Q})$,
$$
h_{\stackT}^{\ext}(y)
=
\sum_{v \in M_{\Q}}
\log(p_v)\,
h^{\ext}_{\stackT,v}(y_v) \, .
$$
\end{definition}

By Proposition~\ref{proposition_forme_classique_parametrisation_entiere_locale}, the height $h_{\stackT}^{\ext}$ is well defined. We can establish the following proposition:

\begin{prop}\label{proposition_adelic_lift_height_univ_torsor}
For any $y \in \stackT^{\ext}(\Q)$, if $P = q^{\ext}(y) \in X(\Q)$, then
\[
h^{\ext}_{\stackT}(y)
=
h(P).
\]
\end{prop}

\begin{remark}
This formula for the lift of the stacky height to the adelic points of the extended universal torsor is useful when studying height zeta functions. In particular if one wants to use harmonic analysis methods on the extended universal torsor, as introduced by Santens
in~\cite{santens2025maninsconjectureintegralpoints} for universal torsor.
\end{remark}

\begin{proof}
This is a consequence of Proposition \ref{proposition_coarse_height_naive_local_height} and Definitions \ref{definition_stacky_height} and \ref{definition_local_height_extended_universal_torsor}.    
\end{proof}

To study the distribution of rational points on the stack using methods from the geometry of numbers, a crucial point is that this formula simplifies in the case where one considers an element of the integral parametrization $\stackT^{\ext}(\Z)$. Indeed, in this case, the height of a rational point lifts to the extended local height at the infinite place as follows:

\begin{theorem}\label{theorem_lift_height_extend_universal_torsor}
    For any $y \in \stackT^{\ext}(\Z)$, if we write $P = q^{\ext}(y)$, then:
    $$h(P) = h^{\ext}_{\stackT,\infty}(y) .$$
\end{theorem}

\begin{proof}
    This is a consequence of Propositions \ref{proposition_adelic_lift_height_univ_torsor} and \ref{proposition_forme_classique_parametrisation_entiere_locale}.
\end{proof}

\subsection{Additional property of the extended local height}

We shall also need a boundedness result, stating that if $\mathrm{D} \subset
\Pic_{\orb}(X)^{\vee}_{\R}$ is a compact subset, then
$h^{\ext}_{\stackT,\infty}{}^{-1}(\mathrm{D})$ is bounded. To this end, we will use the results of Section~\ref{section_appendix_B_boundedness}.

\begin{theorem}\label{theorem_compacity_pull-back-compact_ext_height}
    Let $\mathrm{D} \subset \Pic_{\orb}(X)^{\vee}_{\R}$ be a compact set. Then $(h^{\ext}_{\stackT,\infty})^{-1}(\mathrm{D})$ is bounded.
\end{theorem}

\begin{proof}
Let $K > 0$ be a constant such that $\mathrm{D}$ is contained in the ball of radius $K$
centered at the origin for the supremum norm. Then, for any
$\stackS \in \twistsector$, we have
\[
\bigl| h^{\ext}_{\stackT,\infty}(y)\bigl([\stackS]\bigr) \bigr| \leqslant K,
\]
and therefore
\[
|y_{\stackS}| \leqslant e^{K}.
\]
Moreover, by Proposition \ref{proposition_functoriality_naive_canonical_stack}:
$$h_{\stackT,\infty}^{\stack,\can}(r(y))(-) = h_{\stackT,\infty}^{\stack}(y)\left(p_{\can}^{*}(-)\right) = h_{\stackT,\infty}^{\ext}(y)\left(p_{\can}^{*}(-)\right) \in (p_{\can}^{*})^{\vee}(\mathrm{D})  \, .$$
Thus, because $(p_{\can}^{*})^{\vee}(\mathrm{D})$ is compact,  by Theorem~\ref{theorem_appendix_B_compacity_pull-back-compact_ext_height}, the coordinates
$r(y)_{\rho}$ for $\rho \in \Sigma(1)$ are uniformly bounded. Using the explicit
description of the map $r$ (see Equation~\eqref{equation_description_r}), together with
the fact that for any $\stackS \in \twistsector$ we have $$1 \leqslant |y_{\stackS}| \leqslant e^{K},$$we deduce that the coordinates $y_{\rho}$ for $\rho \in \Sigma(1)$ are also uniformly
bounded. It follows that $(h^{\ext}_{\stackT,\infty})^{-1}(\mathrm{D})$ is bounded.
\end{proof}

\section{Tamagawa measures on toric stacks}

\subsection{Möbius function associated to the extended universal torsor}

In this section, we introduce the Möbius function associated to the extended universal torsor $\stackT^{\ext}(\Z)$ which parametrizes $X(\Q)$. We use the conventions of Notations \ref{notation_extended_torus_fan}.

\begin{Notations}
\begin{enumerate}
    \item $\mathbf{1}_{\stackT^{\ext}(\Z)} : (\N^*)^{\Sigma_{\ext}(1)} \rightarrow \{0,1\}$ is the indicator function of $$\stackT^{\ext}(\Z) \cap (\N^*)^{\Sigma_{\ext}(1)},$$
    \item for $d \in (\N^*)^{\Sigma_{\ext}(1)}$, we write $\mathbf{1}_{d} : (\N^*)^{\Sigma_{\ext}(1)} \rightarrow \{0,1\}$ the indicator function of $\bigoplus\limits_{\rho \in \Sigma_{\ext}(1)} d_{\rho}.\Z \cap (\N^*)^{\Sigma_{\ext}(1)}$, that is to say we have the following equivalence: $$\mathbf{1}_{d}(e) = 1 \Leftrightarrow  d \mid e $$where $d \mid e$ means $d_{\rho} \mid e_{\rho}$ for any $\rho \in \Sigma_{\ext}(1)$. 
\end{enumerate}
\end{Notations}

We want to define and study a Möbius function $\mu : (\N^*)^{\Sigma_{\ext}(1)} \rightarrow \Z$ such that:
$$\mathbf{1}_{\stackT^{\ext}(\Z)} = \sum\limits_d \mu(d) \mathbf{1}_{d}  .$$The rest of the properties we want to establish here are analogous to the ones used by Salberger in \cite[lemmas 11.15-11.19]{salberger_torsor}.

\subsubsection{The local Möbius function}

\begin{Notations}
    Let us write $\Lambda$ the simplicial cone $\R_+^{\Sigma_{\ext}(1)}$. For $\lambda \in \Lambda$, we write $$|\lambda| = \sum\limits_{\rho \in \Sigma_{\ext}(1) } \lambda_{\rho} .$$
    To simplify the notation, we write $M = \Z^{\Sigma_{\ext}(1)}$ and $\log_v$ for the log map $\log_{T^{\ext},v}$ associated to the torus $T^{\ext}(\Q_v)$ and $v \in M_{\Q}^0$, that is to say $$\log_v(y) = (v(y_{\rho}))_{\rho \in \Sigma_{\ext}(1)} .$$
\end{Notations}

Let us remark that $\stackT^{\ext}(\Z) = \stackT^{\ext}(\Q) \cap \bigcap\limits_{v \in M_{\Q}^0} \stackT^{\ext}(\Z_v)$. By definition \ref{definition_orb_universal_torsor_gcd_cond}, $\stackT^{\ext}(\Z_v)$ is defined in $\stackT^{\ext}(\Q_v)$ via a finite number of arithmetical conditions. That is to say there exists a finite set of elements of $\Lambda$, $\{ \lambda_{i,v} \}_{i \in I_v}$ such that if we write $$X_v = \Lambda - \bigcup\limits_{i \in I_v} \Lambda + \lambda_{i,v} ,$$we have $$\stackT^{\ext}(\Z_v) \cap T^{\ext}(\Q_v) = \log_v^{-1}(X_v) .$$In fact, this arithmetical conditions do not depend on $v$ so we can choose a fixed finite set of elements of $\Lambda$ $\{ \lambda_{i} \}_{i \in I}$ which do not depend on $v$ and we can write $X = X_v$ for any $v$.

We shall use the results from \cite[section 3.3]{methode_du_cercle_peyre}. The ring of functions $$M \rightarrow \Z$$ with support in $\Lambda$, $\Z [[\Lambda]]$ is an integral ring equipped with the convolution product given by:
$$fg(y) = \sum\limits_{x \in \Lambda \cap M} f(y-x) g(x) .$$By \cite[lemma 3.3.8]{methode_du_cercle_peyre}, $\mathbf{1}_\Lambda$ admits an inverse in $\Z [[\Lambda]]$. We can write the following definition:

\begin{definition}
    We define by convolution the element of $\Z [[\Lambda]]$ given by $$\mu_{\loc} = \mathbf{1}_X . ( \mathbf{1}_\Lambda )^{-1} .$$That is to say we have: $$ \mathbf{1}_X = \sum\limits_{\lambda \in \Lambda \cap M} \mu_{\loc}(\lambda) \mathbf{1}_{\lambda + \Lambda} .$$
\end{definition}

\begin{Notations}
    For $\lambda, \tau \in \Lambda$, we write $\lambda \leqslant \tau$ when for any $\rho \in \Sigma_{\ext}(1)$, $\lambda_{\rho} \leqslant \tau_{\rho}$ and we write $\lambda < \tau$ if moreover $\lambda \neq \tau$. We denote by $1_{\rho}$ the element of $\Lambda$ with zero for each coordinate except for the one corresponding to $\rho \in \Sigma_{\ext}(1)$ where there is 1. We have $1_{\rho} \in X$.
\end{Notations}

\begin{prop}\label{proposition_local_mobius_prop_one}
    The local Möbius function $\mu_{\loc}$ has the property that $\mu_{\loc}(0) = 1 $ and that if $\lambda \neq 0$, we have $$\mu_{\loc}(\lambda) \neq 0 \implies |\lambda| \geqslant 2 .$$
\end{prop}

\begin{proof}
    First, let us recall that $$\mathbf{1}_{\Lambda} . \mu_{\loc} = \mathbf{1}_X$$ and for $\lambda \in \Lambda \cap M$,  $\mathbf{1}_{\Lambda}(-\lambda) \neq 0 $ is equivalent to $\lambda = 0$. So if we evaluate the previous equality in $y = 0$, because $0 \in X$, we get $\mu_{\loc}(0) = 1 $.

    Now for $\tau \in \Lambda \cap M$, we have:
    $$\mu_{\loc}(\tau) = \mathbf{1}_X(\tau) - \sum\limits_{\lambda < \tau} \mu_{\loc}(\lambda) . $$Using that $1_{\rho} \in X$, the previous equality and a finite recurrence, we get that if $|\tau| < 2$ then $\tau \in X$ and because $\mu_{\loc}(0) = 1$ then $\mu_{\loc}(\tau) = 0$.
\end{proof}

\begin{prop}\label{proposition_local_mobius_prop_two}
    Let $\lambda \in \Lambda \cap M$. If $|\lambda| \geqslant 3.\sharp \Sigma_{\ext}(1)$, we have $\mu_{\loc}(\lambda) = 0 $.
\end{prop}

\begin{proof}
    For $\lambda \in \Lambda \cap M$ such that there exists $\rho \in \Sigma_{\ext}(1)$, $\lambda_{\rho} \geqslant 2$, we have using the definition of $X$, that $$\mathbf{1}_X(\lambda + 1_{\rho}) = \mathbf{1}_X(\lambda).  $$Hence we can deduce that for such $\lambda$, $$\mu_{\loc}(\lambda + 1_{\rho}) = \sum\limits_{\lambda < \tau \leqslant \lambda + 1_{\rho}} \mu_{\loc}(\tau) = 0 . $$Using this formula, we get by an induction on $n \in \N, n \geqslant 3$, that if there exists $\rho \in \Sigma_{\ext}(1)$ with $\lambda_{\rho} \geqslant 3$ and $|\lambda| = n$ then $\mu_{\loc}(\lambda) = 0$. Thus if $|\lambda| \geqslant 3 . \sharp \Sigma_{\ext}(1)$, as there exists $\rho \in \Sigma_{\ext}(1)$ with $\lambda_{\rho} \geqslant 3$, $$\mu_{\loc}(\lambda) = 0 .$$
\end{proof}

\begin{prop}\label{proposition_local_study_of_mobius_function}
   There exists a polynomial $Q$ with non negative integers coefficients such that for $s \in \R$, $$ \sum\limits_{\lambda \in \Lambda \cap M} \frac{|\mu_{\loc}(\lambda)|}{p_v^{s.|\lambda|}} = 1 + p_v^{-2.s} Q(\frac{1}{p_v^s}) .$$Note that $Q$ does not depend on $v \in M_{\Q}^0$.

   Moreover, for $v \in M_{\Q}^0$, we have the equality $$\sum\limits_{\lambda \in \Lambda \cap M} \frac{\mu_{\loc}(\lambda)}{p_v^{|\lambda|}} = \w_{\stackT,v}(\stackT^{\ext}(\Z_v)) $$ where $\w_{\stackT,v}$ is the restriction of the normalized Haar measure $d x_v$ over $\Q_v^{\Sigma_{\ext}(1)}$ to the open subset $\stackT^{\ext}(\Q_v)$.
\end{prop}

\begin{Notations}
    For $\lambda \in \Lambda \cap M$, let us write $p^{\lambda} . \Z_p = \bigoplus\limits_{\rho \in \Sigma_{\ext}(1)} p^{\lambda_{\rho}} . \Z_p .$ We shall denote by $\mathbf{1}_{\stackT^{\ext}(\Z_p)}$  the indicator function of $\stackT^{\ext}(\Z_p) \cap T^{\ext}(\Q_v)$ while $\mathbf{1}_{p^{\lambda} . \Z_p}$ is the indicator function of $p^{\lambda} . \Z_p \cap T^{\ext}(\Q_v)$.
\end{Notations}

\begin{proof}
    First using proposition \ref{proposition_local_mobius_prop_one} and \ref{proposition_local_mobius_prop_two}, we can write:
$$\sum\limits_{\lambda \in \Lambda \cap M} \frac{|\mu_{\loc}(\lambda)|}{p_v^{s.|\lambda|}} = 1 + p_v^{-2.s} Q(\frac{1}{p_v^s}) $$ with $Q(X) = \sum\limits_{2 \leqslant |\lambda| \leqslant l} | \mu_{\loc}(\lambda) |.X^{|\lambda|-2} $ where $l = 3. \sharp \Sigma_{\ext}(1)$. Then we have:
    \begin{align*}
        \mathbf{1}_{\stackT^{\ext}(\Z_p)} &= \mathbf{1}_X(\log_v(-)) \\
    &= \sum\limits_{\lambda \in \Lambda \cap M} \mu_{\loc}(\lambda) \mathbf{1}_{\lambda + \Lambda}(\log_v(-)) \\
    &= \sum\limits_{\lambda \in \Lambda \cap M} \mu_{\loc}(\lambda) \mathbf{1}_{p^{\lambda} . \Z_p}(-) .
    \end{align*}So if we integrate with respect to the measure $\w_{\stackT,v}$, because $$\w_{\stackT,v}(p^{\lambda} . \Z_p) = \frac{1}{p^{|\lambda|}}$$we get the expected formula:
    $$\sum\limits_{\lambda \in \Lambda \cap M} \frac{\mu_{\loc}(\lambda)}{p^{|\lambda|}} = \w_{\stackT,v}(\stackT^{\ext}(\Z_p)) .$$
\end{proof}

\subsubsection{The global Möbius function}

\begin{definition}
    We define the global Möbius function as a product over all the finite places:
    $$\mu = \prod\limits_{v \in M_{\Q}^0} \mu_{\loc}(\log_v(-)) .$$
\end{definition}

We may remark that $\mathbf{1}_{\stackT^{\ext}(\Z)} = \prod\limits_{v \in M_{\Q}^0} \mathbf{1}_X(\log_v(-))$ and $\mathbf{1}_{d} = \prod\limits_{v \in M_{\Q}^0} \mathbf{1}_{C + \lambda_v}(\log_v(-))$ where $\lambda_v = \log_v(d)$ for any $d \in (\N^*)^{\Sigma_{\ext}(1)}$. Hence we get the following theorem:

\begin{theorem}\label{theorem_mobius_inversion}
    We have:
    $$\mathbf{1}_{\stackT^{\ext}(\Z)} = \sum\limits_d \mu(d) \mathbf{1}_{d} .$$
\end{theorem}


\begin{definition}
    Two elements $a, b \in (\N )^{\Sigma_{\ext}(1)}$ are said coprime if for any $v \in M_{\Q}^0$, we either have $\log_v(a) = 0 $ or $\log_v(b) = 0$.
    An arithmetic function $f : (\N )^{\Sigma_{\ext}(1)} \rightarrow \C$ is said to be multiplicative is for any $a, b \in (\N )^{\Sigma_{\ext}(1)}$ coprime we have $f(a.b) = f(a).f(b)$.
\end{definition}

\begin{remark}\label{remark_multiplicativity_of_mobius_function}
    It is clear that the Mobius function $\mu$ we have defined is multiplicative.
\end{remark}


\subsubsection{Study of the convergence of a Dirichelet series associated to our Möbius function and its consequences}

\begin{Notations}
    For a prime number $p$, we denote by $$\sum\limits_{d}^{(p)}$$ the sum over the $\Sigma_{\ext}(1)$-tuple of the form $d = (p^{e_{\rho}})_{\rho \in \Sigma_{\ext}(1)}$ where $(e_{\rho})$ belongs to $\N^{\Sigma_{\ext}(1)}$. We also define $$\Pi(d) = \prod\limits_{\rho \in \Sigma_{\ext}(1)} d_{\rho}$$ for any $d \in (\N^*)^{\Sigma_{\ext}(1)}$.
\end{Notations}

As a consequence of the definition of the global Möbius function and using proposition \ref{proposition_local_study_of_mobius_function}, we have:

\begin{prop}\label{value_local_mesure_mobius_function}
    For any prime number $p$, $$\sum\limits_{d}^{(p)} \frac{|\mu(d)|}{\Pi(d)} = \sum\limits_{\lambda \in C \cap M} \frac{|\mu_{\loc}(\lambda)|}{p^{|\lambda|}} = \w_{\stackT,v}(\stackT^{\ext}(\Z_p)) .$$
\end{prop}

\begin{theorem}\label{value_global_mesure_mobius_fonction}

\begin{itemize}
    \item The product $$\prod\limits_p \left( \sum\limits_{d}^{(p)} \frac{|\mu(d)|}{\Pi(d)^s} \right)$$ over all the prime numbers $p$ is absolute convergent for $s > \frac{1}{2}$.
    \item The sum $$\sum\limits_d \frac{|\mu(d)|}{\Pi(d)^s}$$ is convergent for $s > \frac{1}{2}$ and equal to $$\prod\limits_p \left( \sum\limits_{d}^{(p)} \frac{|\mu(d)|}{\Pi(d)^s} \right)$$ for $s > \frac{1}{2}$.
    \item The product $$\prod\limits_p \left( \sum\limits_{d}^{(p)} \frac{\mu(d)}{\Pi(d)} \right)$$ over all the prime numbers $p$ is absolute convergent and equal to $$\sum\limits_d \frac{\mu(d)}{\Pi(d)} .$$
\end{itemize}    
\end{theorem}

\begin{proof}
   The proof is analogous to the one of \cite[Lemma 11.15]{salberger_torsor}. The first assertion is a consequence of proposition \ref{proposition_local_study_of_mobius_function}. Using remark \ref{remark_multiplicativity_of_mobius_function}, we can establish the second assertion. The last assertion is a consequence of the first two.
\end{proof}

\begin{cor}\label{majoration_fonction_de_mobius}
Let $Y > 1$ and $\varepsilon \in ]0,\frac{1}{2}[$. We have the following upper bounds:
    \begin{enumerate}
        \item $\sum\limits_{\Pi(d)\leqslant Y} |\mu(d)| = O\left( Y^{\frac{1}{2} + \varepsilon}\right)$
        \item $\sum\limits_{\Pi(d)\geqslant Y} \frac{| \mu(d) |}{\Pi(d)} = O\left( \frac{1}{Y^{ \frac{1}{2} - \epsilon}} \right).$
    \end{enumerate}
\end{cor}

\begin{proof}
    The proof is analogous to \cite[lemma 11.19]{salberger_torsor} and it is a consequence of the previous theorem.
\end{proof}

\subsection{Local Tamagawa measure of a toric stack}

We write $$\rho_{\orb}(X) = t + \sharp \twistsector$$ for the rank of the orbifold Picard group $\Pic_{\orb}(X)$. Here $t$ denotes the rank of the usual Picard group. For any $v \in M_{\Q}$, we denote by $\w_{T_{\NS,\orb},v}$ the natural measure on the torus $T_{\NS,\orb}(\Q_v)$ (see \cite[Definition 3.1.2]{peyre_zeta_height_function}). 

For a place $v \in M_{\Q}$, we define the measure $\w_{\Sal,v}$ on $\stackT(\Q_v)$ by
$$
\w_{\Sal,v}(dy) = \frac{\w_{\stackT,v}(dy)}{H^{\ext}_{\stackT,v}\left(y,\w_{X,\orb}^{-1}\right)}.
$$By construction, the measure $\w_{\Sal,v}(dy)$ is invariant under the action of $T_{\NS,\orb}(\Q_v)$. We can therefore introduce the following definition:

\begin{definition}\label{definition_quotient_measure_over_orbifold}
We define the measure $\mu_{X,v}$ as the quotient measure on $X(\Q_v)$ given by $$\frac{\w_{\Sal,v}}{\w_{T_{\NS,\orb},v}} \, .$$
\end{definition}

\begin{remark}
    This construction is inspired both by the work of Salberger (see \cite{salberger_torsor}), who was the first to observe that the correct construction of $\mu_{X,v}$ is as a quotient measure (see also \cite{peyre_zeta_height_function}), and by the work of Loughran and Santens in defining a Tamagawa measure on $BG$ (see \cite[Section 8]{loughran_santens_malle_conjecture}).
\end{remark}

\begin{lemma}\label{lemma_computation_integral_quotient}
    Let $P \mapsto y_P$ be any measurable section of $\stackT^{\ext}(\Q_v) \rightarrow X(\Q_v)$. For any function $f : X(\Q_v) \rightarrow \R$ compactly supported and measurable, the function:
    $$ P \longmapsto \int\limits_{T_{\NS}(\Q_v)} | \w_{X,\orb}^{-1}(t) |_{v}^{-1} f(t.y_P) \w_{T_{\NS,\orb},v}(dt)$$ does not depend on the choice of the section $P \mapsto y_P$ and we have:
    $$ \int\limits_{\stackT^{\ext}(\Q_v)} f(y) \w_{\stackT,v}(dy) = \int\limits_{X(\Q_v)} \mu_{X,v}(dP) \int\limits_{T_{\NS}(\Q_v)} | \w_{X,\orb}^{-1}(t) |_{v}^{-1} f(t.y_P) \w_{T_{\NS,\orb},v}(dt) .$$
\end{lemma}

\begin{proof}
The first assertion is a consequence of the invariance of the Haar measure (see \cite[Chapitre VII, §2, n°2, proposition 1]{bourbaki2007integration}). To prove the second assertion, we first use that $\w_{\Sal,v}(dy)$ is a $T_{\NS,\orb}(\Q_v)$-invariant measure and the definition of $\mu_{X,v}$ as a quotient measure. Hence, by \cite[Chapitre VII, §2, n°2, Proposition 4]{bourbaki2007integration}, we can write that for any compactly supported function $f : X(\Q_v) \rightarrow \R$,
\begin{align*}
\int_{\stackT^{\ext}(\Q_v)} f(y) \w_{\stackT,v}(dy)
&= \int_{\stackT^{\ext}(\Q_v)} f(y) H_{\stackT,v}^{\ext}(y,\w_{X,\orb}^{-1}) \w_{\Sal,v}(dy) \\
&= \int_{X(\Q_v)} \mu_{X,v}(dP)
\int_{T_{\NS,\orb}(\Q_v)} H_{\stackT,v}^{\ext}(t.y_P,\w_{X,\orb}^{-1}) f(t.y_P) \w_{T_{\NS,\orb},v}(dt).
\end{align*}

We choose a section $P \mapsto y_P$ such that
$$
H_{\stackT,v}^{\ext}(y_P,\w_{X,\orb}^{-1}) = 1
$$
for any $P$. In particular, if $v \in M_{\Q}^0$, by Proposition \ref{proposition_forme_classique_parametrisation_entiere_locale}, it is enough to choose a section $P \mapsto y_P$ such that for all $P$,
$$
y_P \in \stackT^{\ext}(\Z_v).
$$We thus obtain the desired result:
\[
 \int_{\stackT^{\ext}(\Q_v)} f(y)\,\w_{\stackT,v}(dy)
 =
 \int_{X(\Q_v)} \mu_{X,v}(dP)
 \int_{T_{\NS,\orb}(\Q_v)}
 |\w_{X,\orb}^{-1}(t)|_{v}^{-1} f(t.y_P)\,\w_{T_{\NS,\orb},v}(dt).
\]

\end{proof}

Using this formula, we get:

\begin{prop}\label{proposition_finite_place_quotient_measure_over_X}
    We have for any $v \in M_{\Q}^0$, 
    $$\w_{\stackT,v}(\stackT^{\ext}(\Z_v)) = \left( 1 - \frac{1}{p_v} \right)^{\rho_{\orb}(X)} \cdot \mu_{X,v}\left(X(\Q_v)\right) .$$
\end{prop}

\begin{proof}
    It is enough to apply the formula of Lemma \ref{lemma_computation_integral_quotient} to $f = \mathbf{1}_{\stackT^{\ext}(\Z_v)}$, considering a section $P \mapsto y_P$ such that for all $P$, $y_P \in \stackT^{\ext}(\Z_v)$.
\end{proof}

We shall also compute $\w_{\stackT,\infty}(\mathcal{D}) = \Vol(\mathcal{D})$ where $$\mathcal{D} = \{ y \in \stackT^{\ext}(\R) \mid h^{\ext}_{\stackT,\infty}(y) \in \mathrm{D} \}$$ for a compact subset $\mathrm{D} \subset \Pic_{\orb}(X)^{\vee}_{\R}$. We shall connect it to $\nu(\mathrm{D})$ where $\nu$ is the measure over $\Pic_{\orb}(X)^{\vee}_{\R}$ defined as in \ref{mesure_picard_group}.

\begin{prop}\label{analysis_measure_univ_torsor_at_infty}
   $$ \w_{\stackT,\infty}(\mathcal{D}) = \Vol(\mathcal{D}) = 2^{\rho_{\orb}(X)} .\mu_{X,\infty}(X(\R)) . \nu(\mathrm{D}) . $$
\end{prop}

\begin{proof}
   The proof is identical to \cite[proposition 3.29]{bongiorno2024multiheightanalysisrationalpoints}.   

\end{proof}

\subsection{Tamagawa number of a toric stack}

We define the Tamagawa number of the toric stack $X$ as follows:
   
\begin{definition}\label{definition_orbifold_tamagawa_number}
    The Tamagawa number of $X$ is defined by $$\tau_{\orb}(X) = \frac{1}{\sharp G(\Z)} . \mu_{X,\infty}(X(\R)). \prod\limits_p \left( \left(1 - \frac{1}{p} \right)^{\rho_{\orb}(X)} \cdot \mu_{X,p}(X(\Q_p)) \right)$$ where $G$ is the generic stabilizer of $X$. Here it is the kernel of the natural map $T_{\NS,\orb} \rightarrow T^{\ext}$.
\end{definition}

As suggested by Loughran and Santens, the natural way to count in a groupoid is via the groupoid cardinality, i.e. the number of isomorphism classes of objects in the groupoid weighted by the inverse of the cardinality of the automorphism group of each object. If $W$ is the set of isomorphism classes of a groupoid, we shall write:
\begin{equation}\label{equation_groupoid_cardinality}
\sharp W = \sum\limits_{w \in W} \frac{1}{\sharp \Aut(w)} .    
\end{equation}
The following theorem is inspired by the mass formula obtained in \cite[Corollary 8.11]{loughran_santens_malle_conjecture} and by explicit computations in several examples.

\begin{theorem}\label{theorem_interpretation_mass_formula}
Let $p$ be a prime number such that $X_{\Z_p}$ is a tame DM stack, i.e. $p \nmid N_X$ the exponent of $X$. Then the following equality holds:
\[
\mu_{X,p}\bigl(X(\Q_p)\bigr)
= \frac{\sharp X(\F_p)}{p^{\dim(X)}}  + \frac{1}{p} \cdot
\sum\limits_{\stackS \in \twistsector} 
\frac{\sharp\, \stackS\!\left(\F_p\right)}{p^{\dim(\stackS)}} \, \, .
\]
\end{theorem}

The remainder of this section is devoted to the proof of this theorem. To this end, we introduce the following notation. We keep the notation $\stackT_{\stackS}(\Z_p)$ for $\stackS \in \sector$ introduced in \eqref{equation_definition_notation_t_stacks}. Since
\[
\stackT^{\ext}(\Z_p)=\bigsqcup_{\stackS\in\sector}\stackT_{\stackS}(\Z_p) \, ,
\]
in order to prove the previous theorem, we shall compute, for each $\stackS\in\sector$, the quantity
\[
\w_{\stackT,v}(\stackT_{\stackS}(\Z_p)).
\]

For this purpose, we consider the image of the reduction of $\stackT_{\stackS}(\Z_p)$ in
\[
(\Z / p \Z)^{\Sigma(1)} \times (\Z /p^2 \Z)^{\twistsector},
\]
which we denote by
\[
\stackT_{\stackS}(\F_p) \times Y_{\stackS,p}
\subset
\stackT(\F_p) \times (\Z /p^2 \Z)^{\twistsector}.
\]

\begin{remark}\label{remark_description_reduction_stackT_stackS_mod_p}
The subset $\stackT_{\stackS}(\F_p)$ of $\stackT(\F_p)$ is the set of $y \in \F_p^{\Sigma(1)}$ such that for every primitive collection $C$, there exists $\rho \in C - \sigma_\stackS(1)$ such that $ y_{\rho} \neq 0$.
\end{remark}

In view of the divisibility conditions defining $\stackT^{\ext}(\Z_p)$ (see Definition~\ref{definition_orb_universal_torsor_gcd_cond}) and $\stackT_{\stackS}(\Z_p)$, we obtain the following equality:

\begin{equation}\label{equation_computation_mesure_extedended_univ}
    \w_{\stackT,v}(\stackT_{\stackS}(\Z_p))
    =
    \frac{\sharp \stackT_{\stackS}(\F_p)}{p^{\sharp \Sigma(1)}}
    \cdot
    \frac{\sharp Y_{\stackS,p}}{p^{2 \sharp \twistsector}} \, .
\end{equation}

We begin by determining the term with $\sharp Y_{\stackS,p}$ in the following lemma.

\begin{lemma}\label{lemma_computation_cardinal_sector_part}
We have
\[
\frac{\sharp Y_{\stackS,p}}{p^{2 \sharp \twistsector}}
=
\begin{cases}
\left(1-\frac{1}{p}\right)^{\sharp{\twistsector}} & \text{if } \stackS = 0, \\[1em]
\dfrac{1}{p}\left(1-\frac{1}{p}\right)^{\sharp{\twistsector}} & \text{if } \stackS \in \twistsector .
\end{cases}
\]
\end{lemma}

\begin{proof}
It is enough to see that the set $Y_{\stackS,p}$ is
\[
\left(\Z/p^2\Z - p\Z/p^2\Z\right)^{\twistsector}
\]
if $\stackS = 0$, whereas if $\stackS \in \twistsector$, it is given by
\[
\left(p\Z/p^2\Z - \{0\}\right)
\times
\prod_{\stackS' \in \twistsector - \{\stackS\}}
\left(\Z/p^2\Z - p\Z/p^2\Z\right).
\]
\end{proof}

In the sequel, we denote by $G$ the multiplicative group dual to the group
$$\Pic(X)/\Pic(X^{\rig}) = \Ext^1(N,\Z).$$Recall that by Corollary~\ref{cor_description_map_stack_to_rigidification} we have the exact sequence
\[
1 \rightarrow G \rightarrow T_{\NS} \rightarrow T_{\NS,\rig} \rightarrow 1 .
\]

We keep the notation from \ref{notation_stacky_fan_beta} and \ref{notation_cone_minimal_sector}. For $\stackS \in \sector$, we write $\stackZ_{\stackS}$ the closed subscheme of $\stackT$ over $\F_p$ defined by
\[
\stackZ_{\stackS}
=
\stackT \cap
\bigcap_{\rho \in \sigma_\stackS(1)}
\{y_{\rho}=0\} \, .
\]The scheme $\stackZ_{\stackS}$ is therefore invariant under the action of $T_{\NS}$, and the restricted action of $G$ is trivial. Consequently, we obtain an induced quotient morphism
\[
[\stackZ_{\stackS}/T_{\NS}]
\longrightarrow
[\stackZ_{\stackS}/T_{\NS,\rig}]
\]
which is a $G$-gerbe. We observe that by \cite[Lemma~4.6, Proposition~4.7]{borisov_toric_stack} we have
\[
\stackS \simeq [\stackZ_{\stackS}/T_{\NS}] .
\]

We also denote by $\stackS^{\rig}$ the quotient stack
\[
\stackS^{\rig} = [\stackZ_{\stackS}/T_{\NS,\rig}],
\]
without attempting to justify this notation. The morphism $\stackT \xrightarrow{\wedge a} \stackT$ induces an endomorphism of
$\stackZ_{\stackS}$ which is invariant under the group morphism
$T_{\NS,\rig} \to T_{\NS,\can}$. Hence we obtain an induced quotient morphism
\[
[\stackZ_{\stackS}/T_{\NS,\rig}]
\longrightarrow
[\stackZ_{\stackS}/T_{\NS,\can}] .
\]Again, abusing the notation, we set
\[
\stackS^{\can} = [\stackZ_{\stackS}/T_{\NS,\can}] .
\]From the description of the morphism $X^{\rig} \to X^{\can}$ for toric stacks
(see \cite[Theorem~5.2~(2)]{fantechi_toric_stack} and
Corollary~\ref{cor_description_map_stack_to_canonical}), it follows that the
morphism
\[
\stackS^{\rig} \longrightarrow \stackS^{\can}
\]
is a $\mu_{a_{\rho}}$-gerbe over the restriction of $D_{\rho}^{\rig}$ to
$D_{\rho}^{\can}$ for $\rho \in \sigma_\stackS(1)$, and an isomorphism between
\[
\stackS^{\rig} \setminus \bigcup_{\rho \in \sigma_\stackS(1)} D_{\rho}^{\rig}
\quad\text{and}\quad
\stackS^{\can} \setminus \bigcup_{\rho \in \sigma_\stackS(1)} D_{\rho}^{\can}.
\]

In order to relate the groupoid cardinalities of $\stackS(\F_p)$ and $\stackS^{\can}(\F_p)$ for $\stackS \in \sector$, we shall need a technical lemma. It is motivated by the powerful result obtained by Loughran and Santens (see \cite[Lemma 8.8]{loughran_santens_malle_conjecture}), which states that if $\stackG$ is a $G$-gerbe proper and étale over $\F_p$, where $G$ is a finite étale tame group scheme over $\F_p$, then the groupoid counting of its $\F_p$-points satisfies
\[
\sharp \stackG(\F_p) = 1.
\]

\begin{lemma}
Let $\stackX \xrightarrow{f} \stackY$ be a proper étale $G$-gerbe between tame Deligne--Mumford stacks over $\F_p$. Then the groupoid counting of their $\F_p$-points satisfies
\[
\sharp \stackX(\F_p) = \sharp \stackY(\F_p).
\]
\end{lemma}

\begin{proof}
Let $Q \in \stackY(\F_p)$ and denote by $\stackG_{Q}$ the proper étale $G$-gerbe over $\F_p$ given by the
fiber product
\[
\Spec(\F_p) \underset{Q,\stackY,f}{\times} \stackX .
\]
By \cite[Lemma~8.8]{loughran_santens_malle_conjecture}, $\stackG_Q$ is a trivial $G$-gerbe so that
$\stackG_Q \simeq BG$. For $P \in \stackX(\F_p)$ such that $f(P) = Q$, using the description of fiber products of groupoids, the natural group morphism
\[
\Aut_{\stackX}(P) \longrightarrow \Aut_{\stackY}(Q)
\]
has kernel
\[
\Aut_{\stackG}(P).
\]Thus, we get
\[
\sum_{f(P)=Q} \frac{1}{\sharp \Aut_{\stackX}(P)}
=
\frac{\sharp \stackG_Q(\F_p)}{\sharp \Aut_{\stackY}(Q)} .
\]Using this identity, we compute
\begin{align*}
\sharp \stackX(\F_p)
&=
\sum_{P \in \stackX(\F_p)} \frac{1}{\sharp \Aut_{\stackX}(P)} \\
&=
\sum_{Q \in \stackY(\F_p)} \frac{\sharp \stackG_Q(\F_p)}{\sharp \Aut_{\stackY}(Q)} \\
&=
\sum_{Q \in \stackY(\F_p)} \frac{1}{\sharp \Aut_{\stackY}(Q)}
=
\sharp \stackY(\F_p),
\end{align*}
where the penultimate equality again follows from \cite[Lemma~8.8]{loughran_santens_malle_conjecture}.
\end{proof}

Applying this lemma, we get by the above discussion that
\begin{equation}
\sharp \stackS (\F_p) = \sharp \stackS^{\can}(\F_p) \, \, .    
\end{equation}We can now proceed to the computation of $\sharp(\stackT_{\stackS}(\F_p))$. 

\begin{lemma}\label{lemma_computation_cardinal_torsor_sector}
For $\stackS \in \sector$, we have
\[
\frac{\sharp(\stackT_{\stackS}(\F_p))}{p^{\sharp \Sigma(1)}}
=
\left(1-\frac{1}{p}\right)^t
\cdot
\frac{\sharp \stackS(\F_p)}{p^{\dim(\stackS)}},
\]
where $t = \dim_{\R} \Pic(X)_{\R}$.
\end{lemma}

\begin{proof}
Since the values of $y_{\rho}$ for $\rho \in \sigma_\stackS(1)$ are not constrained by Remark~\ref{remark_description_reduction_stackT_stackS_mod_p}, the set $\stackT_{\stackS}(\F_p)$ can be written as the disjoint union of $p^{\sharp \sigma_\stackS(1)}$ translates of $\stackZ_{\stackS}(\F_p)$. Thus, we have
\[
\sharp \stackT_{\stackS}(\F_p)
=
p^{\sharp \sigma_\stackS(1)} \cdot \sharp \stackZ_{\stackS}(\F_p).
\]

We now relate $\sharp \stackZ_{\stackS}(\F_p)$ and $\sharp \stackS^{\can}(\F_p)$.
Since
\[
H^1(\F_p,T_{\NS,\can}) = 0,
\]
because $T_{\NS,\can}$ is a split torus (see Corollary~\ref{cor_description_map_stack_to_canonical}),
we get from the quotient description of $\stackS^{\can}$ a bijection
\[
\stackS^{\can}(\F_p)
=
\stackZ_{\stackS}(\F_p) / T_{\NS,\can}(\F_p).
\]Thus we can write
\begin{align*}
    \frac{\sharp(\stackT_{\stackS}(\F_p))}{p^{\sharp \Sigma(1)}} &=\frac{\sharp(\stackZ_{\stackS}(\F_p))}{p^{\sharp \Sigma(1) - \sharp \sigma_\stackS(1)}}  \\
    &= \frac{1}{p^{t + \dim(\stackS)}} \cdot \sum\limits_{P \in \stackS^{\can}(\F_p)} \frac{\sharp T_{\NS,\can}(\F_p)}{\sharp \underline{\Aut}(P)(\F_p) } \\
    &=  \frac{\sharp T_{\NS,\can}(\F_p)}{p^t} \cdot \frac{1}{p^{\dim(\stackS)}} \cdot \sum\limits_{P \in \stackS^{\can}(\F_p)} \frac{1}{\sharp \underline{\Aut}(P)(\F_p) } \\
    &= \left( 1 - \frac{1}{p} \right)^{t} . \frac{\sharp \stackS^{\can} \left(\F_p \right) }{p^{\dim(\stackS)}} .
\end{align*}We can conclude the proof using that moreover $\sharp \stackS(\F_p) = \sharp \stackS^{\can}(\F_p)$.
\end{proof}

We can now conclude the proof of Theorem~\ref{theorem_interpretation_mass_formula}.

\begin{proof}
First, we compute $\w_{\stackT,p}(\stackT^{\ext}(\Z_p))$. We have
\begin{align*}
\w_{\stackT,p}(\stackT^{\ext}(\Z_p))
&= \sum_{\stackS \in \sector} \w_{\stackT,p}(\stackT_{\stackS}(\Z_p)) \\
&= \sum_{\stackS \in \sector}
\frac{\sharp \stackT_{\stackS}(\F_p)}{p^{\sharp \Sigma(1)}}
\cdot
\frac{\sharp Y_{\stackS,p}}{p^{2 \sharp \twistsector}},
\end{align*}
where the last equality follows from \eqref{equation_computation_mesure_extedended_univ}. 
Applying Lemmas~\ref{lemma_computation_cardinal_sector_part} and 
\ref{lemma_computation_cardinal_torsor_sector}, we get
\[
\w_{\stackT,p}(\stackT^{\ext}(\Z_p))
=
\left(1-\frac{1}{p}\right)^{\rho_{\orb}(X)}
\cdot
\left( \frac{\sharp X(\F_p)}{p^{\dim(X)}}  + \frac{1}{p} \cdot
\sum\limits_{\stackS \in \twistsector} 
\frac{\sharp\, \stackS\!\left(\F_p\right)}{p^{\dim(\stackS)}} \right) .
\]
The theorem now follows from Proposition~\ref{proposition_finite_place_quotient_measure_over_X}.
\end{proof}

\section{Multi-height distribution of rational points of toric stacks}

\begin{Notations}
    We denote by $G$ the multiplicative group associated to $\Ext^1(N,\Z)$, that is to say the generic stabilizer of $X$.
\end{Notations}

We fix $u \in (\coneorb^{\vee})^{\circ}$ and $\mathrm{D} \subset \Pic_{\orb}(X)^{\vee}_{\R}$ a finite union of compact polyhedrons. Let $B > 1$ and write $\mathrm{D}_B = \mathrm{D} + \log(B) u $. By theorem \ref{theorem_description_of_the_extended_parmetrization}, we have that the map
$$ q : \stackT^{\ext}(\Z) \rightarrow X(\Q)$$is surjective and it makes $X(\Q)$ a $T_{\NS,\orb}(\Z)$-quotient of $\stackT^{\ext}(\Z)$. We get using Theorem \ref{theorem_lift_height_extend_universal_torsor}:
\begin{align*}
    \sharp T(\Q)_{h \in \mathrm{D}_B} &= \frac{ \sharp G(\Z)}{2^{\rho_{\orb}(X)}} \sharp \{ y \in \stackT^{\ext}(\Z) \cap T^{\ext}(\Q) \mid h^{\ext}_{\stackT,\infty}(y) \in \mathrm{D}_B \} \\
    &= \frac{ \sharp G(\Z)}{2^{\rho_{\orb}(X)}} \sharp \left( \stackT^{\ext}(\Z) \cap T^{\ext}(\Q) \right)_{h^{\ext}_{\stackT,\infty} \in \mathrm{D}_B} .
\end{align*}

Using the Möbius inversion, we shall see that estimating $$\sharp \left( \stackT^{\ext}(\Z) \cap T^{\ext}(\R) \right)_{h^{\ext}_{\stackT,\infty} \in \mathrm{D}_B}$$ reduces to estimating:
$$\sharp \Bigl( \Z^{\Sigma_{\ext}(1)} \cap T^{\ext}(\R) \Bigr)_{h^{\ext}_{\stackT,\infty} \in \mathrm{D}_B} .$$

\subsection{Estimation of the error term between the number of points of a lattice in a bounded domain and the volume of the domain}

In what follows, $\Vol$ means the usual euclidean volume given by the Lebesgue measures. The aim of this section is to relate $$\sharp \Bigl( \Z^{\Sigma_{\ext}(1)} \cap T^{\ext}(\R) \Bigr)_{h^{\ext}_{\stackT,\infty} \in \mathrm{D}_{B}}$$ and the volume of $$\mathcal{D}_B = \{ y \in \stackT^{\ext}(\R) \mid h^{\ext}_{\stackT,\infty}(y) \in \mathrm{D}_{B}\}.$$ To estimate the difference between these two terms, we shall use a result of geometry of numbers due to Davenport \cite{davenport_geometry_numbers}. The statement we shall use is from \cite[proposition 24]{bhargava_counting_cubic_field}:

\begin{prop}\label{bhargava_geometrie_nombre}
Let $\mathcal{D}$ be a bounded domain,  semi-algebraic of $\R^m$ which is defined by at most $k$ polynomial inequalities each having degree at most $l$.

Then the number of integral lattice points of $\Z^m$ contained in $\mathcal{D}$ is :

$$\Vol(\mathcal{D}) + O\bigl( \max(1,\Vol(\proj_I(\mathcal{D})) | I \in \mathfrak{P}(\{1,..,m\}),  1 \leqslant \sharp I < m) \bigr)$$where $\Vol(\proj_I(\mathcal{D}))$ means the euclidean volume of the projection of $\mathcal{D}$ via 
\begin{align*}
    \proj_I : &\R^m \rightarrow \R^{ I} \\
    (x_1,..,&x_m) \longmapsto (x_i)_{i \in I} .
\end{align*}
The constant in the error term depends only on $m$, $k$, and $l$.

\end{prop}

First, by proposition \ref{proposition_property_equivariant_extended_height}, we can write $$\mathcal{D}_B = B^{-u} . \mathcal{D}_1$$where $B^u \in T_{\NS,\orb}(\R)$ is the image of $\log(B) u$ via the exponential map
$$    \exp : \Pic_{\orb}(X)^{\vee}_{\R} \hookrightarrow \T_{\NS,\orb}(\R) \, , $$that is, $B^u$ is the element such that for any $\theta \in \Pic_{\orb}(X) = X^*(T_{\NS,\orb})$, we have:
$$\theta(B^{u}) = B^{u(\theta)} \in \G_m(\R) = \R^{\times}  .$$

In the rest of this article, we shall use the conventions of notation \ref{notation_class_orbifold_picard_group}. Note that, by Definition~\ref{definition_extended_universal_torsor}, $B^{-u}$ acts on $\stackT^{\ext}(\R)$ by multiplying the coordinate corresponding to $\rho \in \Sigma_{\ext}(1)$ by $B^{\langle [D_{\rho}],u \rangle}$. So using theorem \ref{theorem_compacity_pull-back-compact_ext_height}, we can write the following theorem:

\begin{theorem}\label{theorem_crucial_majoration_estimation_number}
    There exists a constant $C > 0$ (depending on the compact subset $\mathrm{D}$) such that for any $y \in \mathcal{D}_{B}$, we have for any $\rho \in \Sigma_{\ext}(1)$: $$ |y_{\rho}| \leqslant C B^{ \langle [D_{\rho}],u \rangle } .$$
  
\end{theorem}

We get the following corollary to estimate the error term:

\begin{cor}
    For any $I$ non empty subset of $\Sigma_{\ext}(1)$, there exists a constant $C > 0$ such that:
    $$\Vol(\proj_I(\mathcal{D}_{B})) \leq C B^{ \langle \sum\limits_{\rho \in I} [D_{\rho}] , u \rangle } .$$
\end{cor}

We finish this paragraph by recalling the value of the volume $\Vol(\mathcal{D}_{B}) $ and by estimating the cardinal of $$\sharp \Bigl( \Z^{\Sigma_{\ext}(1)} \cap T^{\ext}(\R) \Bigr)_{h^{\ext}_{\stackT,\infty} \in \mathrm{D}_{B}} .$$

\begin{theorem}
    $$\Vol(\mathcal{D}_{B}) =  \Vol(\mathcal{D}_1) B^{ \langle \w_{X,\orb}^{-1} , u \rangle } = 2^{\rho_{\orb}(X)}. \mu_{X,\infty}(X(\R)). \nu(\mathrm{D}_1) B^{ \langle \w_{X,\orb}^{-1} , u \rangle } . $$
\end{theorem}

\begin{proof}
    By definition \ref{definition_quotient_measure_over_orbifold}, the Lebesgue measure over $\R^{\Sigma_{\ext}(1)}$ restricted to $\stackT^{\ext}(\R)$ is $\w_{\stackT,\infty}$. So the theorem is a consequence of proposition \ref{analysis_measure_univ_torsor_at_infty}.
\end{proof}

Recall that as we take $u \in (\coneorb^{\vee})^{\circ}$  we have for any $$L \in \coneorb, \ \langle L,u \rangle > 0 .$$Then by proposition \ref{proposition_description_cone_orbifold_effective_divisor_toric_stack}, we have $[D_{\rho}] \in \coneorb$ for any $\rho \in \Sigma_{\ext}(1)$ so we get that: 

$$\max\limits_{ I \subsetneq \Sigma_{\ext}(1) } \Vol(\proj_I(\mathcal{D}_B)) = O\left( \frac{B^{\langle \w_{X,\orb}^{-1},u\rangle}}{B^{\min\limits_{\rho \in \Sigma_{\ext}(1)} \langle [D_{\rho}],u\rangle} } \right)$$which is negligible compared to $\Vol(\mathcal{D}_B)$ when $B$ grows. We can then deduce the following theorem using proposition \ref{bhargava_geometrie_nombre}:

\begin{theorem}\label{estimation_lattice}
    There exists a constant $C > 0$ such that for $B > 1$:
    \begin{align*}
        &\Big| \sharp \Bigl(  \Z^{\Sigma_{\ext}(1)} \cap T^{\ext}(\R) \Bigr)_{h^{\ext}_{\stackT,\infty} \in \mathrm{D}_{B}}  -   B^{ \langle \w_{X,\orb}^{-1} , u \rangle } . \Vol(\mathcal{D}_1) \Big| \\
        &\leqslant C \frac{B^{\langle \w_{X,\orb}^{-1},u\rangle}}{B^{\min\limits_{\rho \in \Sigma_{\ext}(1)} \langle [D_{\rho}],u\rangle} }
    \end{align*}
\end{theorem}

\subsubsection{Upper bound for the number of points in the complementary of the torus}

Applying theorem \ref{theorem_compacity_pull-back-compact_ext_height}, we have that there exists $N > 1$ such that $\mathcal{D}_1 \subset \prod\limits_{\rho \in \Sigma_{\ext}(1)} [-N,N]$. Hence $\mathcal{D}_B \subset  \prod\limits_{\rho \in \Sigma_{\ext}(1)} \left[ -N.B^{\langle [D_{\rho}] , u \rangle},N.B^{\langle [D_{\rho}] , u \rangle} \right] $, so we have for any $\rho \in \Sigma_{\ext}(1)$,
\begin{align*}
    \sharp\left(\Z^{\Sigma_{\ext}(1)} \cap \stackT^{\ext}(\R) \cap \{y_{\rho} = 0 \} \right)_{h^{\ext}_{\stackT,\infty} \in \mathrm{D}_B} &\leqslant \prod\limits_{\rho' \in \Sigma_{\ext}(1)-\{\rho\}} \left( 2.\lfloor N.B^{\langle [D_{\rho}] , u \rangle} \rfloor + 1 \right) \\
    &\leqslant \prod\limits_{\rho' \in \Sigma_{\ext}(1)-\{\rho\}} \left( 2 N.B^{\langle [D_{\rho}] , u \rangle} + 1 \right)
\end{align*}Hence there exists a constant $C > 0$ such that:
$$\sharp\left(\Z^{\Sigma_{\ext}(1)} \cap \stackT^{\ext}(\R) \cap \{y_{\rho} = 0 \} \right)_{h^{\ext}_{\stackT,\infty} \in \mathrm{D}_B} \leqslant C. B^{\langle \sum\limits_{\rho' \neq \rho} [D_{\rho'}] , u \rangle} .$$We get the following theorem:

\begin{theorem}\label{number_of_points_complementary_of_the_torus_universal_torsor}
    $$ \sharp\left(\Z^{\Sigma_{\ext}(1)} \cap (\stackT^{\ext}(\R) \setminus T^{\ext}(\R)) \right)_{h_{\stackT,\infty} \in \mathrm{D}_B} = O\left( \frac{B^{\langle \w_{X,\orb}^{-1} , u \rangle}}{B^{\min\limits_{\rho \in \Sigma_{\ext}(1)} \langle [D_{\rho}] , u \rangle }} \right)$$
\end{theorem}

Applying the previous theorem and theorem \ref{estimation_lattice}, we get that:

\begin{cor}
    There exists a constant $C > 0$ such that:
    $$\Big| \sharp \Bigl( \Z^{\Sigma_{\ext}(1)} \cap \stackT^{\ext}(\R) \Bigr)_{h^{\ext}_{\stackT,\infty} \in \mathrm{D}_{B}}  - 2^{\rho_{\orb}(X)} . \mu_{X,\infty}(X(\R)). \nu(\mathrm{D}_1) B^{ \langle \w_{X,\orb}^{-1} , u \rangle } \Big| \leqslant C \frac{B^{\langle \w_{X,\orb}^{-1},u\rangle}}{B^{ \min\limits_{\rho \in \Sigma_{\ext}(1)} \langle [D_{\rho}],u\rangle } } .$$

\end{cor}

As $\sharp\left(X(\Q) \setminus T(\Q)) \right)_{h \in \mathrm{D}_B} \leqslant \sharp\left(\Z^{\Sigma_{\ext}(1)} \cap (\stackT^{\ext}(\R) \setminus T^{\ext}(\R)) \right)_{h^{\ext}_{\stackT,\infty} \in \mathrm{D}_B}$,we can also deduce the following corollary:

\begin{cor}\label{number_of_points_complementary_of_the_torus_variety}
    $$ \sharp\left(X(\Q) \setminus T(\Q)) \right)_{h \in \mathrm{D}_B} = O\left( \frac{B^{\langle \w_{X,\orb}^{-1} , u \rangle}}{B^{\min\limits_{\rho \in \Sigma_{\ext}(1)} \langle [D_{\rho}] , u \rangle }} \right) .$$
\end{cor}

\subsection{Extension of our estimation in order to apply the Möbius inversion}

Let $d \in (\N^*)^{\Sigma_{\ext}(1)}$ and write $\Lambda_d = \bigoplus\limits_{\rho \in \Sigma_{\ext}(1)} d_{\rho}.\Z$. We shall denote $N_d(B)$ the cardinal of the set $\Big( \Lambda_d \cap T^{\ext}(\R) \Big)_{h^{\ext}_{\stackT,\infty} \in \mathrm{D}_{B}}$. We want to extend the estimation made in theorem \ref{estimation_lattice} to $N_d(B)$.

Denote by $\varphi : \R^{\Sigma_{\ext}(1)} \rightarrow \R^{\Sigma_{\ext}(1)}$ the diagonal endomorphism given by multiplication by $d_{\rho}$ the coordinate associated to $\rho \in \Sigma_{\ext}(1)$. Recall that we have $\covol(\Lambda_d) = \Pi(d)$. We can interpret $N_d(B)$ as follows:
    \begin{align*}
        N_d(B) &= \sharp \left\{ y \in \Z^{\Sigma_{\ext}(1)} \mid h^{\ext}_{\stackT,\infty}\left( \varphi(y)  \right) \in \mathrm{D}_B \text{ and } \varphi(y) \in T^{\ext}(\R) \right\} \\
        &= \sharp \left\{ y \in \Z^{\Sigma_ {\ext}(1)} \mid h^{\ext}_{\stackT,\infty}\left( B^{-u}. \varphi(y) \right) \in \mathrm{D} \text{ and } B^{-u}. \varphi(y) \in T^{\ext}(\R) \right\} \\
    \end{align*}
We write $\mathcal{D}_1 = \{ y \in T^{\ext}(\R) \mid h^{\ext}_{\stackT,\infty}(y) \in \mathrm{D} \}$. We can then deduce that $$N_d(B) = \sharp \left( \Z^{\Sigma_{\ext}(1)} \cap \Phi(\mathcal{D}_1) \right)$$ where $\Phi(z) = \varphi^{-1} ( B^u. z ) $. Using this description, we begin by giving an upper bound for $N_d(B)$:

\begin{prop}\label{upper_bound_N_d(B)}
    For any $B > 1$, we have:
    $$N_d(B) = O\left( \frac{B^{\langle \w_{X,\orb}^{-1} , u \rangle}}{\Pi(d)} \right) .$$
\end{prop}

\begin{proof}
    Again by theorem \ref{theorem_compacity_pull-back-compact_ext_height}, there exists $N > 1$ such that $\mathcal{D}_1 \subset \prod\limits_{\rho \in \Sigma_{\ext}(1)} \left( [-N,N] - \{0\} \right)$. Hence $\Phi(\mathcal{D}_1) $ is contained in $$ \prod\limits_{\rho \in \Sigma_{\ext}(1)} \left( \left[ -N.\frac{B^{\langle [D_{\rho}] , u \rangle}}{d_{\rho}},N.\frac{B^{\langle [D_{\rho}] , u \rangle}}{d_{\rho}} \right] - \{0\} \right) .$$ So we have that:
    \begin{align*}
        N_d(B) &= \sharp(\Z^{\Sigma_{\ext}(1)} \cap \Phi(\mathcal{D}_1)) \\
        &\leqslant 2^{\sharp \Sigma_{\ext}(1)}.\prod\limits_{\rho \in \Sigma_{\ext}(1)} \lfloor N.\frac{B^{\langle [D_{\rho}] , u \rangle}}{d_{\rho}} \rfloor \\
        &\leqslant (2.N)^{\sharp \Sigma_{\ext}(1)} . \frac{B^{\langle \w_{X,\orb}^{-1},u\rangle}}{\Pi(d)} .
    \end{align*}
\end{proof}

\begin{remark}\label{remark_nulitte_N_d(B)}
    One may remark using the previous proof that $N_d(B) = 0$ when there exists $\rho \in \Sigma_{\ext}(1)$ such that $N . B^{\langle [D_{\rho}] , u \rangle} < d_{\rho}$.
\end{remark}

We shall later need a better estimation of $N_d(B)$ when $\Pi(d) \leqslant B^{\min\limits_{\rho \in \Sigma_{\ext}(1)} \langle [D_{\rho}] , u \rangle}$. It is given by the following theorem:

\begin{theorem}\label{estimation_N_d(B)}
    There exists a constant $C > 0$ (depending on the choice of $\mathrm{D})$ such that:
    $$\left|  N_d(B) - \frac{1}{\Pi(d)}.2^{\rho_{\orb}(X)}.\mu_{X,\infty}(X(\R)). \nu(\mathrm{D}_1) B^{ \langle \w_{X,\orb}^{-1} , u \rangle } \right| \leqslant C .\frac{B^{\langle \w_{X,\orb}^{-1},u\rangle}}{B^{ \min\limits_{\rho \in \Sigma_{\ext}(1)} \langle [D_{\rho}],u\rangle } } .$$
\end{theorem}

\begin{proof}
    Denote by $\Phi_I : \R^I \rightarrow \R^I$ the morphism which maps $$(x_{\rho})_{\rho \in I} \mapsto ( \frac{B^{\langle [D_{\rho}] , u \rangle}}{d_{\rho}} . x_{\rho} )_{\rho \in I} .$$Then we may remark that:
  $$ \proj_I \circ \Phi = \Phi_I \circ \proj_I .$$Hence we get, using this, that there exists $C > 0$ (which depends only on the compact subset $\mathrm{D}$) such that for any non empty set $I \subsetneq \Sigma_{\ext}(1) $:
   $$ \Vol(\proj_I(\Phi(\mathcal{D})) \leqslant C . \frac{B^{\langle \sum\limits_{\rho \in I} [D_{\rho}] , u \rangle}}{\prod\limits_{ \rho \in I} d_{\rho}} \leqslant C.\frac{B^{\langle \w_{X,\orb}^{-1},u\rangle}}{B^{ \min\limits_{\rho \in \Sigma_{\ext}(1)} \langle [D_{\rho}],u\rangle } }.$$Applying proposition \ref{bhargava_geometrie_nombre} finishes the proof of the lemma.
   
\end{proof}

\subsection{The Möbius inversion}

We recall that we take $\mathrm{D}$ a finite union of compact polyhedrons of $\Pic_{\orb}(X)^{\vee}_{\R}$, $u \in (\coneorb^{\vee})^{\circ}$ and $B > 1$.

Using theorem \ref{theorem_mobius_inversion}, we have the following proposition:

\begin{prop}
    $$ \sharp( \stackT^{\ext}(\Z) \cap T^{\ext}(\R) )_{h^{\ext}_{\stackT,\infty} \in \mathrm{D}_B} = \sum\limits_{d \in (\N^*)^{\Sigma_{\ext}(1)}} \mu(d) . N_d(B)  $$
\end{prop}

\begin{proof}
    First, by remark \ref{remark_nulitte_N_d(B)}, the sum is finite. Now we have:
    \begin{align*}
        &\sharp( \stackT^{\ext}(\Z) \cap T^{\ext}(\R) )_{h^{\ext}_{\stackT,\infty} \in \mathrm{D}_B} = 2^{\sharp \Sigma_{\ext}(1)}. \sum\limits_{y \in (\N^*)^{\Sigma_{\ext}(1)}} \mathbf{1}_{( h^{\ext}_{\stackT,\infty})^{-1}(\mathrm{D}_B)}(y) . \mathbf{1}_{\stackT^{\ext}(\Z)}(y) \\
        &= 2^{\sharp \Sigma_{\ext}(1)}. \sum\limits_{y \in (\N^*)^{\Sigma_{\ext}(1)}} \mathbf{1}_{(h^{\ext}_{\stackT,\infty})^{-1}(\mathrm{D}_B)}(y) . \sum\limits_{d \in (\N^*)^{\Sigma_{\ext}(1)}} \mu(d) . \mathbf{1}_{d.\Z^{\Sigma_{\ext}(1)}}(y) \\
        &= \sum\limits_{d \in (\N^*)^{\Sigma_{\ext}(1)}} \mu(d) . \left( 2^{\sharp \Sigma_{\ext}(1)}. \sum\limits_{y \in (\N^*)^{\Sigma_{\ext}(1)}} \mathbf{1}_{(h^{\ext}_{\stackT,\infty})^{-1}(\mathrm{D}_B)}(y) . \mathbf{1}_{d.\Z^{\Sigma_{\ext}(1)}}(y) \right) \\
        &= \sum\limits_{d \in (\N^*)^{\Sigma_{\ext}(1)}} \mu(d) . N_d(B) .
    \end{align*}
\end{proof}

We shall show the following theorem:

\begin{theorem}
Let $\varepsilon \in ]0,\frac{1}{2}[$. There exists a constant $C > 0$ such that:
\begin{align*}
    &\left| \sharp( \stackT^{\ext}(\Z) \cap T^{\ext}(\R) )_{h^{\ext}_{\stackT,\infty} \in \mathrm{D}_B} - 2^{\rho_{\orb}(X)} . \prod\limits_p \w_{\stackT,p}(\stackT^{\ext}(\Z_p)) . \mu_{X,\infty}(X(\R)). \nu(\mathrm{D}_1) B^{ \langle \w_{X,\orb}^{-1} , u \rangle } \right| \\
    &\leqslant C . \frac{B^{\langle \w_{X,\orb}^{-1} , u \rangle}}{B^{(\frac{1}{2} - \varepsilon) \min\limits_{\rho \in \Sigma_{\ext}(1)} \langle [D_{\rho}] , u \rangle }}.
\end{align*}
\end{theorem}

\begin{proof}
    Recall that by proposition \ref{value_local_mesure_mobius_function} and theorem \ref{value_global_mesure_mobius_fonction}, we have that:
    $$\sum\limits_d \frac{\mu(d)}{\Pi(d)} = \prod\limits_p \w_{\stackT,p}(\stackT^{\ext}(\Z_p)) .$$ Hence using the previous proposition, we get:
    \begin{align*}
        &\left| \sharp( \stackT^{\ext}(\Z) \cap T^{\ext}(\R) )_{h^{\ext}_{\stackT,\infty} \in \mathrm{D}_B} - 2^{\rho_{\orb}(X)}. \prod\limits_p \w_{\stackT,p}(\stackT^{\ext}(\Z_p)) . \mu_{X,\infty}(X(\R)). \nu(\mathrm{D}_1) B^{ \langle \w_{X,\orb}^{-1} , u \rangle } \right| \\
        =& \left| \sum\limits_d \mu(d) \left( N_d(B) - 2^{\rho_{\orb}(X)} . \frac{1}{\Pi(d)} . \mu_{X,\infty}(X(\R)). \nu(\mathrm{D}_1) B^{ \langle \w_{X,\orb}^{-1} , u \rangle } \right) \right| \\
        \leqslant&  \sum\limits_{\Pi(d) \leqslant B^{\min\limits_{\rho \in \Sigma_{\ext}(1)} \langle [D_{\rho}] , u \rangle }} | \mu(d) | . \left| N_d(B) - 2^{\rho_{\orb}(X)} . \frac{1}{\Pi(d)} . \mu_{X,\infty}(X(\R)). \nu(\mathrm{D}_1) B^{ \langle \w_{X,\orb}^{-1} , u \rangle } \right| \\
        +& \sum\limits_{\Pi(d) \geqslant B^{\min\limits_{\rho \in \Sigma_{\ext}(1)} \langle [D_{\rho}] , u \rangle }} | \mu(d) | . \left| N_d(B) - 2^{\rho_{\orb}(X)} . \frac{1}{\Pi(d)} . \mu_{X,\infty}(X(\R)). \nu(\mathrm{D}_1) B^{ \langle \w_{X,\orb}^{-1} , u \rangle } \right|.
    \end{align*}
    When $\Pi(d) \leqslant B^{\min\limits_{\rho \in \Sigma_{\ext}(1)} \langle [D_{\rho}] , u \rangle }$ we use theorem \ref{estimation_N_d(B)}, we have that there exists a constant $C > 0$ such that:
    $$\left|  N_d(B) - \frac{1}{\Pi(d)}.2^{\rho_{\orb}(X)}.\mu_{X,\infty}(X(\R)). \nu(\mathrm{D}_1) B^{ \langle \w_{X,\orb}^{-1} , u \rangle } \right| \leqslant C .\frac{B^{\langle \w_{X,\orb}^{-1},u\rangle}}{B^{ \min\limits_{\rho \in \Sigma_{\ext}(1)} \langle [D_{\rho}],u\rangle } } .$$Let $\varepsilon \in ]0,\frac{1}{2}[$. Then by corollary \ref{majoration_fonction_de_mobius}, we get that:
    \begin{align*}
        &\sum\limits_{\Pi(d) \leqslant B^{\min\limits_{\rho \in \Sigma(1)} \langle [D_{\rho}] , u \rangle }} | \mu(d) | . \left| N_d(B) - 2^{\rho_{\orb}(X)} . \frac{1}{\Pi(d)} . \mu_{X,\infty}(X(\R)). \nu(\mathrm{D}_1) B^{ \langle \w_{X,\orb}^{-1} , u \rangle } \right| \\
         &\leqslant C.\frac{B^{\langle \w_{X,\orb}^{-1},u\rangle}}{B^{ \min\limits_{\rho \in \Sigma_{\ext}(1)} \langle [D_{\rho}],u\rangle } } . B^{(\frac{1}{2} + \varepsilon).\min\limits_{\rho \in \Sigma_{\ext}(1)} \langle [D_{\rho}],u\rangle }.
    \end{align*}
    Now by proposition \ref{upper_bound_N_d(B)}, there exists $C' > 0$ such that:
    $$N_d(B) \leqslant C'.\frac{B^{\langle \w_{X,\orb}^{-1} , u \rangle}}{\Pi(d)} .$$ Hence we get using the second assertion of lemma \ref{majoration_fonction_de_mobius} that there exists $C'' > 0$ such that:
    \begin{align*}
        &\sum\limits_{\Pi(d) \geqslant B^{\min\limits_{\rho \in \Sigma(1)} \langle [D_{\rho}] , u \rangle }} | \mu(d) | . \left| N_d(B) - 2^{\rho_{\orb}(X)} . \frac{1}{\Pi(d)} . \mu_{X,\infty}(X(\R)). \nu(\mathrm{D}_1) B^{ \langle \w_{X,\orb}^{-1} , u \rangle } \right| \\
        &\leqslant C''. B^{\langle \w_{X,\orb}^{-1},u\rangle} .\sum\limits_{\Pi(d) \geqslant B^{\min\limits_{\rho \in \Sigma(1)} \langle [D_{\rho}] , u \rangle }} \frac{| \mu(d) |}{\Pi(d)} \\
        &\leqslant C''.\frac{B^{\langle \w_{X,\orb}^{-1},u\rangle}}{B^{(1-\frac{1}{2} - \varepsilon) \min\limits_{\rho \in \Sigma_{\ext}(1)} \langle [D_{\rho}],u\rangle } } .
    \end{align*}This concludes our proof.    
\end{proof}

Applying theorem \ref{number_of_points_complementary_of_the_torus_universal_torsor}, we get:

\begin{cor}\label{conclusion_torsor_univ}
Let $\varepsilon \in ]0,\frac{1}{2}[$. There exists a constant $C > 0$ such that:
\begin{align*}
    &\left| (\stackT^{\ext}(\Z) )_{h^{\ext}_{\stackT,\infty} \in \mathrm{D}_B}  - 2^{\rho_{\orb}(X)} . \prod\limits_p \w_{\stackT,p}(\stackT^{\ext}(\Z_p)) . \mu_{X,\infty}(X(\R)). \nu(\mathrm{D}_1) B^{ \langle \w_{X,\orb}^{-1} , u \rangle } \right| \\
    &\leqslant C . \frac{B^{\langle \w_{X,\orb}^{-1} , u \rangle}}{B^{(\frac{1}{2} - \varepsilon) \min\limits_{\rho \in \Sigma_{\ext}(1)} \langle [D_{\rho}] , u \rangle }}
\end{align*}
\end{cor}

\subsection{Conclusion}

Finally we can establish the theorem \ref{theorem_manin_peyre_conjecture}:

\begin{theorem}\label{theorem_final_champ_torique}
Let $\epsilon \in ]0,\frac{1}{2}[$. Then we have:
$$ \sharp ( X(\Q) )_{h \in \mathrm{D}_B} \underset{B \rightarrow +\infty}{=} \nu(\mathrm{D}_1 ) . \tau_{\orb}(X) .  B^{\langle \w_{X,\orb}^{-1},u \rangle} \left( 1 + O\left(B^{- ( \frac{1}{2} - \varepsilon).\min\limits_{\rho \in \Sigma_{\ext}(1)} \langle [D_{\rho}] , u \rangle} \right) \right)$$where $\nu$ is the measure over $\Pic_{\orb}(X)^{\vee}_{\R}$ defined as in \ref{mesure_picard_group}.
\end{theorem}

\begin{proof}
      Let us recall that $\sharp X(\Q)_{h \in \mathrm{D}_B} = \frac{\sharp G (\Z)}{2^{\rho_{\orb}(X)}} \sharp \left( \stackT^{\ext}(\Z) \right)_{h^{\ext}_{\stackT,\infty} \in \mathrm{D}_B}$. Applying corollary \ref{conclusion_torsor_univ}, we get that there exists a constant $C > 0$ such that:
      \begin{align*}
          &\left| (\sharp X(\Q)_{h \in \mathrm{D}_B}  - \frac{1}{\sharp G (\Z)} .\left( \prod\limits_p \w_{\stackT,p}(\stackT^{\ext}(\Z_p)) \right) . \mu_{X,\infty}(X(\R)). \nu(\mathrm{D}_1) B^{ \langle \w_{X,\orb}^{-1} , u \rangle } \right| \\
          & \leqslant C . \frac{B^{\langle \w_{X,\orb}^{-1} , u \rangle}}{B^{(\frac{1}{2} - \varepsilon) \min\limits_{\rho \in \Sigma_{\ext}(1)} \langle [D_{\rho}] , u \rangle }} .
      \end{align*}
       Moreover by proposition \ref{proposition_finite_place_quotient_measure_over_X} and definition \ref{definition_orbifold_tamagawa_number}, we get: $$ \frac{1}{\sharp G (\Z)} . \left( \prod\limits_p \w_{\stackT,p}(\stackT^{\ext}(\Z_p)) \right) . \mu_{X,\infty}(X(\R)) = \tau_{\orb}(X)$$hence we have the result stated.
\end{proof}

\newpage

\appendix

\section{Complements on toric stacks}

In this appendix, we provide the details of a technical lemma used in the final argument of the proof of Proposition~\ref{prop_age_divisor_associated_to_an_edge} and in the proof of Theorem~\ref{theorem_computation_orbifold_picard_group_toric_stack_appendix}. In a second part, we give a sufficient condition for the $\mathcal{C}$-extended Picard group (see Definition~\ref{definition_orbifold_picard_group}) $\Pic_{\mathcal{C}}(X)$ to be torsion-free. Finally, we present a detailed proof of Theorem~\ref{theorem_action_orbi_neron_severi_torus}, making explicit why our conventions for defining the orbifold Picard group (see Definition~\ref{definition_orbifold_picard_group}) differ from those of \cite[Definition 15]{Coates_Corti_Iritani_Tseng_miror_theorem_toric_stack}.

\subsection{A proof of diagram chasing}

Suppose we are given a short exact sequence
\begin{equation}\label{equation_appendix_a_lemma_technical}
 0 \longrightarrow \Z^m \xrightarrow{f} B \xrightarrow{\pi} C \longrightarrow 0,
\end{equation}
where $B$ is a finitely generated abelian group and $C$ is a finite abelian group. We aim to describe the morphism
\[
\Z^m \longrightarrow \Ext^1(C,\Z)
\]obtained by applying the functor $\RHom(-,\Z)$ to the above exact sequence. To this end, recall that there is a natural isomorphism
\[
\Hom(C,\Q/\Z) \xrightarrow{\ \sim\ } \Ext^1(C,\Z),
\]
obtained by applying the functor $\RHom(C,-)$ to the short exact sequence
\[
0 \longrightarrow \Z \longrightarrow \Q \longrightarrow \Q/\Z \longrightarrow 0.
\]We may now state the following lemma:

\begin{lemma}\label{lemma_diagramme_chasing_appendix}
The map
\[
\Z^{m} \longrightarrow \Ext^1(C,\Z)
\]
is given, via the canonical isomorphism
\[
\Hom(C,\Q/\Z) \simeq \Ext^1(C,\Z),
\]
by the morphism which sends $x \in \Z^m$ to the element of $\Hom(C,\Q/\Z)$ defined as follows:
\[
c \longmapsto \left\langle (f_{\Q}^*)^{-1} (x), b_c \right\rangle \bmod \Z 
\]where $c \mapsto b_c$ is a set-theoretic section of $\pi$. The definition is independent of the choice of the lift.
\end{lemma}

\begin{proof}
Using the exact sequence of equation \ref{equation_appendix_a_lemma_technical} together with the exact sequence
\[
0 \rightarrow \Z \rightarrow \Q \rightarrow \Q/\Z \rightarrow 0
\]
and the functoriality of the functor $\mathbf{RHom}(-,-)$, we obtain the following commutative diagram:

\begin{center}
\begin{tikzcd}
            &                                                   &                                                           & 0 \arrow[d]                                                 &                                \\
            &                                                   &                                                           & {\mathrm{Hom}(C,\Q/\Z)} \arrow[d]                           &                                \\
0 \arrow[r] & {\mathrm{Hom}(B,\Z)} \arrow[r] \arrow[d, "a"]     & {\mathrm{Hom}(B,\Q)} \arrow[r, "p"] \arrow[d, "f_{\Q}^*"] & {\mathrm{Hom}(B,\Q/\Z)} \arrow[d] \arrow[r, "\mathrm{Res}"] & {\mathrm{Hom}(B_{\tor},\Q/\Z)} \\
0 \arrow[r] & {\mathrm{Hom}(\Z^m,\Z)= \Z^m} \arrow[r] \arrow[d] & {\mathrm{Hom}(\Z^m,\Q)} \arrow[r]                         & {\mathrm{Hom}(\Z^m,\Q/\Z)} \arrow[r]                        & 0                              \\
            & {\mathrm{Ext}^1(C,\Z)} \arrow[d]                  &                                                           &                                                             &                                \\
            & {\mathrm{Ext}^1(B,\Z)} \arrow[d]                  &                                                           &                                                             &                                \\
            & 0                                                 &                                                           &                                                             &                               
\end{tikzcd}
\end{center}If we replace $\Hom(B,\Q/\Z)$ by $\mathrm{Im}(p)$, the induced map
\[
\mathrm{Im}(p) \xrightarrow{c} \Hom(\Z^m,\Q/\Z)
\]
has kernel equal to the kernel of the natural map
\[
\Hom(C,\Q/\Z) \longrightarrow \Hom(B_{\tor},\Q/\Z).
\]
We may therefore write the following commutative diagram:
\begin{center}
\begin{tikzcd}
0 \arrow[r] & {\mathrm{Hom}(B,\Z)} \arrow[r] \arrow[d, "a"]               & {\mathrm{Hom}(B,\Q)} \arrow[r] \arrow[d, "f_{\Q}^*"] & \mathrm{Im}(p) \arrow[d, "c"] \arrow[r]        & 0 \\
0 \arrow[r] & {\mathrm{Hom}(\Z^m,\Z)= \Z^m} \arrow[r] & {\mathrm{Hom}(\Z^m,\Q)} \arrow[r]                         & {\mathrm{Hom}(\Z^m,\Q/\Z)} \arrow[r] & 0
\end{tikzcd}   
\end{center}to which the snake lemma can be applied. The snake lemma implies that the map
$$\Z^m \rightarrow \coker(a) = \ker\left(\Ext^1(C,\Z) \rightarrow \Ext^1(B_{\tor},\Z) \right) \hookrightarrow  \Ext^1(C,\Z) $$is given by the map
\[
\Z^m \longrightarrow \ker(c) =
\ker\!\left(
\Hom(C,\Q/\Z) \longrightarrow \Hom(B_{\tor},\Q/\Z)
\right)
\]
which sends $x \in \Z^{m}$ to the morphism
\[
c \in C \longmapsto \left\langle (f_{\Q}^*)^{-1} (x), b_c \right\rangle \bmod \Z,
\]
where $c \mapsto b_c \in B$ is any set-theoretic section of $\pi$. By \cite[Addendum 1.3.3]{weibel1994homological}, the isomorphism
\[
\ker(c) \longrightarrow \coker(a)
\]
is the restriction of the isomorphism
\[
\Hom(C,\Q/\Z) \longrightarrow \Ext^1(C,\Z).
\]
\end{proof}

We recall that the toric stack $X$ (see Definition \ref{definition_stacky_fan_toric_stack}) is defined from a stacky fan given by a triple
\[
\mathbf{\Sigma} = (\Sigma, \Z^{\Sigma(1)} \xrightarrow{\beta} N),
\]
where $N$ is a finitely generated abelian group and $\Sigma$ is a rational complete simplicial fan in $N_{\Q} = N \otimes_{\Z} \Q$ with set of rays $\Sigma(1)$. We assume that, denoting by $\overline{\beta}$ the composition of $\beta$ with the natural map $N \rightarrow N_{\Q}$, one has $\overline{\beta}(e_{\rho}) \in \rho$ for all $\rho \in \Sigma(1)$. Let $\sigma \in \Sigma_{\max}$ be a maximal cone, and use the notation of \ref{notation_stacky_fan_beta}. In particular, for $\rho \in \Sigma(1)$, we set $b_{\rho} = \beta(e_{\rho})$, and the morphism
\[
\beta_{\sigma} : \Z^{\sigma(1)} \longrightarrow N
\]
is the restriction of $\beta$ to $\Z^{\sigma(1)}$. We now give the argument that completes the proof of Proposition~\ref{prop_age_divisor_associated_to_an_edge}. For this purpose, we determine the natural map
\[
\Z^{\sigma(1)} \longrightarrow \Ext^1(N(\sigma),\Z)
\]
induced by applying the functor $\mathbf{RHom}(-,\Z)$ to the exact sequence
\begin{equation}\label{equation_exact_sequence_proof_proposition}
0 \longrightarrow \Z^{\sigma(1)} \xrightarrow{\beta_{\sigma}} N \longrightarrow N(\sigma) \longrightarrow 0   
\end{equation}where $N(\sigma)$ is finite by our assumption on $\beta$.

\begin{prop}\label{proposition_diagramme_chasing_appendix}
    The map $$\Z^{\sigma(1)} \rightarrow \Ext^1(N(\sigma),\Z) $$sends a vector of the canonical basis $e'_{\rho} \in \Z^{\sigma(1)}$ to the morphism $$n \in N(\sigma) \longmapsto \langle b_{\rho}^* , n \rangle \mod \Z $$where $(b_{\rho}^*)_{\rho \in \sigma(1)}$ is the dual basis of the basis $(\overline{b_{\rho}})_{\rho \in \sigma(1)}$ of $N_{\Q}$.
\end{prop}

\begin{proof}
It suffices to apply Lemma~\ref{lemma_diagramme_chasing_appendix} to the exact sequence \eqref{equation_exact_sequence_proof_proposition}. To conclude, we see that
\[
(f_{\Q}^*)^{-1}(e'_{\rho}) = b_{\rho}^*,
\]
where $(b_{\rho}^*)_{\rho \in \sigma(1)}$ denotes the dual basis of the basis $(\overline{b_{\rho}})_{\rho \in \sigma(1)}$ of $N_{\Q}$.
\end{proof}

\subsection{Property of the extended Picard group}

Let $\mathcal{C} \subset \twistsector$. We establish a criterion to determine when the extended Picard group $\Pic_{\mathcal{C}}(X)$ (see Definition \ref{definition_orbifold_picard_group}) is torsion free. We have the following natural map:
\begin{align*}
\Pic_{\mathcal{C}}(X) \longrightarrow &\Pic(X) / \Pic(X)_{\tor} \times \Hom(\mathcal{C},\Q) \\
 (L,\varphi) \longmapsto &\left( \overline{L}  , \varphi \right) \ .
\end{align*}

\begin{theorem}\label{appendix_a_theorem_freeness_orbifold_picard_group}
    We identify $\sector$ with $\Boxe(\Sigma)$. If there exists $\sigma \in \Sigma_{\max}$ such that the image of $\mathcal{C} \cap \Boxe(\sigma)$ in $N(\sigma)$ is a generating set, the previous map is injective and in particular, $$\Pic_{\mathcal{C}}(X)$$ is torsion free.
\end{theorem}

\begin{proof}
Let $L \in \Pic(X)_{\tor}$ such that for any $\stackY \in \mathcal{C}$, $\age(\stackY,L) = 0$. We keep the notation introduced after the diagram
\eqref{equation_morphism_triangle_distingue_associe_open_immersion}. Let
$p_{\sigma}(L)$ denote the image of $L$ under the map
\[
p_{\sigma} : \Pic(X) \longrightarrow \Ext^1(N(\sigma),\Z).
\]

By Proposition~\ref{prop_age_divisor_associated_to_an_edge}, the pairing
between the element $p_{\sigma}(L)$ and $b \in N(\sigma)$ is given by
\[
\langle p_{\sigma}(L), b \rangle = -\age(\stackS_b,L),
\]
where we identify $b \in N(\sigma)$ with its unique lift in $\Boxe(\sigma)$,
which we also denote by $b$, and we write $\stackS_b$ for the sector
associated with $b$ as in Theorem~\ref{theorem_description_sector_box_element}.

Thus, using the assumption that the image of
$\mathcal{C} \cap \Boxe(\sigma)$ in $N(\sigma)$ is a generating set,
together with the fact that the pairing between $N(\sigma)$ and
$\Ext^1(N(\sigma),\Z)$ is perfect, we obtain that $p_{\sigma}(L)=0$. Now note that we have the following morphism of exact triangles:
\begin{center}
\begin{tikzcd}
\ker(\beta) \arrow[r, "0"] & \coker(\beta) \arrow[r]       & \coneform(\beta)                    \\
0 \arrow[u] \arrow[r]      & N(\sigma) \arrow[r] \arrow[u] & \coneform(\beta_{\sigma}) \arrow[u]
\end{tikzcd}.
\end{center}

Applying the functor $\RHom(-,\Z)$ we get the following commutative diagram where the top horizontal map is exact:
\begin{center}
\begin{tikzcd}
0 \arrow[r] & {\mathbf{R^0Hom}(\ker(\beta),\Z)} \arrow[r] & {\mathbf{R^1 Hom}(\coneform(\beta),\Z)} \arrow[r] \arrow[d, "p_{\sigma}"] & {\mathbf{R^1Hom}(\coker(\beta),\Z)} \arrow[d] \\
            &                                             & {\mathbf{R^1 Hom}(\coneform(\beta_{\sigma}),\Z)} \arrow[r,equal]                & {\Ext^1(N(\sigma),\Z)}                       
\end{tikzcd}.
\end{center}Now because the map $N(\sigma) \rightarrow \coker(\beta)$ is surjective with kernel a finite abelian group, we get that the map $\mathbf{R^1Hom}(\coker(\beta),\Z) \rightarrow \Ext^1(N(\sigma),\Z)$ is injective. Hence the kernel of $p_{\sigma}$ is $\mathbf{R^0Hom}(\ker(\beta),\Z)$ a torsion free abelian group. So as $L$ is torsion and in the kernel of $p_{\sigma}$, we get that $L = 0$.
\end{proof}

\begin{cor}\label{corollaire_orbifold_picard_group_torsion_free}
    $\Pic_{\orb}(X)$ is torsion free.
\end{cor}

\subsection{Computation of the orbifold Picard group of a toric stack}

In this paragraph, we present a proof of Theorem~\ref{theorem_action_orbi_neron_severi_torus}. This result is due to Coates, Corti, Iritani, and Tseng (see \cite[Proposition 21]{Coates_Corti_Iritani_Tseng_miror_theorem_toric_stack}), and the proof we give is inspired by theirs, which was communicated to us by Iritani. However, since the definition of the orbifold Picard group used in this text (see Definition \ref{definition_orbifold_picard_group}) differs from that given in \cite[Definition 15]{Coates_Corti_Iritani_Tseng_miror_theorem_toric_stack}, we provide a complete proof in order to explain our choice of definition. We recall the result, using again the notations from \ref{notation_orbifold_T_invariant_divisor}.

\begin{theorem}\label{theorem_computation_orbifold_picard_group_toric_stack_appendix}
    We have an isomorphism:
    $$ \Pic_{\mathcal{C}}(X) \simeq \mathbf{R^1 Hom}(\coneform(\beta_{\mathcal{C}}),\Z).$$Moreover, the corresponding map $$ \Z^{\Sigma(1)} \oplus \Z^{\mathcal{C} } \xrightarrow{\pi} \Pic_{\mathcal{C}}(X)$$is obtained by mapping $e_{\rho}$ to $[D_{\rho}]_{0}$ for $\rho \in \Sigma(1)$ and $e_{\stackS}$ to $[\stackS]$ for $\stackS \in \mathcal{C}$.
\end{theorem}

We begin by recalling the construction from \cite{Coates_Corti_Iritani_Tseng_miror_theorem_toric_stack}  of the map
$$\mathbf{R^1 Hom}(\coneform(\beta_{\mathcal{C}}),\Z) \rightarrow \Pic(X) \oplus \Q^{\mathcal{C}}$$
which allows one to identify $\mathbf{R^1 Hom}(\coneform(\beta_{\mathcal{C}}),\Z)$ with $\Pic_{\mathcal C}(X)$. The exact sequences associated with $\beta$ and $\beta_{\mathcal{C}}$ fit into the following commutative diagram:

\begin{equation}\label{equation_diagram_computation_orbifold_picard_group_appendix}
\begin{tikzcd}
0 \arrow[r] 
& \ker(\beta) \arrow[r] \arrow[d]
& \ker(\beta_{\mathcal{C}}) \arrow[r] \arrow[d]
& \Z^{\mathcal{C}} \arrow[d, equal] \arrow[r]
& 0 \\
0 \arrow[r]
& \Z^{\Sigma(1)} \arrow[r] \arrow[d, "\beta"']
& \Z^{\Sigma(1) \cup \mathcal{C}} \arrow[r] \arrow[d, "\beta_{\mathcal{C}}"']
& \Z^{\mathcal{C}} \arrow[r] \arrow[d]
& 0 \\
0 \arrow[r]
& N \arrow[r, equal]
& N \arrow[r]
& 0
\end{tikzcd}
\end{equation}with exact rows and columns. We consider a splitting of the first row over the rational numbers given by the section
\[
\mu \colon \Q^{\mathcal{C}} \longrightarrow \ker(\beta_{\mathcal{C}})\otimes\Q
\]
defined by sending the canonical basis vector associated with $\stackS \in \mathcal{C}$ to
\[
e_{\stackS} - \psi_{\stackS}
\;\in\;
\ker(\beta_{\mathcal{C}})\otimes\Q
\subset
\Q^{\Sigma(1)\cup\mathcal{C}},
\]
where $\psi_{\stackS} = \sum\limits_{\rho \in \sigma_\stackS(1)} q_{\rho}(\stackS) e_{\rho} $ with $q_{\rho}(\stackS) = \age_{\num}(\stackS,-[D_{\rho}])$. This map is well defined by
Proposition~\ref{prop_age_divisor_associated_to_an_edge}. Taking duals, $\mu$ induces a map
\begin{equation}\label{equation_definition_application_appendix_a_3}
\begin{aligned}
f :
\mathbf{R^1 Hom}(\coneform(\beta_{\mathcal{C}}),\Z)
\rightarrow
\mathbf{R^1 Hom}(\coneform(\beta_{\mathcal{C}}),\Z)
=
&\ker(\beta_{\mathcal{C}})^{\vee}_{\Q}
\xrightarrow{\ \mu^*\ }
\Q^{\mathcal{C}} \\
&\varphi \longmapsto \left( \varphi(e_{\stackS} - \psi_{\stackS}) \right)_{\stackS \in \mathcal{C}}
\end{aligned}
\end{equation}
Together with the canonical morphism
\[
\mathbf{R^1 Hom}(\coneform(\beta_{\mathcal{C}}),\Z)
\rightarrow
\mathbf{R^1 Hom}(\coneform(\beta),\Z)= \Pic(X),
\]
we obtain a map
\begin{equation}\label{equation_definition_morphism_appendix}
\mathbf{R^1 Hom}(\coneform(\beta_{\mathcal{C}}),\Z) \rightarrow \Pic(X) \oplus \Q^{\mathcal{C}} \, .    
\end{equation}

In order to describe this map, we fix a resolution of $N$
\[
0 \rightarrow F_1 \xrightarrow{\iota} F_0 \rightarrow N \rightarrow 0,
\]
where the $F_i$ are finitely generated free abelian groups. We also fix a lift
\[
\widehat{\beta} : \Z^{\Sigma(1)} \rightarrow F_0
\]
of $\beta$, and similarly a lift
\[
\widehat{\beta}_{\mathcal{C}} : \Z^{\Sigma(1)\cup\mathcal{C}} \rightarrow F_0
\]
of $\beta_{\mathcal{C}}$ coinciding with $\widehat{\beta}$ on $\Z^{\Sigma(1)}$. We set, for any $\rho \in \Sigma(1)$, $\widehat{\beta}(e_{\rho}) = \widehat{b}_{\rho}$, and for any $\stackS \in \mathcal{C}$, $\widehat{\beta}_{\mathcal{C}}(e_{\stackS}) = \widehat{b}_{\stackS}$.

With these notations, we have the following isomorphisms in $\mathcal{D}(\mathrm{Ab})$ (see \cite[Section 2]{borisov_toric_stack}):
\[
\coneform(\beta) \simeq \left[ \Z^{\Sigma(1)} \oplus F_1 \xrightarrow{\widehat{\beta}\oplus \iota} F_0 \right] = \coneform(\widehat{\beta}\oplus \iota),
\]
and
\[
\coneform(\beta_{\mathcal{C}}) \simeq \left[ \Z^{\Sigma(1)\cup\mathcal{C}} \oplus F_1 \xrightarrow{\widehat{\beta}_\mathcal{C}\oplus \iota} F_0 \right] = \coneform(\widehat{\beta}_\mathcal{C}\oplus \iota).
\]

The idea of the proof of \cite[Proposition 21]{Coates_Corti_Iritani_Tseng_miror_theorem_toric_stack} is to use this description to prove Theorem~\ref{theorem_computation_orbifold_picard_group_toric_stack_appendix}. Indeed, we use that the maps induced by the functor $\mathbf{RHom}(-,\Z)$,
\[
\Z^{\Sigma(1)} \oplus F_1 \rightarrow \mathbf{R^1Hom}(\coneform(\widehat{\beta}\oplus \iota),\Z),
\]
which we denote by $(a,b) \mapsto [a,b]$, and
\[
\Z^{\Sigma(1)\cup\mathcal{C}} \oplus F_1 \rightarrow \mathbf{R^1 Hom}(\coneform(\widehat{\beta}_\mathcal{C}\oplus \iota),\Z),
\]
which we denote by $(\Tilde{a},b) \mapsto [\Tilde{a},b]$, are surjective.

We can now state the following lemma, which provides an explicit description of the morphism defined in \eqref{equation_definition_morphism_appendix}.

\begin{lemma}\label{lemma_appendix_a_description_map_isomorphism}
Under the natural identifications
\[
\mathbf{R^1 Hom}(\coneform(\widehat{\beta}_\mathcal{C}\oplus \iota),\Z) \simeq \mathbf{R^1 Hom}(\coneform(\beta_{\mathcal{C}}),\Z)
\]
and
\[
\mathbf{R^1 Hom}(\coneform(\widehat{\beta}\oplus \iota),\Z) \simeq \mathbf{R^1 Hom}(\coneform(\beta),\Z),
\]
the map
\[
\mathbf{R^1 Hom}(\coneform(\beta_{\mathcal{C}}),\Z) \longrightarrow \Pic(X) \oplus \Q^{\mathcal{C}}
\]
is given by
\[
[\Tilde{a},b] \longmapsto \left( [a,b] , \left( \Tilde{a}_{\stackS} - \langle a , \psi_{\stackS} \rangle - \langle b , x_{\stackS} \rangle \right)_{\stackS \in \mathcal{C}} \right),
\]
where $a \in \Z^{\Sigma(1)}$ is the projection of $\Tilde{a} \in \Z^{\Sigma(1) \cup \mathcal{C}}$, and $x_{\stackS} = \widehat{\beta}_\mathcal{C}(e_{\stackS} - \psi_{\stackS}) \in (F_{1})_{ \Q}$.
\end{lemma}

\begin{proof}
First, we have the following commutative diagram in $\mathcal{D}(\mathrm{Ab})$:
\begin{center}
\begin{tikzcd}
\coneform(\beta) \arrow[d] \arrow[r]     & \coneform(\widehat{\beta}\oplus \iota) \arrow[d]     \\
\coneform(\beta_{\mathcal{C}}) \arrow[r] & \coneform(\widehat{\beta}_{\mathcal{C}}\oplus \iota)
\end{tikzcd}
\end{center}
where the horizontal arrows are isomorphisms. Applying the functor $\mathbf{RHom}(-,\Z)$ and using the identifications given in the statement of the lemma, we obtain that the map
\[
\mathbf{R^1 Hom}(\coneform(\beta_{\mathcal{C}}),\Z)
\longrightarrow
\mathbf{R^1 Hom}(\coneform(\beta),\Z) = \Pic(X)
\]
is given by
\[
[\Tilde{a},b] \longmapsto [a,b],
\]
where $a \in \Z^{\Sigma(1)}$ is the projection of $\Tilde{a} \in \Z^{\Sigma(1) \cup \mathcal{C}}$.

Next, note that the isomorphism
\[
\ker(\beta_{\mathcal{C}}) \longrightarrow \ker(\widehat{\beta}_\mathcal{C}\oplus \iota)
\]
is given by $x \mapsto (x,-\widehat{\beta}_\mathcal{C}(x))$, and set $x_{\stackS} = \widehat{\beta}_\mathcal{C}(e_{\stackS} - \psi_{\stackS}) \in (F_{1})_{\Q}$. By the definition of the map
\[
\mathbf{R^1 Hom}(\coneform(\beta_{\mathcal{C}}),\Z) \longrightarrow \Q^{\mathcal{C}},
\]
see \eqref{equation_definition_application_appendix_a_3}, we obtain, under our identification, the description
\[
[\Tilde{a},b] \longmapsto \left( \left\langle (\Tilde{a},b) , e_{\stackS} - \psi_{\stackS} - x_{\stackS} \right\rangle \right)_{\stackS \in \mathcal{C}}.
\]
This concludes the proof.
\end{proof}

In particular, using this description, we deduce the following proposition:

\begin{prop}
For $\rho \in \Sigma(1)$ and $\stackS \in \mathcal{C}$, the images of the canonical basis vectors $e_{\rho}$ and $e_{\stackS}$ of $\Z^{\Sigma(1) \cup \mathcal C}$ under the composition
\[
\mathbf{R^1 Hom}(\coneform(\beta_{\mathcal{C}}),\Z) \longrightarrow \Pic(X) \oplus \Q^{\mathcal{C}}
\]
are respectively $[D_{\rho}]_0$ and $[\stackS]$ for $\rho \in \Sigma(1)$ and $\stackS \in \mathcal{C}$.
\end{prop}

\begin{proof}
By Lemma~\ref{lemma_appendix_a_description_map_isomorphism}, for $\rho \in \Sigma(1)$, the image of $e_{\rho}$ is given by
\[
\left( [e_{\rho},0] , \left(- q_{\rho}(\stackS) \right)_{\stackS \in \mathcal{C}} \right)
= \left( [D_{\rho}] , \left(- \age_{\num}(\stackS,-[D_{\rho}]) \right)_{\stackS \in \mathcal{C}} \right)
= [D_{\rho}]_0,
\]
and for $\stackS \in \mathcal{C}$, the image of $e_{\stackS}$ is given by
\[
\left( [0,0] , \left(\delta_{\stackS}(\stackS') \right)_{\stackS' \in \mathcal{C}} \right)
= [\stackS].
\]
\end{proof}

We now use our resolution of $N$ to compute the age pairing.

\begin{lemma}\label{lemma_appendix_computation_age_pairing}
For any $(a,b) \in \Z^{\Sigma(1)} \oplus F_1$ and any $\stackS \in \mathcal{C}$, we have
\[
\age\left( \stackS , [a,b] \right) = - \langle a , \psi_{\stackS} \rangle - \langle b , x_{\stackS} \rangle \bmod{\Z}.
\]
\end{lemma}

\begin{remark}
This is where our result differs from the proof of \cite[Proposition 21]{Coates_Corti_Iritani_Tseng_miror_theorem_toric_stack}.
\end{remark}

For $\sigma \in \Sigma_{\max}$, we denote by $\widehat{\beta}_{\sigma}$ the restriction of $\widehat{\beta}$ to $\Z^{\sigma(1)} \oplus F_1$.

\begin{proof}
First, we have the following commutative diagram in $\mathcal{D}(\mathrm{Ab})$:
\begin{center}
\begin{tikzcd}
N(\sigma) \simeq \coneform(\beta_{\sigma}) \arrow[d] \arrow[r]     & \coneform(\widehat{\beta}_{\sigma}\oplus \iota) \arrow[d]     \\
\coneform(\beta) \arrow[r] & \coneform(\widehat{\beta}\oplus \iota)
\end{tikzcd}
\end{center}
where the horizontal arrows are isomorphisms. Applying the functor $\mathbf{RHom}(-,\Z)$ and arguing as in Proposition~\ref{prop_expression_age_sector_toric_stack}, the composition of the map
\[
\mathbf{R^1 Hom}(\coneform(\widehat{\beta}\oplus \iota),\Z) \simeq \Pic(X) \longrightarrow \Ext^1(N(\sigma),\Z)
\]
with the morphism
\[
\Ext^1(N(\sigma),\Z) \longrightarrow \frac{1}{r_{\stackS}} \Z / \Z
\]
associated with $\stackS \in \twistsector \cap \Boxe(\sigma)$ corresponds to the map
\[
[a,b] \longmapsto - \age(\stackS,[a,b]).
\]

We now determine this map by another method. Applying the functor $\mathbf{RHom}(-,\Z)$, we obtain the following commutative diagram:
\begin{center}
\begin{tikzcd}
\Z^{\Sigma(1)} \oplus F_1 \arrow[d] \arrow[r] & \Pic(X) \arrow[d]      \\
\Z^{\sigma(1)} \oplus F_1 \arrow[r]           & {\Ext^1(N(\sigma),\Z)}
\end{tikzcd}
\end{center}
where the bottom horizontal map will be computed. Indeed, it is enough to apply Lemma~\ref{lemma_diagramme_chasing_appendix} to the exact sequence
\[
0 \rightarrow \Z^{\sigma(1)} \oplus F_1 \xrightarrow{\widehat{\beta}_{\sigma} \oplus \iota} F_0 \rightarrow N(\sigma) \rightarrow 0,
\]
where the map $F_0 \rightarrow N(\sigma)$ is the composition $F_0 \rightarrow N \rightarrow N(\sigma)$. We then obtain that this map is given by sending $(a_{\sigma},b) \in \Z^{\sigma(1)} \oplus F_1$ to
\[
c \mapsto \left\langle (\widehat{\beta}_{\sigma} \oplus \iota)^*_{\Q}{}^{-1}(a_{\sigma},b), n_c \right\rangle
\]where $c \mapsto n_c$ is a set-theoretic section of the surjection $F_0 \rightarrow N(\sigma)$. For $\stackS \in \mathcal{C} \cap \Boxe(\sigma)$, we denote by $c_{\stackS} \in N(\sigma)$ the corresponding element. Recall that
\[
\widehat{\beta}_{\mathcal{C}}(e_{\stackS}) = \widehat{b}_{\stackS} \in F_0
\]
is a lift of $c_{\stackS}$. Therefore,
\begin{align*}
    \age(\stackS,[a,b]) &= - \left\langle (\widehat{\beta}_{\sigma} \oplus \iota)^*_{\Q}{}^{-1}(a_{\sigma},b), \widehat{b}_{\stackS} \right\rangle \bmod{\Z} \\
    &= - \left\langle (a_{\sigma},b), (\widehat{\beta}_{\sigma} \oplus \iota)_{\Q}^{-1}(\widehat{b}_{\stackS}) \right\rangle \bmod{\Z} \\
    &= - \left\langle (a_{\sigma},b), \psi_{\stackS} + x_{\stackS} \right\rangle \bmod{\Z} \\
    &= - \langle a , \psi_{\stackS} \rangle - \langle b , x_{\stackS} \rangle \bmod{\Z}
\end{align*}where $a_{\sigma} \in \Z^{\sigma(1)}$ is the projection of $a \in \Z^{\Sigma(1)}$.

\end{proof}

By Lemmas~\ref{lemma_appendix_a_description_map_isomorphism} and \ref{lemma_appendix_computation_age_pairing}, we deduce that the map
\[
\mathbf{R^1 Hom}(\coneform(\beta_{\mathcal{C}}),\Z) \longrightarrow \Pic(X) \oplus \Q^{\mathcal{C}}
\]
factors through a map
\begin{equation}\label{equation_appendix_identification_extended_picard_group}
\mathbf{R^1 Hom}(\coneform(\beta_{\mathcal{C}}),\Z) \longrightarrow \Pic_{\mathcal{C}}(X) \, .    
\end{equation}

We can now complete the proof of the theorem by showing that this map is an isomorphism.
\begin{proof}
We begin by proving injectivity. Let $(\Tilde{a},b) \in \Z^{\Sigma(1) \cup \mathcal{C}} \oplus F_1$, and suppose that the image of $[\Tilde{a},b]$ in $\Pic_{\mathcal{C}}(X)$ is zero. By Lemma~\ref{lemma_appendix_a_description_map_isomorphism}, this implies that $[a,b] = 0$, hence there exists $c \in F_0^{\vee}$ such that
\[
(a,b) = (\widehat{\beta} \oplus \iota)^{\vee}(c).
\]
Moreover, for all $\stackS \in \mathcal{C}$, again by Lemma~\ref{lemma_appendix_a_description_map_isomorphism}, we have
\[
a_{\stackS} = \left\langle (a,b) , \psi_{\stackS} + x_{\stackS} \right\rangle 
= \left\langle c , \widehat{\beta}(\psi_{\stackS}) + \iota(x_{\stackS}) \right\rangle 
= \left\langle c , \widehat{\beta}_{\mathcal{C}}(e_{\stackS}) \right\rangle.
\]
It follows that
\[
(\Tilde{a},b) = (\widehat{\beta}_{\mathcal{C}} \oplus \iota)^{\vee}(c),
\]
and therefore $[\Tilde{a},b] = 0$.

We now prove surjectivity. Let $(L,\varphi) \in \Pic_{\mathcal{C}}(X)$. There exists $(a,b) \in \Z^{\Sigma(1)} \oplus F_1$ such that $[a,b] = L$. Then, by Lemma~\ref{lemma_appendix_computation_age_pairing}, for all $\stackS \in \mathcal{C}$, we have
\[
\varphi(\stackS) + \langle (a,b) , (\psi_{\stackS},x_{\stackS}) \rangle \in \Z.
\]
Choose $\Tilde{a} \in \Z^{\Sigma(1)\cup \mathcal{C}}$ such that its restriction to $\Z^{\Sigma(1)}$ is $a$ and
\[
\Tilde{a}(e_{\stackS}) = \varphi(\stackS) + \langle (a,b) , (\psi_{\stackS},x_{\stackS}) \rangle
\]
for all $\stackS \in \mathcal{C}$. Then the image of
\[
[\Tilde{a},b]
\]
in $\Pic_{\mathcal{C}}(X)$ is equal to $(L,\varphi)$.
\end{proof}


\bibliographystyle{amsplain}
\bibliography{bibliography}

@book{weibel1994homological,
  title     = {An Introduction to Homological Algebra},
  author    = {Weibel, Charles A.},
  year      = {1994},
  publisher = {Cambridge University Press},
  series    = {Cambridge Studies in Advanced Mathematics},
  volume    = {38},
  address   = {Cambridge},
  isbn      = {9780521559874}
}

@Inbook{Peyre_beyond_height,
author="Peyre, Emmanuel",
title="Chapter \rm V: Beyond Heights: Slopes and Distribution of Rational Points",
bookTitle="Arakelov Geometry and Diophantine Applications",
year="2021",
publisher="Springer International Publishing",
address="Cham",
pages="215--279",
isbn="978-3-030-57559-5",
doi="10.1007/978-3-030-57559-5_6",
url="https://doi.org/10.1007/978-3-030-57559-5_6"
}

@incollection{methode_du_cercle_peyre,
  author    = {Peyre, Emmanuel},
  title     = {Torseurs universels et m{\'e}thode du cercle},
  booktitle = {Rational Points on Algebraic Varieties},
  series    = {Progress in Mathematics},
  volume    = {199},
  pages     = {221--274},
  year      = {2001},
  publisher = {Birkh{\"a}user},
  address   = {Basel}
}

@article{peyre_duke,
author="Peyre, Emmanuel",
title = {{Hauteurs et mesures de Tamagawa sur les variétés de Fano}},
volume = {79},
journal = {Duke Mathematical Journal},
number = {1},
publisher = {Duke University Press},
pages = {101--218},
year = {1995},
doi = {10.1215/S0012-7094-95-07904-6},
URL = {https://doi.org/10.1215/S0012-7094-95-07904-6}
}

@incollection{peyre_zeta_height_function,
     author="Peyre, Emmanuel",
     title = {Terme principal de la fonction z\^eta des hauteurs et torseurs universels},
     booktitle = {Nombre et r\'epartition de points de hauteur born\'ee},
     editor = {Peyre Emmanuel},
     series = {Ast\'erisque},
     pages = {259--298},
     publisher = {Soci\'et\'e math\'ematique de France},
     number = {251},
     year = {1998},
     mrnumber = {1679842},
     zbl = {0966.14016},
     language = {fr},
     url = {http://www.numdam.org/item/AST_1998__251__259_0/}
}

@article{Batyrev1990,
author = {Batyrev, Victor V. and Manin, Yuri},
journal = {Mathematische Annalen},
language = {fre},
number = {1-3},
pages = {27-44},
title = {Sur le nombre des points rationnels de hauteur borné des variétés algébriques.},
url = {http://eudml.org/doc/164626},
volume = {286},
year = {1990},
}

@article{FrankeManinTschinkel1989,
  author  = {Franke, Jens and Manin, Yuri I. and Tschinkel, Yuri},
  title   = {Rational points of bounded height on Fano varieties},
  journal = {Inventiones Mathematicae},
  volume  = {95},
  year    = {1989},
  number  = {2},
  pages   = {421--435},
  doi     = {10.1007/BF01393904},
  mrnumber = {0978033},
  mrclass  = {11G35 (14G25)}
}

@incollection{ColliotTheleneSansuc1979,
  author    = {Colliot-Th{\'e}l{\`e}ne, Jean-Louis and Sansuc, Jean-Jacques},
  title     = {La descente sur les vari{\'e}t{\'e}s rationnelles},
  booktitle = {Journ{\'e}es de G{\'e}om{\'e}trie Alg{\'e}brique d'Angers},
  year      = {1980},
  pages     = {223--237},
  publisher = {Sijthoff \& Noordhoff},
  address   = {Alphen aan den Rijn},
  language  = {fr},
}

@article{ColliotTheleneSansuc1987,
  author  = {Colliot-Th{\'e}l{\`e}ne, Jean-Louis and Sansuc, Jean-Jacques},
  title   = {La descente sur les vari{\'e}t{\'e}s rationnelles. {II}},
  journal = {Duke Mathematical Journal},
  volume  = {54},
  number  = {2},
  year    = {1987},
  pages   = {375--492},
  language = {fr},
}

@article{ColliotTheleneSansuc1976,
  author  = {Colliot-Th{\'e}l{\`e}ne, Jean-Louis and Sansuc, Jean-Jacques},
  title   = {Torseurs sous des groupes de type multiplicatif : applications {\`a} l'{\'e}tude des points rationnels de certaines vari{\'e}t{\'e}s alg{\'e}briques},
  journal = {Comptes rendus de l'Académie des Sciences, Série A},
  volume  = {282},
  number  = {20},
  year    = {1976},
  pages   = {1113--1116},
  language = {fr},
}

@incollection{HarariSkorobogatov2013,
  author    = {Harari, David and Skorobogatov, Alexei N.},
  title     = {Descent theory for open varieties},
  booktitle = {Torsors, Étale Homotopy and Applications to Rational Points},
  series    = {London Mathematical Society Lecture Note Series},
  volume    = {405},
  pages     = {250--279},
  year      = {2013},
  publisher = {Cambridge University Press},
  editor    = {Skorobogatov, Alexei N.}
}

@book{Skorobogatov2001,
  author    = {Skorobogatov, Alexei},
  title     = {Torsors and Rational Points},
  series    = {Cambridge Tracts in Mathematics},
  volume    = {144},
  publisher = {Cambridge University Press},
  address   = {Cambridge},
  year      = {2001}
}

@article{Brochard2021,
  author  = {Brochard, Sylvain},
  title   = {Duality for commutative group stacks},
  journal = {International Mathematics Research Notices},
  year    = {2021},
  number  = {3},
  pages   = {2321--2388},
  doi     = {10.1093/imrn/rnz161}
}

@inbook{salberger_torsor,
    author = {SALBERGER, Per},
    title = {Tamagawa measures on universal torsors and points of bounded height on Fano varieties},
    publisher = {Société mathématique de France},
    book = {Nombre et répartition de points de hauteur bornée},
    year = 1998,
    pages = {91-258},
    series = {Astérique},
    volume = {251},
}

@article{davenport_geometry_numbers,
    author = {DAVENPORT, Harold},
    title = "{On a Principle of Lipschitz}",
    journal = {Journal of the London Mathematical Society},
    volume = {s1-26},
    number = {3},
    pages = {179-183},
    year = {1951},
    month = {07},
    issn = {0024-6107},
    doi = {10.1112/jlms/s1-26.3.179},
    url = {https://doi.org/10.1112/jlms/s1-26.3.179},
    eprint = {https://academic.oup.com/jlms/article-pdf/s1-26/3/179/2536641/s1-26-3-179.pdf},
}

@book{CoxLittleSchenck2011,
  author    = {Cox, David A. and Little, John B. and Schenck, Henry K.},
  title     = {Toric Varieties},
  series    = {Graduate Studies in Mathematics},
  volume    = {124},
  year      = {2011},
  publisher = {American Mathematical Society},
  address   = {Providence, RI},
  isbn      = {978-0-8218-4819-7},
  doi       = {10.1090/gsm/124},
}

@article {fantechi_toric_stack,
	title = {Smooth toric DM stacks},
	year = {2010},
	journal = {Journal für die reine und angewandte Mathematik},
    volume = {648},
    pages = {201-244},
	url = {http://hdl.handle.net/1963/2120},
	author = {Barbara Fantechi and Etienne Mann and Fabio Nironi}
}

@article{borisov_toric_stack,
title = "The orbifold chow ring of toric deligne-mumford stacks",
author = "Borisov, {Lev A.} and Linda Chen and Smith, {Gregory G.}",
year = "2005",
month = jan,
doi = "10.1090/S0894-0347-04-00471-0",
language = "English (US)",
volume = "18",
pages = "193--215",
journal = "Journal of the American Mathematical Society",
issn = "0894-0347",
publisher = "American Mathematical Society",
number = "1",

}

@article{Malle2002,
  author  = {Malle, Gunter},
  title   = {On the distribution of Galois groups},
  journal = {Journal of Number Theory},
  volume  = {92},
  year    = {2002},
  pages   = {315--329}
}

@article{Malle2004,
  author  = {Malle, Gunter},
  title   = {On the distribution of Galois groups II},
  journal = {Experimental Mathematics},
  volume  = {13},
  year    = {2004},
  number  = {2},
  pages   = {129--135}
}

@misc{gundlach2022mallesconjecturemultipleinvariants,
  author       = {Fabian Gundlach},
  title        = {Malle's conjecture with multiple invariants},
  year         = {2022},
  eprint       = {2211.16698},
  archivePrefix= {arXiv},
  primaryClass = {math.NT},
  note         = {arXiv:2211.16698}
}

@misc{shankar2024asymptoticscubicfieldsordered,
  author       = {Arul Shankar and Frank Thorne},
  title        = {On the asymptotics of cubic fields ordered by general invariants},
  year         = {2022},
  eprint       = {2207.06514},
  archivePrefix= {arXiv},
  primaryClass = {math.NT},
  note         = {arXiv:2207.06514},
  url          = {https://arxiv.org/abs/2207.06514}
}

@article{bhargava_counting_cubic_field,
    title = {On the Davenport–Heilbronn theorems and second order terms},
    author = {Manjul BHARGAVA, Arul SHANKAR, Jacob TSIMERMAN},
    journal = {Inventiones mathematicae},
    year = {2013},
    volume={193},
    pages = {439-499}
}

@article{bhargava_mass_formula,
    author = {Bhargava, Manjul},
    title = {Mass Formulae for Extensions of Local Fields, and Conjectures on the Density of Number Field Discriminants},
    journal = {International Mathematics Research Notices},
    volume = {2007},
    pages = {rnm052},
    year = {2007},
    month = {01},
    doi = {10.1093/imrn/rnm052},
    url = {https://doi.org/10.1093/imrn/rnm052},
    eprint = {https://academic.oup.com/imrn/article-pdf/doi/10.1093/imrn/rnm052/19150310/rnm052.pdf},
}

@phdthesis{dardathesis,
    author = {Ratko, Darda},
    title = {Rational points of bounded height on weighted projective stacks},
    school = {	École doctorale Sciences mathématiques de Paris centre},
    year = {2021}
}

@article{darda2024batyrevmanin,
  author  = {Darda, Ratko and Yasuda, Takehiko},
  title   = {The Batyrev–Manin conjecture for {DM} stacks},
  journal = {Journal of the European Mathematical Society},
  year    = {2024},
  doi     = {10.4171/JEMS/???}, 
  note    = {Publié en ligne (online first)},
  pages= {N/A},
}

@article{ellenberg2022heights,
  author  = {Ellenberg, Jordan S. and Satriano, Matthew and Zureick-Brown, David},
  title   = {Heights on stacks and a generalized Batyrev--Manin--Malle conjecture},
  journal = {Forum of Mathematics, Sigma},
  volume  = {11},
  year    = {2023},
  pages   = {e14},
  doi     = {10.1017/fms.2023.5},
}

@misc{loughran_santens_malle_conjecture,
  author       = {Daniel Loughran and Tim Santens},
  title        = {Malle's conjecture and {B}rauer groups of stacks},
  year         = {2024},
  eprint       = {2412.04196},
  archivePrefix= {arXiv},
  primaryClass = {math.NT},
  note         = {arXiv:2412.04196},
  url          = {https://arxiv.org/abs/2412.04196}
}

@misc{darda_yasuda_toric_stacks_batyrev,
  title        = {The Manin conjecture for toric stacks},
  author  = {Darda, Ratko and Yasuda, Takehiko},
  year         = {2023},
  eprint       = {2311.02012},
  archivePrefix= {arXiv},
  primaryClass = {math.NT},
  note         = {arXiv:2311.02012},
  url          = {https://arxiv.org/abs/2311.02012}
}

@misc{darda2025orbifoldpseudoeffectiveconestoric,
  author  = {Darda, Ratko and Yasuda, Takehiko},
  title        = {Orbifold pseudo-effective cones of toric stacks},
  year         = {2025},
  eprint       = {2508.20434},
  archivePrefix= {arXiv},
  primaryClass = {math.AG},
  note         = {arXiv:2508.20434},
  url          = {https://arxiv.org/abs/2508.20434}
}

@article{mann2006orbifoldquantumcohomologyweighted,
  author  = {Mann, Ethan},
  title   = {Orbifold quantum cohomology of weighted projective spaces},
  journal = {Duke Mathematical Journal},
  volume  = {131},
  number  = {1},
  year    = {2006},
  pages   = {1--52},
}

@article{Coates_Corti_Iritani_Tseng_miror_theorem_toric_stack, 
title={A mirror theorem for toric stacks}, 
volume={151}, 
DOI={10.1112/S0010437X15007356}, 
number={10}, 
journal={Compositio Mathematica}, 
author={Coates, Tom and Corti, Alessio and Iritani, Hiroshi and Tseng, Hsian-Hua}, 
year={2015}, 
pages={1878–1912}}

@article{Jiang2005TheOC,
  title={The orbifold cohomology ring of simplicial toric stack bundles},
  author={Yunfeng Jiang},
  journal={Illinois Journal of Mathematics},
  year={2005},
  volume={52},
  pages={493-514},
  url={https://api.semanticscholar.org/CorpusID:15373456}
}

@article{yasuda2004motivicintegrationdelignemumfordstacks,
  author  = {Yasuda, Takehiko},
  title   = {Motivic integration over {D}eligne--{M}umford stacks},
  journal = {Advances in Mathematics},
  volume  = {207},
  number  = {2},
  pages   = {707--761},
  year    = {2006},
  doi     = {10.1016/j.aim.2006.01.004}
}

@article{2023_Bresciani,
author = {Giulio Bresciani and Angelo Vistoli},
title = {An arithmetic valuative criterion for proper maps of tame algebraic stacks},
journal = {Manuscripta Mathematica},
year = {2023},
publisher = {Springer Nature},
month = {july},
pages = {1061-1071},
url = {https://doi.org/10.1007/s00229-023-01491-6},
doi = {10.1007/s00229-023-01491-6}
}

@article{abramovich2008gromovwittentheorydelignemumfordstacks,
  author  = {Abramovich, Dan and Graber, Tom and Vistoli, Angelo},
  title   = {Gromov--Witten theory of {D}eligne--{M}umford stacks},
  journal = {Inventiones Mathematicae},
  volume  = {173},
  number  = {3},
  pages   = {513--557},
  year    = {2008},
  doi     = {10.1007/s00222-008-0131-6}
}

@article{de_gaay_definition_real_topology,
    author = {De Gaay Fortman, Olivier},
    title = {Real Moduli Spaces And Density Of Non-simple Real Abelian Varieties},
    journal = {The Quarterly Journal of Mathematics},
    volume = {73},
    number = {3},
    pages = {969-989},
    year = {2022},
    month = {03},
    issn = {0033-5606},
    doi = {10.1093/qmath/haab060},
    url = {https://doi.org/10.1093/qmath/haab060},
    eprint = {https://academic.oup.com/qjmath/article-pdf/73/3/969/45800048/haab060.pdf},
}

@misc{ambrosi2025topologyrealalgebraicstacks,
  title        = {On the topology of real algebraic stacks},
  author       = {Emiliano Ambrosi and Olivier de Gaay Fortman},
  year         = {2025},
  eprint       = {2504.02720},
  archivePrefix= {arXiv},
  primaryClass = {math.AG},
  note         = {arXiv:2504.02720},
  url          = {https://arxiv.org/abs/2504.02720}
}

@article{Rydh_2015,
	doi = {10.1016/j.jalgebra.2014.09.012},
  
	url = {https://doi.org/10.1016%2Fj.jalgebra.2014.09.012},
  
	year = 2015,
	month = {jan},
  
	publisher = {Elsevier {BV}},
  
	volume = {422},
  
	pages = {105--147},
  
	author = {David Rydh},
  
	title = {Noetherian approximation of algebraic spaces and stacks},
  
	journal = {Journal of Algebra}
}

@misc{Brochard2008,
  author       = {Brochard, Sylvain},
  title        = {Champs alg\'ebriques et foncteur de Picard},
  year         = {2007},
  eprint       = {0808.3253},
  archivePrefix= {arXiv},
  primaryClass = {math.AG},
  note         = {Th\`ese de doctorat, Universit\'e Rennes 1},
  url          = {https://arxiv.org/abs/0808.3253}
}

@article{pieropan_hyperbola_method,
     author = {Marta Pieropan and Damaris Schindler},
     title = {Hyperbola method on toric varieties},
     journal = {Journal de l{\textquoteright}\'Ecole polytechnique {\textemdash} Math\'ematiques},
     pages = {107--157},
     publisher = {\'Ecole polytechnique},
     volume = {11},
     year = {2024},
     doi = {10.5802/jep.251},
     mrnumber = {4683391},
     zbl = {07811890},
     language = {en},
     url = {https://jep.centre-mersenne.org/articles/10.5802/jep.251/}
}

@book{giraud1971cohomologie,
  title     = {Cohomologie non ab{\'e}lienne},
  author    = {Giraud, Jean},
  publisher = {Springer-Verlag},
  year      = {1971},
  series    = {Grundlehren der mathematischen Wissenschaften},
  volume    = {179},
  address   = {Berlin, Heidelberg},
  isbn      = {978-3-540-05307-1, 0387053077}
}

@book{bourbaki2007integration,
  author    = {Bourbaki, Nicolas},
  title     = {Int{\'e}gration. Chapitres~7 {\`a}~8},
  series    = {\'El\'ements de math\'ematique},
  publisher = {Springer},
  address   = {Berlin},
  year      = {2007}
}

@misc{bongiorno2024multiheightanalysisrationalpoints,
  title={Multi-height analysis of rational points of toric varieties},
  author={Nicolas Bongiorno},
  year={2024},
  note={arXiv:2412.04226},
  url={https://arxiv.org/abs/2412.04226}
}

@misc{bongiorno_hyperbola_method,
      title={The multi-height distribution implies the Batyrev-Manin principle}, 
      author={Nicolas Bongiorno},
      year={2026},
      note={arXiv:2603.13938},
      eprint={2603.13938},
      archivePrefix={arXiv},
      primaryClass={math.NT},
      url={https://arxiv.org/abs/2603.13938}, 
}

@misc{santens2025maninsconjectureintegralpoints,
  title        = {Manin's conjecture for integral points on toric varieties},
  author       = {Tim Santens},
  year         = {2023},
  eprint       = {2312.13914},
  archivePrefix= {arXiv},
  primaryClass = {math.NT},
  note         = {arXiv:2312.13914},
  url          = {https://arxiv.org/abs/2312.13914}
}

@misc{santens2023brauermaninobstructionstackycurves,
  author       = {Tim Santens},
  title        = {The Brauer--Manin obstruction for stacky curves},
  year         = {2022},
  eprint       = {2210.17184},
  archivePrefix= {arXiv},
  primaryClass = {math.AG},
  note         = {arXiv:2210.17184},
  url          = {https://arxiv.org/abs/2210.17184}
}

@book{ArzhantsevDerenthalHausenLaface2015,
  author    = {Ivan Arzhantsev and Ulrich Derenthal and
               Jürgen Hausen and Antonio Laface},
  title     = {Cox Rings},
  series    = {Cambridge Studies in Advanced Mathematics},
  volume    = {144},
  publisher = {Cambridge University Press},
  year      = {2015},
  isbn      = {9781107024625},
  doi       = {10.1017/CBO9781139175852}
}

@misc{darda2026stackybatyrevmaninconjecturemodular,
      title={The stacky Batyrev-Manin conjecture and modular curves}, 
      author={Ratko Darda and Changho Han},
      year={2026},
      eprint={2602.19771},
      archivePrefix={arXiv},
      primaryClass={math.NT},
      note = {arXiv:2210.17184},
      url={https://arxiv.org/abs/2602.19771}, 
}

\end{document}